\documentclass[a4paper,reqno]{amsart}
\usepackage{amsfonts,dsfont}
\usepackage{amsmath}
\usepackage{amssymb}
\usepackage{fullpage}
\usepackage{color}

\usepackage[hidelinks]{hyperref}
\usepackage{cite}

\newtheorem{theorem}{Theorem}
\newtheorem{corollary}[theorem]{Corollary}
\newtheorem{lemma}[theorem]{Lemma}
\newtheorem{proposition}[theorem]{Proposition}

\theoremstyle{definition}

\newtheorem{definition}[theorem]{Definition}
\newtheorem{remark}[theorem]{Remark}

\numberwithin{equation}{section}
\numberwithin{theorem}{section}

\DeclareMathOperator{\I}{Id}

\DeclareMathOperator{\V}{var}

\newcommand{\FL}{F\langle X|L\rangle}

\newcommand{\B}{\mathcal{B}}

\DeclareMathOperator{\spn}{span}
\DeclareMathOperator{\e}{\varepsilon}

\DeclareMathOperator{\ad}{ad}
\DeclareMathOperator{\D}{Der}

\DeclareMathOperator{\id}{id}
\DeclareMathOperator{\E}{End}

\usepackage{xcolor}
\definecolor{ipb}{RGB}{211, 159, 194}

\newcommand{\black}{\color{black}}

\begin{document}

\title{Differential varieties of upper triangular matrices}

\author[D.~La~Mattina]{Daniela La Mattina}

\address{Daniela La Mattina, Dipartimento di Matematica e Informatica, Universit\'{a} degli Studi di Palermo, via
Archirafi 34, 90123, Palermo, Italy.}
\email{daniela.lamattina@unipa.it}


\author[C.~Rizzo]{Carla Rizzo}
\address{Carla Rizzo, CMUC, Department of Mathematics, University of Coimbra, 3000-143
Coimbra, Portugal.}
\email{carlarizzo@mat.uc.pt}

 \thanks{Carla Rizzo was financially supported by the Fundação para a Ciência e a Tecnologia (Portuguese Foundation for Science and Technology) under the scope of the projects UID/00324/2025 (https://doi.org/10.54499/UID/00324/2025) (Centre for Mathematics of the University of Coimbra).}



\begin{abstract} 
Let $L$ be a Lie algebra acting by derivations on an associative algebra $A$ over a field $F$ of characteristic zero. The polynomial identities satisfied by $A$ with respect to this action are called differential identities, or $L$-identities. In this paper, we study the differential identities of the algebra $UT_k$ of $k\times k$ upper triangular matrices and take a first step toward the classification of minimal $L$-varieties of differential exponent $3$.

We first prove that, whenever $UT_k$ generates a minimal variety of algebras with derivations, the $L$-action can be replaced by its semisimple part. More precisely, it is enough to consider inner derivations induced by diagonal elements. We then apply this reduction to $UT_3$ and explicitly determine the $T_L$-ideal of differential identities and the corresponding differential codimension sequence for every such action on $UT_3$. Finally, we show that every $L$-variety generated by $UT_k$, with $k\geq 3$, contains $UT_3$ endowed with one of these $L$-actions.
\end{abstract}

\maketitle

\section{Introduction}
The theory of polynomial identities studies associative algebras through the identical
relations they satisfy. Among the numerical invariants introduced to measure such relations
quantitatively, the codimension sequence $c_n(A)$ of an algebra $A$, due to Regev  \cite{Regev1972}, plays a
central role. When $A$ satisfies a nontrivial polynomial identity, $c_n(A)$ is exponentially
bounded, and Giambruno and Zaicev proved that the limit
$\exp(A):=\lim_{n\to\infty}\sqrt[n]{c_n(A)}$ always exists and is a nonnegative integer, the
\emph{PI-exponent} of $A$ (see \cite[Chapter 6]{GiambrunoZaicevbook}), settling a conjecture of
Amitsur. This integer provides a scale on which to compare varieties of algebras, and a natural
problem is to classify, among the varieties of a given exponent, the \emph{minimal} ones.
Giambruno and Zaicev showed that, for $d\geq2$, minimal varieties of exponent $d$ are generated
by  upper block-triangular matrix algebras \cite{GiambrunoZaicev2003}; in
particular, among these, the algebra $UT_k$ of $k\times k$ upper triangular matrices, with
trivial (one-dimensional) diagonal blocks, realizes a minimal variety of exponent $k$. The case
$k=2$ is especially distinguished: by a classical theorem of Kemer \cite{Kemer1979}, $UT_2$, together with the
infinite dimensional Grassmann algebra $G$, generates the only variety of \emph{almost
polynomial growth}, i.e., of exponential growth all of whose proper subvarieties have
polynomial growth.

Differential identities generalize ordinary polynomial identities to the setting of algebras equipped with an action of a Lie algebra by derivations. Given a Lie algebra $L$ acting on $A$ by derivations, one obtains, via the
universal enveloping algebra $U(L)$, the notion of \emph{$L$-identity} (or \emph{differential identity}) of
$A$, first introduced by Kharchenko in \cite{Kharchenko1978}. As in the ordinary case, the differential codimension
sequence $c_n^L(A)$ of a finite dimensional $L$-algebra $A$ grows exponentially, and Gordienko in
\cite{Gordienko2013} proved that $\exp^L(A):=\lim_{n\to\infty}\sqrt[n]{c_n^L(A)}$ exists and is
a nonnegative integer, the \emph{$L$-exponent} of $A$. Since every ordinary polynomial identity of $A$ is, in particular, an $L$-identity, we have $c_n(A)\leq c_n^L(A)$ for all $n$, and hence
one always has $\exp(A)\leq\exp^L(A)$. Gordienko and Kochetov conjectured in \cite{GordienkoKochetov2014} that equality always holds and proved this conjecture when $L$ is a finite-dimensional semisimple Lie algebra. More recently, Rizzo in \cite{Rizzo2023} established the equality for an arbitrary Lie algebra $L$, namely,
$$
\exp^L(A)=\exp(A).
$$
Thus, although the presence of a Lie algebra action may substantially enlarge the space of polynomial identities, it does not affect the exponential rate of growth of the corresponding codimension sequence.

This equality suggests that the classical theory of minimal varieties, and of varieties of
almost polynomial growth, should admit a differential counterpart closely related to the ordinary one, generated this time by
algebras with derivations. A first piece of evidence in this direction comes from the case $k=2$: Giambruno
and Rizzo \cite{GiambrunoRizzo2019} showed that $UT_2$, endowed with the inner derivation
induced by $e_{22}$, still generates an $L$-variety of almost polynomial growth, and Rizzo
\cite{Rizzo2021} showed that, up to $T_L$-equivalence, $UT_2$ with these two actions (the trivial action and the action induced by this inner derivation) are  the only two families of  $2\times2$ upper triangular matrix $L$-algebras that generate varieties of algebras with derivations with this property. Moreover, in \cite{Rizzo2023}, it was shown that, for solvable $L$, $UT_2$ with either one of these two $L$-actions by derivations are the only two families of finite-dimensional algebras generating $L$-varieties with almost polynomial growth.  See also \cite{MartinoRizzo2022} for a complete classification of the $L$-subvarieties generated by $UT_2$ with these two actions.

Together, these results provide a complete differential analogue of Kemer's theorem for exponent $2$ in the solvable case. The picture changes markedly when $L$ is semisimple.  In \cite{BroxRizzo2024}, Brox and Rizzo showed that, for every $k\geq 2$, the algebra $M_k(F)$ of $k\times k$ matrices, endowed with the natural action of its Lie algebra of derivations $\D(M_k(F))$, which is a simple Lie algebra, generates an $L$-variety of almost polynomial growth. Moreover, Argenti \cite{Argenti2024} completely classified the $L$-varieties of almost polynomial growth when $L$ is the special linear Lie algebra of order $2$. These results highlight the significant influence of the structural properties of the acting Lie algebra on the growth of the varieties it determines.

It is then natural to ask what happens one step further, namely, for algebras of exponent $3$: since
$UT_3$ realizes a minimal variety of exponent $3$
in the ordinary setting, and since $\exp^L(UT_3^\theta)=\exp(UT_3)=3$ for \emph{every}
$L$-action $\theta$ by Rizzo's theorem, one expects $UT_3$, suitably endowed with a nontrivial
derivation action, to still generate minimal $L$-varieties of exponent $3$. We refer the reader to \cite{DiVincenzoNardozza2021,Nardozza2023} for partial results concerning the differential identities of $UT_k$ for $k\geq 3$.

This paper takes a first step toward the classification of minimal $L$-varieties. Our key structural result places a strong restriction, \emph{a priori}, on the actions that can give rise to a minimal $L$-variety. Namely, we prove that, for every minimal $L$-variety of exponent $k$ generated by $UT_k$, the $L$-action on $UT_k$ may be assumed, without loss of generality, to be semisimple; equivalently, it may be generated by derivations induced by diagonal elements of $UT_k$.
In the case $k=3$, this reduction allows us to restrict our analysis to semisimple $L$-actions on $UT_3$. Thus, these actions are not merely more tractable from a computational point of view: up to replacing the action by its semisimple part, they are precisely the actions that can generate a minimal $L$-variety of exponent $3$.

We determine explicitly, for every one-dimensional Lie subalgebra and for the unique two-dimensional Lie subalgebra of diagonal derivations of $UT_3$, the $T_L$-ideal of differential identities and the corresponding differential codimension sequence. Moreover, we prove that, for arbitrary $k$ and every $L$-action, the $L$-variety generated by $UT_k$ contains the $L$-algebra $UT_3$ endowed with one of these actions. Thus, our computations identify the building blocks that occur in $L$-varieties generated by upper triangular matrices. Combined with the structural reduction above, these results provide the candidates for minimal $L$-varieties of exponent $3$.

\section{Differential identities and $L$-algebras}
Throughout this paper, $F$ will denote a field of characteristic zero, $A$ an associative $F$-algebra, and $(L,[\cdot\, ,\cdot]_L)$ a Lie algebra over $F$ that we will denote simply as $L$ when confusion does not arise.

Recall that $A^-$ denotes the Lie algebra obtained from $A$ by endowing its underlying vector space with the \emph{commutator bracket} $[a,b]:=ab-ba$ for all $a,b\in A$.

Let $\E_F(A)$ denote the associative algebra of $F$-linear endomorphisms of $A$ acting on the right of $A$. Accordingly, composition is taken from left to right, and we adopt exponential notation for their action.

We say that an $F$-linear endomorphism $\delta:A \to A$ is a \textit{derivation} on $A$ if it satisfies the \textit{Leibniz rule}: for all $a,b\in A,$
$$(a b)^\delta=a^\delta b + ab^\delta.$$ 
If $a\in A$, then the $F$-linear map $\ad_a:A\to A$ defined by $ b^{\ad_a}=[b,a]=ba-ab$ for all $b\in A$ is a derivation on $A$ called \textit{inner derivation} induced by the element $a$.

The vector space $\D(A)$ of all derivations of $A$, equipped with the commutator bracket, forms a Lie subalgebra of $\E_F(A)^-$. Moreover, the subspace $\ad (A)$ of inner derivations of $A$ is a Lie ideal of $\D(A)$.

\begin{definition}
    We say that $A$ is an \emph{$L$-algebra} (or an \emph{algebra with derivations}), or equivalently that $L$ \emph{acts on $A$ by derivations},  if there exists a Lie algebra homomorphism
    $$
    \theta\colon L\to\D(A).
    $$
    We write $ \bar L=\theta(L)\subseteq\D(A)$ for the image of $L$ under $\theta$.
\end{definition}

An algebra $A$ may admit several $L$-algebra structures corresponding to distinct Lie algebra homomorphisms $\theta\colon L\to\D(A).$  When it is necessary to specify the homomorphism defining the action, we write $A^\theta$ for the algebra $A$ endowed with the $L$-action determined by $\theta$.

Since $\bar L=\theta(L)$ is a Lie subalgebra of $\D(A)$, the inclusion $\bar L\hookrightarrow\D(A)$
endows $A$ with the structure of a $\bar L$-algebra.
Thus, the action of $L$ on $A$ factors through $\bar L$. Moreover, distinct homomorphisms $ \theta\colon L\to\D(A)$ with the same image $\bar L$ induce the same $\bar L$-algebra structure on $A$. For this reason, throughout the paper we describe an $L$-action on $A$ by means of its image $\bar L$, and say that $L$ acts on $A$ \emph{via} $\bar L$. 

Whenever it is necessary to specify the image Lie algebra of derivations, we denote by
$A^{\bar L}$ the algebra $A$ endowed with an $L$-action whose image is $\bar L.$ 
Remark that this notation records only the Lie subalgebra $\bar L\subseteq\D(A)$ and suppresses the homomorphism
$\theta\colon L\to\D(A)$ defining the action.
Consequently, different homomorphisms with the same image give rise to the same notation
$A^{\bar L}$, although they need not define the same $L$-algebra structure. A homomorphism $\theta$ is therefore always implicitly part of the data, even when the notation
$A^{\bar L}$ does not display it; we will write $\theta$ explicitly only when
needed to avoid ambiguity.

Following~\cite[Definition~4.1]{AGV}, we make the following definition.

\begin{definition}\label{def: equivalence}
We say that $L$-algebra structures $A^\theta$ and
$A^{\theta'}$ are \emph{equivalent} if $\theta(L)=\theta'(L)$ as subalgebras of $\D(A)$.
\end{definition}
 This defines an equivalence relation on the set of $L$-algebra structures on $A$, with one class for each Lie subalgebra of $\D(A)$ arising as an
image. In particular, $A^{\bar L}$ always denotes a fixed representative of a single
equivalence class.

\smallskip

Since associative algebras are often easier to study than Lie algebras, it is convenient to translate Lie algebra action into an associative setting via the universal enveloping algebra $U(L)$ of $L$.
Recall that $U(L)$ is an unital associative algebra into which $ L$ embeds naturally as a Lie subalgebra of $U(L)^-$. Moreover, $U(L)$ is infinite dimensional whenever $L\neq 0$, even if $L$ itself is finite-dimensional.
For further details on the structure of $U(L)$ see \cite[Section 17]{Humphreys1972}.

By the universal property of $U(L),$ an $L$-action on $A$ can be uniquely extended to a right $U(L)$-action (which by abuse of notation we also call an $L$-action), by extending $\theta$ to the homomorphism of unital associative algebras $\Theta\colon U(L)\to\E_F(A)$ with $\Theta(1_{U(L)})=\id_A$.  In this way, $A$ becomes a right $U(L)$-module algebra, where the action is written in exponential notation. 

Note that when $\theta(L)=0$, then $\Theta(U(L))\cong F$, and the $U(L)$-action is just the linear action of $F$.

\smallskip

For fixed $L$, the class of $L$-algebras is a nontrivial variety as it contains $U(L)$. Ideals of $L$-algebras (\textit{$L$-ideals}) are understood to be invariant under the $U(L)$-action, and homomorphisms $\psi:A\rightarrow B$ between $L$-algebras $A,B$ (\textit{$L$-homomorphisms}) must satisfy $\psi(a^u)=\psi(a)^u$ for all $a\in A$ and  $u\in U(L)$.

The variety of $L$-algebras contains the \textit{free (unital associative) $L$-algebra} $\FL$, freely generated by a countably infinite set of variables $X:=\{x_1,x_2, \dots\},$ i.e., $\FL$ is uniquely determined up to an isomorphism by the following universal property: given a $L$-algebra $A$, any map $\varphi: X\rightarrow A$ can be uniquely extended to a homomorphism of $L$-algebras $\Bar{\varphi}:\FL\rightarrow A$, which we call the \textit{evaluation} of $\FL$ at elements $\varphi(x_1),\varphi(x_2),\ldots$ from $A$. Notice that the evaluation $\Bar{\varphi}$ of $\FL$ in $A$ depends not only on the map $\varphi$ but also on the right action of $L$ on $A.$ Hence, $\Bar{\varphi}$ depends on the Lie algebra homomorphism $\theta:L \to \D(A)$ that defines the $L$-action on $A.$

Since $A$ can be a $L$-algebra with respect to different $L$-actions, we will denote the evaluation $\Bar{\varphi}$ by $\Bar{\varphi}_{\theta}$, or simply $\varphi_{\theta}$,  when necessary to clarify which action we are considering.

We can construct $\FL$ as follows: given the free vector space $V_X$ spanned by $X,$ we consider the tensor product $V_X\otimes U(L),$ then the free $L$-algebra is the tensor algebra generated by $V_X\otimes U(L),$ i.e,
$$
\FL=\bigoplus_{n=0}^\infty \big( V_X\otimes U(L) \big)^{\otimes n},
$$
where $L$ acts on $V_X\otimes U(L)$ by right multiplication on the right component $U(L)$ and it is extended to the whole tensor algebra according to the Leibniz rule.

From now on, we write  $x^u:=x\otimes u$ and $x:= x \otimes 1_{U(L)}$ for $x\in X$ and $u, 1_{U(L)}\in U(L).$ Then, given a basis $\mathcal{B}_{U(L)} :=\{u_i \mid i\geq 1 \}$ of $U(L)$, the set
$$
\mathcal{B}_{\FL}:=\left\{ 1,  x_{i_1}^{u_{j_1}}\cdots x_{i_n}^{u_{j_n}} \mid n\geq 1, \, i_1,\ldots,i_n\geq 1, \; u_{j_1},\ldots,u_{j_n}\in\mathcal{B}_{U(L)}\right\}
$$
is a basis of $\FL$ and the $L$-action is determined by
$$
(x_{i_{1}}^{u_{j_{1}}} x_{i_{2}}^{u_{j_{2}}}\dots x_{i_{n}}^{u_{j_{n}}})^d =x_{i_{1}}^{ u_{j_{1}}d} x_{i_{2}}^{u_{j_{2}}}\dots x_{i_{n}}^{u_{j_{n}}}+ x_{i_{1}}^{ u_{j_{1}}} x_{i_{2}}^{u_{j_{2}}d}\dots x_{i_{n}}^{u_{j_{n}}}+ \dots+x_{i_{1}}^{u_{j_{1}}} x_{i_{2}}^{u_{j_{2}}}\dots x_{i_{n}}^{u_{j_{n}}d}
$$
for $u_{j_1},\ldots,u_{j_n}\in\mathcal{B}_{U(L)}$ and $d\in L.$
The elements of the free $L$-algebra are called \emph{differential polynomials} or \emph{$L$-polynomials}. 

A \emph{$T_L$-ideal} of the free $L$-algebra is an $L$-ideal which in addition is invariant under all $L$-endomorphisms of $\FL$ or \emph{substitutions}, which send variables of $x_i\in X$ to elements of $\FL.$ Given a set $S\subseteq\FL$, by $\langle S\rangle_{T_L}$ we denote the smallest $T_L$-ideal containing $S$; also, we say that an $L$-polynomial $f\in\FL$ is a \emph{consequence} of the $L$-polynomials in $S$, or \emph{follows} from the $L$-polynomials in $S$, if $f\in \langle S \rangle_{T_L}$.

\smallskip 

Given an $L$-algebra  $A^{\theta},$  a differential polynomial $f(x_1, \dots, x_n)\in \FL$ is an \emph{$L$-identity} of $A^{\theta}$, or a \emph{differential identity}, when the role of $L$ is clear, and we write $f\equiv 0$, if $f(a_{1},\dots,a_{n})=0$ for any $a_1,\ldots,a_n\in A,$  where the evaluation replaces each exponent $u\in U(L)$ occurring in $f$ by $\Theta(u)\in\E_F(A)$. 
We denote by $\I^L(A)$ the set of $L$-identities of $A$, which is a $T_L$-ideal of the free $L$-algebra $\FL$. Note that $\I^L(A)$ is the intersection of all kernels of evaluations of $\FL$ in $A$.

Remark that when $\theta(L)=0$, then $\Theta(U(L))\cong F$, and we are dealing with the ordinary polynomial identities.

\smallskip

For $n\geq 1$, we denote by $P_n^L$ the vector space of \emph{multilinear $L$-polynomials} of degree $n$ in the variables $x_{1},\dots,x_{n}$, that is
	$$
	P_n^L:=\spn_F\{x_{\sigma(1)}^{u_{i_1}}\cdots x_{\sigma(n)}^{u_{i_n}} \mid  \sigma\in S_{n} ,\; u_{i_1},\ldots,u_{i_n}\in \mathcal{B}_{U(L)} \},
	$$
	where $S_n$ denotes the symmetric group acting on $\{1,\ldots,n\}$.
	Since $F$ has characteristic zero, a standard argument shows that the $T_L$-ideal $\I^L(A)$ is completely determined by its multilinear  $L$-polynomials. 
	Thus it is reasonable to consider the quotient space
	$$
	P_n^L(A):= \dfrac{P_n^L}{P_n^L \cap \I^L(A)},
	$$
    and so one can define the $n$th \emph{$L$-codimension}, or \emph{differential codimension}, of $A$ as
    $$
    c_n^L(A):=\dim_F P_n^L(A), \quad n\geq 1.
    $$
 Note that $c_n^L(A)$ is not necessarily finite. However, if the action of $U(L)$ is finite dimensional, i.e., if $\Theta(U(L))$ is a finite-dimensional algebra, then $c_n^L(A)$ is finite for all $n \geq 1.$
 
 In \cite{Gordienko2013} Gordienko captured the exponential growth of the differential codimension sequence of a finite dimensional algebra  $A.$ More precisely, he proved that if $A$ is a finite dimensional $L$-algebra, then the limit 
$$
\exp^L(A):=\lim_{n\to \infty}\sqrt[n]{c_n^L(A)}
$$ 
exists and is a nonnegative integer called the \emph{$L$-exponent of $A$}.
In case $L=0,$ we are dealing with ordinary polynomial identities, and the $L$-exponent of $A$ is the (ordinary) exponent of $A,$ denoted by $\exp(A)$ (see \cite[Chapter 6]{GiambrunoZaicevbook}). 
In \cite{Rizzo2023} Rizzo proved that
\begin{equation}\label{exponent}
    \exp^L(A)=\exp(A)
\end{equation}
in case $A$ is a finite dimensional $L$-algebra.

\smallskip

A variety of $L$-algebras generated by an $L$-algebra $A,$ denoted by $\V^L(A),$  is called {\em $L$-variety}, or {\em differential variety} and $\I^L(\mathcal{V}):=\I^L(A)$. The {\em growth} of $\mathcal{V}= \V^L(A)$ is the growth of the sequence $c^L_{n}(\mathcal{V}):=c^L_{n}(A)$, $n\geq 1,$ and the $L$-exponent of $\mathcal{V}$ is $\exp^L(\mathcal{V}):=\exp^L(A).$ A $L$-variety $ \mathcal{V}$ is said \emph{minimal} of exponent $d\geq 2$ if $\exp^L(\mathcal{V})=d$ and $\exp^L(\mathcal{U})<d$ for all proper $L$-subvariety $\mathcal{U}$ of $\mathcal{V}$.

\section{$L$-varieties generated by upper triangular matrices}

For $k\geq 2$, let $UT_k$ be the algebra of $k \times k$ upper triangular matrices over $F$, and $e_{ij}$ be the $(i,j)$-unit matrix whose only nonzero entry is $1_F$ in position $(i,j).$ In \cite{CoelhoPolcinoMilies1993}, the authors studied the derivations of $UT_k$ and proved the following. 
\begin{theorem} \label{Thm Der(UT_n)}
	Any derivation of $UT_k$ is inner.
\end{theorem}
As a direct consequence of the above theorem we have the following.
\begin{corollary}
    $\dim_F \D(UT_k)=\frac{k^2 +k -2}{2}$.
\end{corollary}

For $a \in UT_k$, write $a = a^{(s)} + a^{(n)}$, where $a^{(s)}$ denotes the diagonal part of $a$
and $a^{(n)} := a - a^{(s)}$ is strictly upper triangular.
\black
For any $\delta\in \D(UT_k)$, by Theorem~\ref{Thm Der(UT_n)},  there exists $a\in UT_k,$ such that $\delta=\ad_a.$ We set
\begin{equation}\label{eq: delta as delta^s and delta^n}
    \delta=\delta^{(s)}+\delta^{(n)},
\end{equation}
where $\delta^{(s)}=\ad_{a^{(s)}}$  
and $\delta^{(n)}=\ad_{a^{(n)}}.$ 

Notice that, for $1\leq i\leq j \leq k$,
\begin{equation}\label{eq: action of delta^s and delta^n on e_ij}
    e_{ij}^{\delta^{(s)}}=\begin{cases}
    \gamma_{i,j} e_{ij} & \mbox{ if } i\neq j\\
    0 & \mbox{ if } i= j
\end{cases} \quad \quad \text{and} \quad \quad e_{ij}^{\delta^{(n)}}\in J^{j-i+1}
\end{equation}
where $ \gamma_{i,j}
\in F$.

\smallskip

\begin{lemma}\label{diagonal-hom} Let $k \geq 2$. For $a = a^{(s)} + a^{(n)} \in UT_k$, define $$s: \D(UT_k) \longrightarrow \D(UT_k), \qquad s(\ad_a)=\ad_{a^{(s)}}.$$ Then $s$ is a homomorphism of Lie algebras.
\end{lemma}
\begin{proof} By Theorem ~\ref{Thm Der(UT_n)} every derivation of $UT_k$ is inner, so $s$ is defined on all of $\D(UT_k).$ We first check that $s$ is well defined.
If $\ad_a=\ad_b$ for some $a, b \in UT_k$,
then $a - b$ lies in the center of $UT_k$, which consists of scalar matrices; in particular $a-b$ is
diagonal, so $a^{(s)} - b^{(s)} = a - b$ and $\ad_{a^{(s)}}=\ad_{b^{(s)}}.$ Linearity of $s$ follows immediately from the linearity of the map $a \mapsto a^{(s)}$.
It remains to check that $s$ preserves the Lie bracket.  Notice that for any $a,b$ the commutator $[a,b] = ab - ba$ is strictly upper triangular, i.e., $[a,b]^{(s)}=0.$
Hence:
$$
s([\ad_a, \ad_b])=s(\ad_{[a,b]})=\ad_{[a,b]^{(s)}}=\ad_0=0.
$$
On the other hand, since $a^{(s)}$ and $b^{(s)}$ are diagonal matrices, they commute, so
$$
[s(\ad_a), s(\ad_b)]=[\ad_{a^{(s)}}, \ad_{b^{(s)}}]=\ad_{[a^{(s)}, b^{(s)}]}=\ad_0=0.
$$
Therefore
$
s([\ad_a, ad_b])=[s(\ad_a), s(\ad_b)]
$ 
for all $a, b \in UT_k,$
proving that $s$ is a Lie algebra homomorphism.
\end{proof}

Let us denote by $UT_k^{\delta_1, \ldots, \delta_r}$ the algebra $UT_k$ in which $L$ acts by derivations through the Lie subalgebra of $\D(UT_k)$ spanned by $\{\delta_1, \ldots, \delta_r\}$, for $1\leq r \leq \frac{k^2 + k -2}{2}$.

We stress that $UT_k^{\delta_1,\ldots,\delta_r}$ always denotes a fixed representative of a single equivalence class in the sense of Definition~\eqref{def: equivalence}. More precisely, there exists a Lie algebra homomorphism $\theta:L\to\D(UT_k)$ with image
$\spn\{\delta_1,\ldots,\delta_r\}$, such that the $L$-algebra $UT_k^{\delta_1,\ldots,\delta_r}$ is precisely
$UT_k^\theta$. Thus, although the homomorphism $\theta$ is suppressed in the notation, it is always understood to be part of the underlying data.

\begin{theorem}\label{Thm: var UT_k} 
Consider the $L$-algebra $UT_k^{\delta_1,\ldots,\delta_r}$ for some $1 \leq r \leq \frac{k^2+k-2}{2}$. Then there exists a Lie algebra homomorphism $\theta^{(s)}:L \to \D(UT_k)$ such that $\theta^{(s)}(L)$ is spanned by
$\{\delta_1^{(s)},\ldots,\delta_r^{(s)}\}$ and 
$$UT_k^{\delta_1^{(s)},\ldots,\delta_r^{(s)}} \in \V^L\big(UT_k^{\delta_1,\ldots,\delta_r}\big),$$
where $UT_k^{\delta_1^{(s)},\ldots,\delta_r^{(s)}}=UT_k^{\theta^{(s)}}.$
\end{theorem}
\begin{proof}
Since $L$ acts on $UT_k^{\delta_1, \ldots, \delta_r}$ by derivation, there exists a Lie algebra homomorphism
$\theta: L \to \D(UT_k)$
such that $\theta(L)$ is the Lie subalgebra of $\D(UT_k)$ spanned by $\{\delta_1, \ldots, \delta_r\}.$ 
Let $s : \D(UT_k) \to \D(UT_k)$ be the Lie algebra homomorphism of
Lemma~\ref{diagonal-hom}, and set $\theta^{(s)}:=s\circ \theta$. Since $s$ and $\theta$ are
both Lie algebra homomorphisms, so is $\theta^{(s)}$, and $\theta^{(s)}(L) = s(\theta(L))$ is the
Lie subalgebra of $\D(UT_k)$ spanned by  $\{\delta_1^{(s)}, \ldots, \delta_r^{(s)}\},$ since $\delta_i^{(s)}=s(\delta_i)$ for $i=1, \ldots, r.$

In particular, if $d_1, \ldots, d_r$ are
elements of $L$ with $\theta(d_i) = \delta_i$ for $1 \le i \le r$, then
$\theta^{(s)}(d_i) = s(\theta(d_i)) = \delta_i^{(s)}$ for $1 \le i \le r$, and $\theta^{(s)}(d) = 0$
whenever $\theta(d) = 0$.

    Now, let $f\in \I^L \big(UT_k^{\delta_1, \ldots, \delta_r}\big)$  with $\deg f =n$ and assume, as we may, that $f$ is multilinear. Decompose $f$ as $f=f^{(d)}+f^{(1)}$, where $f^{(d)}\in P_n^L$ is a differential polynomial in which in each monomial appears at least a variable with exponent different from $1_{U(L)}$, and $f^{(1)}\in P_n$ is a multilinear ordinary polynomial, i.e., in each monomial all the variables appear with exponent equal to $1_{U(L)}$. 

    Since $f$ is multilinear, it is enough to restrict to evaluations on matrix units.
    Moreover, $f^{(1)}$ is a multilinear ordinary polynomial, and hence its evaluation in $UT_k$ does not depend on the Lie algebra action on $UT_k$.

Since $f$ is a differential identity of $UT_k^{\delta_1, \ldots, \delta_r}$, for every evaluation $\varphi:X \longrightarrow UT_k$  we have 
 \begin{equation}\label{eq: evaluation d^d and f^1}
        \varphi_{\theta}(f^{(d)})=-\varphi(f^{(1)}).
    \end{equation}


Now fix an evaluation  $\varphi:X \longrightarrow UT_k$ and let 
\begin{equation*}
    \varphi_{\theta^{(s)}}(f^{(d)})= \sum_{1\leq i < j \leq k} \alpha_{ij} e_{ij}
\end{equation*}
and 
\begin{equation*}
    \varphi_{\theta}(f^{(d)})=  \sum_{1\leq i < j \leq k} \beta_{ij} e_{ij},
\end{equation*}
where $\alpha_{ij},\beta_{ij}\in F.$  
We aim to show that $\alpha_{ij}=\beta_{ij}$ for all $1\leq i < j \leq k$. It will then follow that  
$\varphi_{\theta^{(s)}}(f^{(d)})=\varphi_{\theta}(f^{(d)})$,
and, using \eqref{eq: evaluation d^d and f^1}, we conclude that
$ \varphi_{\theta^{(s)}}(f)=0$.
As this holds for any evaluation $\varphi$, we deduce that $f\in \I^L \big( UT_k^{\delta_1^{(s)}, \ldots, \delta_r^{(s)}}\big)$, as required.

For distinct indices
$
p_1,\ldots,p_l\in\{1,\ldots,n\},
$
let $M_{p_1,\ldots, p_l}$ denote the set of all monomials $m$ of $f^{(d)}$ with  nonzero coefficient such that only the variables $x_{p_1}, \ldots, x_{p_l}$, may appear with exponent different from $1_{U(L)}$.


Let $1\leq i < j \leq k$ and suppose that   $\beta_{ij}\neq 0$. 
Then, by \eqref{eq: evaluation d^d and f^1}, the coefficient of $e_{ij}$ in  $\varphi(f^{(1)})$ is nonzero.
Since $f^{(1)}$ is an ordinary polynomial,there exists a monomial of $f^{(1)}$ whose evaluation is a nonzero multiple of $e_{ij}.$ Then by the
multiplication table of $UT_k$  there exist distinct indices $p_1, \ldots, p_l$ such that the corresponding variables are evaluated on off-diagonal matrix units whose ordered product is $e_{ij},$ while all the remaining variables are evaluated on diagonal matrix units.

 By \eqref{eq: delta as delta^s and delta^n} and \eqref{eq: action of delta^s and delta^n on e_ij} if $m\in M_{p_1,\ldots, p_l}$, then
 $$\varphi_{\theta}(m)=\gamma_m^{(s)} e_{ij}+ j_m$$
where $\gamma_m^{(s)}
\in F$ depends  on $\delta_1^{(s)}, \ldots, \delta_r^{(s)}$ and $j_m\in J^{j-i+1}$ depends on  $\delta_1^{(n)}, \ldots, \delta_r^{(n)}$; if $m\notin M_{p_1,\ldots, p_l}$, then
$ \varphi_{\theta}(m)\in J^{j-i+1}.$
Since $e_{ij}\in J^{j-i}$ but $e_{ij}\notin J^{j-i+1}$, we obtain
$$
\beta_{ij}= \sum_{m\in M_{p_1,\ldots, p_l}} \tau_m  \gamma_m^{(s)},
$$
where $\tau_m\in F$ is the coefficient of $m$ in $f^{(d)}$ for any $m\in M_{p_1,\ldots, p_l}$. 

On the other hand, since by \eqref{eq: action of delta^s and delta^n on e_ij}
 $\varphi_{\theta^{(s)}}(m)=\gamma_m^{(s)} e_{ij}$ for all $m\in M_{p_1,\ldots, p_l}$, and
$\varphi_{\theta^{(s)}}(m)=0$ for all $m\notin M_{p_1,\ldots, p_l}$, we get 
$$
\alpha_{ij}=\sum_{m\in M_{p_1,\ldots, p_l}} \tau_m \gamma_m^{(s)}.
$$
Therefore, $\alpha_{ij}=\beta_{ij},$ as claimed.


Now suppose that $\beta_{ij}=0$ for some $1\leq i <j\leq k$. We prove that $\alpha_{ij}=0$. Assume, by contradiction, that $\alpha_{ij}\neq0$.
Then there exists at least a monomial $m$ in $f^{(d)}$, with nonzero coefficient, such that 
$$
\varphi_{\theta^{(s)}}(m)=\gamma_m^{(s)} e_{ij}\neq 0,
$$ 
where $\gamma_m^{(s)}
\in F$ depends on $\delta_1^{(s)}, \ldots, \delta_r^{(s)}$. Since by \eqref{eq: action of delta^s and delta^n on e_ij} each $\delta_l^{(s)}$ sends any $e_{ab}$ with $a\neq b$ to a scalar multiple of itself, the multiplication table of $UT_k$ implies that there exist distinct indices
$p_1,\ldots,p_l$
such that the corresponding variables are evaluated on off-diagonal matrix units whose ordered product is $e_{ij}$, while all the remaining variables are evaluated on diagonal matrix units.

By \eqref{eq: action of delta^s and delta^n on e_ij} and the multiplication table of $UT_k$ we get that: 
if $m'\in M_{p_1,\ldots, p_l}$, then
 $\varphi_{\theta^{(s)}}(m')=\gamma_{m'}^{(s)} e_{ij}$ where $\gamma_{m'}^{(s)}
 \in F$ depends on $\delta_1^{(s)}, \ldots, \delta_r^{(s)}$; otherwise, if $m'\notin M_{p_1,\ldots, p_l}$,
 $\varphi_{\theta^{(s)}}(m')=0$. 
Consequently, the coefficient of $e_{ij}$ in $\varphi_{\theta^{(s)}}(f^{(d)})$ is 
 $$
 \alpha_{ij} =\sum_{m'\in M_{p_1,\ldots, p_l}} \tau_{m'}\gamma_{m'}^{(s)},
 $$ 
 where $\tau_{m'}$ is the coefficient of $m'$ in $f^{(d)}$ for any $m'\in M_{p_1,\ldots, p_l}$.
 
On the other side,  by \eqref{eq: delta as delta^s and delta^n} and \eqref{eq: action of delta^s and delta^n on e_ij} we have that
$\varphi_{\theta}(m')=\gamma_{m'}^{(s)} e_{ij}+ j_{m'}$ with $j_{m'}\in J^{j-i+1}$ for all $m'\in M_{p_1,\ldots, p_l}$ and
 $\varphi_{\theta}(m')\in J^{j-i+1}$ for all $m'\notin M_{p_1,\ldots, p_l}$. Thus,  since $e_{ij}\in J^{j-i}$ but $e_{ij}\notin J^{j-i+1}$, the coefficient of $e_{ij}$ in $\varphi_{\theta}(f^{(d)})$ is
 $$
 \beta_{ij}=\sum_{m'\in M_{p_1,\ldots, p_l}} \tau_{m'} \gamma_{m'}^{(s)}.
 $$
 Therefore, $\alpha_{ij}=\beta_{ij}$, contradicting the assumption that
$
\beta_{ij}=0$ and $
\alpha_{ij}\neq0.
$
Hence
$\alpha_{ij}=0.
$

\end{proof}

Observe that in case $k=2$ the above result was proved in \cite[Corollary 4.5]{Rizzo2023} by using the explicit description of the $T_L$-ideals of differential identities of $UT_2$ with all the possible $L$-actions by derivations determined in \cite{GiambrunoRizzo2019, Rizzo2021}. 

\begin{remark}\label{rmk: semisimple action}
For any minimal $L$-variety of exponent $k$ generated by $UT_k$, we may assume, without loss of generality, that the $L$-action on $UT_k$ is semisimple, that is, generated by derivations induced by diagonal elements of $UT_k$.
   Indeed, given an $L$-action $\theta$ on $UT_k$, let $\theta^{(s)}:=s\circ\theta$ denote its semisimple part, where $s$ is the Lie algebra homomorphism defined in Lemma~\ref{diagonal-hom}. Then, by \eqref{exponent},
$\exp^L(UT_k^{\theta^{(s)}})=\exp^L(UT_k^\theta)=\exp(UT_k)=k$. On the other end, 
the above theorem implies that $\V^L(UT_k^{\theta^{(s)}})\subseteq\V^L(UT_k^\theta).$ 
Therefore, if $\V^L(UT_k^\theta)$ is minimal, the above inclusion cannot be proper; otherwise, $\V^L(UT_k^\theta)$ would contain a proper $L$-subvariety with the same exponent, contradicting the minimality of $\V^L(UT_k^\theta)$. Hence,
$\V^L(UT_k^\theta)=\V^L(UT_k^{\theta^{(s)}})$. 
\end{remark}

\section{Differential identities of $UT_3$}

In this section, we determine the $T_L$-ideals of the algebra $ UT_3 $ under distinct $L$-actions by derivations. To this end, let 
$$
\varepsilon_1 := \ad_{(-e_{11})} \quad \text{and} \quad \varepsilon_2 := \ad_{e_{33}},
$$
and consider the $L$-algebras:
\begin{enumerate}
    \item $ UT_3 $, where $ L $ acts trivially;
    \vspace{1mm}
    
    \item $ UT_3^{\varepsilon_1} $, where $ L $ acts via the one-dimensional Lie algebra generated by $ \varepsilon_1 $;
    \vspace{1mm}
    
    \item $ UT_3^{\varepsilon_2} $, where $ L $ acts via the one-dimensional Lie algebra generated by $ \varepsilon_2 $;

  \vspace{1mm}
  
  \item $ UT_3^{\varepsilon_1+ \beta\varepsilon_2} $, where $\beta \in F\setminus \{0\}$, and $ L $ acts via the one-dimensional Lie algebra generated by $ \varepsilon_1+ \beta\varepsilon_2$;

    \vspace{1mm}

    \item $ UT_3^{\varepsilon_1, \varepsilon_2} $, where $ L $ acts via the two-dimensional abelian Lie algebra with basis $ \{ \varepsilon_1, \varepsilon_2 \} $.
\end{enumerate}

\begin{remark}\label{rmk: equivalence}Recall (Definition~\ref{def: equivalence}) that $UT_3^\theta$ and $UT_3^{\theta'}$ are equivalent
if $\theta(L)=\theta'(L)$.
Equivalent structures need not have the same $T_L$-ideal of identities, but they always have the
same $L$-codimension sequence, as follows.

Let $\bar L=\spn\{\delta_1,\ldots,\delta_r\}$ be the common image of $\theta,\theta'$, and choose
adapted bases $\{d_i\}$, $\{d_i'\}$ of $L$ with $\theta(d_i)=\theta'(d_i')=\delta_i$ for $i\le r$
and $d_j,d_j'\in\ker\theta,\ker\theta'$ for $j>r$. For an $L$-polynomial $g$ whose exponents are
written in terms of $d_1,\ldots,d_r$ only, let $g'$ be obtained by relabelling $d_i\mapsto d_i'$.
Since $\theta(d_i)=\theta'(d_i')$, evaluating $g$ via $\theta$ and $g'$ via $\theta'$ at the same
$\varphi:X\to UT_3$ performs the same operations, so $\varphi_\theta(g)=\varphi_{\theta'}(g')$ for
every $\varphi$; hence $g\equiv0$ on $UT_3^\theta$ if and only if $g'\equiv0$ on $UT_3^{\theta'}$.

Consequently, if $\I^L(UT_3^\theta)$ is generated by finitely many $L$-polynomials $g_1,\ldots,g_m$
of this form (exponents in $d_1,\ldots,d_r$ only) together with $x^{d_j}\equiv0$ for $j>r$, then
$\I^L(UT_3^{\theta'})$ is generated by $g_1',\ldots,g_m'$ together with the same vanishing
conditions $x^{d_j'}\equiv0$, and $c_n^L(UT_3^\theta)=c_n^L(UT_3^{\theta'})$ for all $n$.

For this reason, we fix one representative $\theta$ for each of the five images spanned by
$\varepsilon_1$ and $\varepsilon_2$, and write \emph{the} $L$-algebra $UT_3^{\varepsilon_1}$ (etc.) for
this fixed representative: as verified in each theorem below, its generators are of the form just
described, so the resulting $T_L$-ideal (up to relabelling) and $L$-codimension sequence do not
depend on the choice of representative.
\end{remark}

Then we have the following result,  which underlines the importance of determining the $T_L$-ideals of the $L$-algebras listed above.

\begin{proposition}\label{pro: UT_k and UT_3}
Let $L$ be a Lie algebra and $\theta:L\to\D(UT_k)$ a Lie algebra homomorphism. Then, for every $k\geq 3$, the $L$-variety $\V^L(UT_k^\theta)$
contains at least one of the following algebras:
$$
UT_3,\quad UT_3^{\varepsilon_1},\quad UT_3^{\varepsilon_2},\quad
UT_3^{\varepsilon_1+\beta\varepsilon_2},\quad UT_3^{\varepsilon_1,\varepsilon_2},
$$
where $\beta\in F\setminus\{0\}$, for a suitable Lie algebra homomorphism $\overline{\theta}:L \to \D(UT_3)$. 
\end{proposition}
\begin{proof}
We first consider the case $k=3$, so $\theta:L\to\D(UT_3)$. If $\theta(L)=0$, then $UT_3^\theta=UT_3$
is already in the list. So assume $\theta(L)\neq0$ and let $r=\dim\theta(L)\ge1$; fix a basis
$\{\delta_1,\ldots,\delta_r\}$ of $\theta(L)$, so that $UT_3^\theta=UT_3^{\delta_1,\ldots,\delta_r}$.
By Theorem~\ref{Thm: var UT_k}, setting $\theta^{(s)}:=s\circ\theta$, we have $\theta^{(s)}(L)=\spn\{\delta_1^{(s)},\ldots,\delta_r^{(s)}\}$ and 
$UT_3^{\delta_1^{(s)},\ldots,\delta_r^{(s)}} \in \V^L(UT_3^\theta)$, where $UT_3^{\delta_1^{(s)},\ldots,\delta_r^{(s)}}=UT_3^{\theta^{(s)}}$.

 By Lemma~\ref{diagonal-hom}, $s(\D(UT_3))=\{\ad_b\mid b\text{ diagonal}\}$, and so $\{\varepsilon_1,\varepsilon_2\}$ is a basis of $s(\D(UT_3))$, which therefore has dimension $2$.
Hence $\theta^{(s)}(L)\subseteq s(\D(UT_3))=
\spn\{\varepsilon_1,\varepsilon_2\}$, and $\dim\theta^{(s)}(L)\in\{0,1,2\}$.
Clearly if $\dim\theta^{(s)}(L)=0$, then $UT_3^{\delta_1^{(s)},\ldots,\delta_r^{(s)}}=UT_3.$ 

If $\dim\theta^{(s)}(L)=1$, then
$\theta^{(s)}(L)$ is generated by $\alpha\varepsilon_1+\beta\varepsilon_2$ for some
$(\alpha,\beta)\ne(0,0);$
 up to rescaling, this is one of the algebras $UT_3^{\varepsilon_1}$, $UT_3^{\varepsilon_2}$,
or $UT_3^{\varepsilon_1+\beta'\varepsilon_2}$ for some $\beta'\neq0$.

If $\dim\theta^{(s)}(L)=2$, then
$\theta^{(s)}(L)=\spn\{\varepsilon_1,\varepsilon_2\}$, so $UT_3^{\delta_1^{(s)},\ldots,\delta_r^{(s)}}
=UT_3^{\varepsilon_1,\varepsilon_2}$.

Now suppose $k>3$ and consider $I=\spn\{e_{ij}\mid j-i  >0  ,\ j>3\}$.   Clearly $I$ is an $L$-ideal of $UT_k^{\theta}$, and in particular $I$ is invariant under $\theta(L)$.   Hence the quotient $A=UT_k/I$ has a natural structure of $L$-algebra given by the Lie homomorphism $\theta_A:L\to\D(A)$ defined by
$$
(a+I)^{\theta_A(d)}:=a^{\theta(d)}+I,
$$
for all $d\in L$ and $a+I\in A$.

Let $B=\spn\{e_{ij}+I\mid i\le j,\ j=1,2,3\}$. Then $B$ is an $L$-subalgebra of $A^{\theta_A}$ (i.e., invariant
under $\theta_A(L)$), and $B$ is isomorphic to $UT_3$ as an ordinary algebra via some isomorphism
$\chi:B\to UT_3,$ transporting this action along $\chi$, i.e.\ setting
$\tilde\theta(d):=\chi\circ\theta_A(d)|_B\circ\chi^{-1}$ for $d\in L$, yields a Lie algebra
homomorphism $\tilde\theta:L\to\D(UT_3)$ such that $\chi(a^{\theta_A(d)})=\chi(a)^{\tilde\theta(d)}$
for all $a\in B,\ d\in L$; that is, $\chi$ is an $L$-isomorphism between $B$ (with this action) and
$UT_3^{\tilde\theta}$. Consequently, $\I^L(B)=\I^L(UT_3^{\tilde\theta})$, and hence
$\V^L(B)=\V^L(UT_3^{\tilde\theta})$.

Applying the first part of the proof to $\tilde\theta$, we conclude that $\V^L(UT_3^{\tilde\theta})$
contains at least one of the algebras
$$
UT_3,\quad UT_3^{\varepsilon_1},\quad UT_3^{\varepsilon_2},\quad
UT_3^{\varepsilon_1+\beta\varepsilon_2},\quad UT_3^{\varepsilon_1,\varepsilon_2},
$$
where $\beta\in F\setminus\{0\}$.  Since $\V^L(UT_3^{\tilde\theta})=\V^L(B)\subseteq\V^L(A)\subseteq
\V^L(UT_k^\theta)$, it follows that $\V^L(UT_k^\theta)$ contains at least one of the algebras
listed above.
\end{proof}

Next we determine the $T_L$-ideals of the $L$-algebras $UT_3,$ $UT_3^{\varepsilon_1},$ $UT_3^{\varepsilon_2},$ $UT_3^{\varepsilon_1+\beta \varepsilon_2}$, with $\beta \neq 0$, and $UT_3^{\varepsilon_1, \varepsilon_2}$.

For $k\geq 2$, we denote by $[x_1,x_2,\ldots, x_k]$ the left-normed commutator of length $k$ in $x_1,x_2,\ldots, x_k$. Moreover, we call any variable with an exponent different from $1_{U(L)}$ a commutator of length $1$.  We begin with the following general technical results concerning left-normed commutators.

\begin{lemma}\label{lem: ordered comm}
For any $n\geq 4 $ and $3\leq a\leq n-1$, we have
$$
[x_1,\ldots, x_a,x_{a+1}, \ldots, x_{n}]= [x_1,\ldots, x_{a-1}, x_{a+1},x_{a}, x_{a+2}, \ldots, x_{n}] + \Gamma,
$$
where $\Gamma$ is a linear combination of polynomials of the form
\begin{align*}
    &[x_1,x_2, x_{i_1}, \ldots, x_{i_k}][x_a,x_{a+1}, x_{j_1}, \ldots, x_{j_{n-k-4}}],\\
    &[x_a,x_{a+1}, x_{i_1}, \ldots, x_{i_{k}}][x_1,x_2, x_{j_1}, \ldots, x_{j_{n-k-4}}],
\end{align*}
where $0\leq k \leq n-4$.  
\end{lemma}
\begin{proof}
    By the Jacobi identity, it follows that
$$
[[x_1,x_2],[x_3,x_4],x_5]=[[x_1,x_2,x_5],[x_3,x_4]]+ [[x_1,x_2],[x_3,x_4,x_5]].
$$
As a consequence, for $n\geq 4$ and $3\leq a\leq n-1$, we have that $[[x_1,\ldots, x_{a-1}], [x_a,x_{a+1}], x_{a+2}, \ldots, x_{n}]$ is a linear combination of products of the form
\begin{align*}
    &[x_1,x_2, x_{i_1}, \ldots, x_{i_k}][x_a,x_{a+1}, x_{j_1}, \ldots, x_{j_{n-k-4}}],\\
    &[x_a,x_{a+1}, x_{i_1}, \ldots, x_{i_{k}}][x_1,x_2, x_{j_1}, \ldots, x_{j_{n-k-4}}],
\end{align*}
where $0\leq k \leq n-4$.
Hence, since
\begin{align*}
    [x_1,\ldots, x_a,x_{a+1}, \ldots, x_{n}]= &[[x_1,\ldots, x_{a-1}], [x_a,x_{a+1}], x_{a+2}, \ldots, x_{n}] \nonumber\\
    &+[x_1,\ldots, x_{a-1}, x_{a+1},x_{a}, x_{a+2}, \ldots, x_{n}],
\end{align*}
for $n\geq 4$ and $3\leq a \leq n-1$, we obtain the desired conclusion. 
\end{proof}

\begin{lemma}\label{lem: x_1^u[x_2,x_3]}
    Let $u\in U(L)$. For any $n\geq 3$, the commutator $[x_1^u[x_2,x_3],x_4, \ldots , x_n]$ is a linear combination of $L$-polynomial of the form
    $$[x_{1}^u,x_{i_1},\dots,x_{i_k}][x_{2}, x_3, x_{j_1},\dots,x_{j_{n-k-3}}],$$ where $0\leq k\leq n-3$.
\end{lemma}
\begin{proof}
    If $n=3$, there is nothing to prove. So, assume $n=4$. Then
    \begin{equation*}
    \begin{split}
          [x_1^u[x_2,x_3],x_4]&= x_1^u[x_2,x_3]x_4- x_4x_1^u[x_2,x_3] = x_1^u[x_2,x_3,x_4]+ x_1^u x_4[x_2,x_3]- x_4x_1^u[x_2,x_3]\\
          &=x_1^u[x_2,x_3,x_4]+ [x_1^u ,x_4][x_2,x_3],
    \end{split}
    \end{equation*}
    which establishes the claim for $n=4$.
    Iterating this procedure for general $n$ we obtain the desired conclusion.
\end{proof}

Similarly, we can prove the following.

\begin{lemma}\label{lem: [x_1,x_2]x_3^u}
    Let $u\in U(L)$. For any $n\geq 3$, the commutator $[[x_1,x_2]x_3^u, x_4, \ldots , x_n]$ is
    a linear combination of $L$-polynomial of the form  $$[x_{1},x_2, x_{i_1},\dots,x_{i_k}][x_3^u, x_{j_1},\dots,x_{j_{n-k-3}}],$$ where $0\leq k\leq n-3$.
\end{lemma}

\begin{lemma}\label{lem: x_1^{u_1}x_2^{u_2}}
    Let $u_1, u_2\in U(L)$. For any $n\geq 2$, the commutator $[x_1^{u_{1}}x_2^{u_{2}}, x_3, \ldots, x_n]$ can be written as a linear combination of $L$-polynomial of the form 
    \begin{align*}
        &[x_{1}^{u_{1}},x_{i_1},\dots,x_{i_{k}}][x_{2}^{u_{2}},x_{j_1},\dots,x_{j_{n-k-2}}],
    \end{align*}
     where $0\leq k \leq n-2$. 
\end{lemma}
\begin{proof}
     If $n=2$, there is nothing to prove. So, assume $n=3$. Then
    \begin{equation*}
    \begin{split}
          [x_1^{u_1}x_2^{u_2},x_3]&= x_1^{u_1}x_2^{u_2}x_3- x_3x_1^{u_1}x_2^{u_2} = x_1^{u_1}[x_2^{u_2},x_3]+ x_1^{u_1} x_3x_2^{u_2}- x_3 x_1^{u_1}x_2^{u_2}=x_1^{u_1}[x_2^{u_2},x_3]+ [x_1^{u_1} ,x_3]x_2^{u_2},
    \end{split}
    \end{equation*}
    which establishes the claim for $n=3$.
    Iterating this procedure for general $n$ we obtain the desired conclusion.
\end{proof}

\begin{lemma}\label{lem: ordered comm x_1^u}
    Let $u\in U(L)$. For any $n\geq 3$ and $2\leq a \leq n-1$, we have
$$
[x_1^u,x_2,\ldots, x_a,x_{a+1}, \ldots, x_{n}]= [x_1^u,x_2,\ldots, x_{a-1}, x_{a+1},x_{a}, x_{a+2}, \ldots, x_{n}] + \Gamma,
$$
where $\Gamma$ is a linear combination of $L$-polynomials of the form
\begin{align*}
    &[x_1^u, x_{i_1}, \ldots, x_{i_k}][x_a, x_{a+1}, x_{j_1}, \ldots, x_{j_{n-k-3}}],\\
    &[x_a, x_{a+1}, x_{i_1}, \ldots, x_{i_k}][x_1^u, x_{j_1}, \ldots, x_{j_{n-k-3}}],
\end{align*}
where $0\leq k \leq n-3$.
\end{lemma}
\begin{proof}
    Let $n\geq 3$. If $3\leq  a \leq n-1$, then the desired conclusion follows directly from Lemma~\ref{lem: ordered comm}. 
    
    It remains to consider the case $a=2$. By the Jacobi identity, we obtain
    $$
    [x_1^u,x_2,x_3, \ldots, x_n]=[x_1^u[x_2,x_3],x_4,\ldots, x_n]- [[x_2,x_3]x_1^u,x_4, \ldots, x_n]+ [x_1^u,x_3,x_2,x_4,\ldots,x_n].
    $$
    Applying Lemmas~\ref{lem: x_1^u[x_2,x_3]} and~\ref{lem: [x_1,x_2]x_3^u} to the first two terms yields the desired conclusion.
\end{proof}

Now, we determine the $T_L$-ideal of the differential identities of the $L$-algebra $UT_3.$

\begin{theorem}\label{Teo: Id^L UT_3 ordinarie}
Let $L$ be a Lie algebra and consider the algebra $UT_3$ where $L$ acts trivially (i.e., $\theta=0$). Then 
there exists a basis $\mathcal{B}_L=\{d_i \mid i \geq 1\}$ of $L$ over $F$ such that 
the $T_L$-ideal of $L$-identities of $UT_3$ is generated by the following $L$-polynomials
$$x^{d_i}, \quad  [x_1,x_2][x_3,x_4][x_5,x_6],$$
for all  $i\geq 1.$ Moreover, $c_n^L(UT_3)=(n^2-7n+9) 3^{n-2}+ (3n-6)2^{n-1}+3.$ 
\end{theorem}
\begin{proof}
     Since $x^d\equiv 0$ for all $d\in L$ is a differential identity of $UT_3$, then we are dealing with the ordinary polynomial identities, and by \cite[Theorem 5.2.1]{Drenskybook} we get $\I^L(UT_3)=\langle x^{d_i}, \, [x_1,x_2][x_3,x_4][x_5,x_6] \mid i\geq 1\rangle_{T_L},$
and  the following $L$-polynomials form a basis of $P_{n}^L$ modulo $P_{n}^L\cap \I^L(UT_3)$
\begin{equation*}
x_{i_{1}}\dots x_{i_{r}}[x_{j_1},x_{j_2},\dots,x_{j_s}][x_{k_1},x_{k_2},\dots,x_{k_t}],
\end{equation*}
where $r,s,t \geq 0,$ $s,t\neq 1,$  $i_1< \cdots < i_r,$ $j_1>j_2<  \cdots < j_s,$ $k_1>k_2<  \cdots < k_t.$
Now, by counting these polynomials, we obtain
     \begin{align*}
         c_n^L(UT_3)= & 1+ \sum_{s=2}^n \binom{n}{s}(s-1) + \sum_{r=0}^{n-4}\binom{n}{r}\left( \sum_{s=2}^{n-r-2}\binom{n-r}{s}(s-1)(n-r-s-1)\right)\\
        = & 1+ \sum_{s=2}^n \binom{n}{s}(s-1) + \sum_{m=4}^{n}\binom{n}{m}\left( \sum_{s=2}^{m-2}\binom{m}{s}\big(ms-s^2-(m-1)\big)\right)\\
          =& 1+ \sum_{s=2}^n \binom{n}{s}(s-1) + \sum_{m=4}^{n}\binom{n}{m}\Big(m^2 2^{m-2}-5m2^{m-2}+2^{m}+2m-2\Big)\\
          =& (n^2-7n+9) 3^{n-2}+ (3n-6)2^{n-1}+3,
     \end{align*}
     as claimed.
\end{proof}

Now, we determine the $T_L$-ideal of the differential identities of the $L$-algebra $UT_3^{\e_1}.$ 
To this end, note that $ \varepsilon_1 $ acts on the matrix units in the following way:
\begin{equation}
    (e_{11})^{\varepsilon_1} = (e_{22})^{\varepsilon_1} = (e_{33})^{\varepsilon_1} = (e_{23})^{\varepsilon_1} = 0, \quad (e_{12})^{\varepsilon_1} = e_{12}, \quad (e_{13})^{\varepsilon_1} = e_{13}, \label{eq: e1}
\end{equation}

Now, we establish the following further useful lemmas.

\begin{lemma} \label{lem: consequences Ide_1 1 var}
Let $d\in L$. Then $x_1^{d}x_2^{d},\ x_1^{d}x_2 x_3^{d}\in \langle x^{d^2}-x^{d}\rangle_{T_L}.$
\end{lemma}
\begin{proof}
Using the Leibniz rule, we have that
$$(x_1 x_2)^{d^2} -(x_1 x_2)^d= x_1^{d^2} x_2 + 2 x_1^d x_2^d + x_1 x_2^{d^2} - x_1^d x_2 - x_1 x_2^d.$$
Since the characteristic of $F$ is different from $2,$ it follows that $x_1^d x_2^d\in \langle x^{d^2}-x^{d}\rangle_{T_L} .$
Moreover, since $(x_1x_2)^d x_3^d=x_1^{d}x_2 x_3^{d}+ x_1x_2^{d} x_3^{d},$ it follows that $x_1^{d}x_2 x_3^{d}$ is a consequence of  $x_1^d x_2^d$, and we are done.
\end{proof}

As a consequence of the above lemma, we obtain the following.
\begin{lemma}\label{lem: consequences Ide_1}
    Let $d\in L$. Then 
    $$ x_1^{d}[x_2,x_3][x_4,x_5], \, x_1^{d}x_2[x_3,x_4][x_5,x_6]\in \langle  x^{d^2}-x^{d}, \ \big([x_1,x_2]^{d}-[x_1,x_2]\big)[x_3,x_4]\rangle_{T_L}.$$
\end{lemma}

By straightforward computation, we obtain the following result.
\begin{lemma}\label{lem: conseq 3 commutators 1} 
\black
    Let $u\in U(L)$. Then $[x_1,x_2][x_3,x_4][x_5,x_6]\in \langle [x_1,x_2]x_3^{u},\; \big([x_1,x_2]^{u}-[x_1,x_2]\big)[x_3,x_4] \rangle_{T_L}$.
\end{lemma}

\begin{theorem}\label{Teo: Id^L UT_3 epsilon1}
Let $L$ be a Lie algebra and let $UT_3^{\e_1}$  be  the $L$-algebra $UT_3$ where $L$ acts via the one-dimensional Lie subalgebra of $\D(UT_3)$ spanned by the inner derivation $\e_1=\ad_{(-e_{11})}$. Then 
there exists a basis $\mathcal{B}_L=\{d_i \mid i \geq 1\}$ of $L$ over $F$ such that 
the $T_L$-ideal of differential identities of $UT_3^{\e_1}$ is generated by the following $L$-polynomials
$$
x^{d_i}, \quad  x^{d_1^2}-x^{d_1}, \quad [x_1,x_2]x_3^{d_1}, \quad \big([x_1,x_2]^{d_1}-[x_1,x_2]\big)[x_3,x_4],
$$
for all  $i\geq 2.$ Moreover,
 $c_n^L(UT_3^{\e_1})=(n^2-4n)3^{n-2}+ (3n-2)2^{n-1} +2.$ 
\end{theorem}
\begin{proof}
 Since $L$ acts on $UT_3^{\e_1}$ by derivation, there exists a Lie algebra homomorphism $\theta\colon L \to  \D(UT_3)$ with image $\theta(L)=F\e_1$. Therefore, we may choose a basis $\mathcal{B}_L=\{d_i \mid i \geq 1\}$ of $L$ such that
 \begin{equation}\label{eq: B_L e_1}
      \theta(d_1)=\e_1 \qquad \mbox{and} \qquad d_i\in \ker \theta \quad \mbox{for all} \ i\geq 2.
 \end{equation}
Let $I=\langle x^{d_i}, \  x^{d_1^2}-x^{d_1}, \ [x_1,x_2]x_3^{d_1}, \ \big([x_1,x_2]^{d_1}-[x_1,x_2]\big)[x_3,x_4] \mid i\geq 2 \rangle_{T_L}$. By \eqref{eq: e1}, \eqref{eq: B_L e_1} and straightforward calculations, we have that $I\subseteq\I^L(UT_3^{\e_1}).$

     To prove the opposite inclusion, let  $f$ be a differential polynomial of $P_n^L.$ By the Poincaré--Birkhoff--Witt (PBW) Theorem, $f$ can be written as a linear combination of polynomials of the type
     \begin{equation}\label{eq: PBW}
          x_{i_1}^{u_{a_1}}\dots x_{i_r}^{u_{a_r}} c_1 \dots c_b,
     \end{equation}
     where $r,b\geq 0,$ $i_1< \dots < i_r,$ $u_{a_1}, \ldots, u_{a_r}\in \B_{U(L)}$ and each $c_1, \ldots, c_b$ is a left normed commutator in the variables $x_j^{u_e}$ with $u_e\in \B_{U(L)}.$ Now, since $x^{d_i}, \; x^{d_1^2}-x^{d_1}\in I$, $i\geq 2$, it follows that, modulo $I$, the exponents satisfy $u_{a_1}, \ldots, u_{a_r}, u_e \in \{1_{U(L)}, d_1\}.$ Moreover, by Lemma~\ref{lem: consequences Ide_1 1 var}, we have $x_1^{d_1}x_2^{d_1},\, x_1^{d_1}x_2 x_3^{d_1} \in I$, and hence, modulo $I$, at most one variable in \eqref{eq: PBW} may occur with exponent $d_1$. In addition, Lemma~\ref{lem: conseq 3 commutators 1} yields $[x_1,x_2][x_3,x_4][x_5,x_6]\in I$, so that in \eqref{eq: PBW} at most two left-normed commutators may appear modulo $I$; that is, $0\leq b\leq 2$. Furthermore, since $[x_1,x_2]x_3^{d_1}\in I$ and by the Leibniz rule also $[x_1,x_2]x_3x_4^{d_1}\in I$, whenever a variable with exponent $d_1$ appears in \eqref{eq: PBW}, it must occur either in the tail $x_{i_1}^{u_{a_1}}\dots x_{i_r}^{u_{a_r}} $, i.e., $u_{a_l}=d_1$ for some $1\leq l \leq r$, or inside the first commutator $c_1$.
     If $d_1$ appears as an exponent of a variable in $c_1$, then by Lemma~\ref{lem: ordered comm} and by the Jacobi identity, we may assume that it occurs in the first entry of $c_1.$

     Now, observe that for any $0\leq k \leq r$, 
     $$
      x_{i_1}\cdots x_{i_k}x_{i_{k+1}}^{d_1}x_{i_{k+2}}\cdots x_{i_r}=  x_{i_1}\cdots x_{i_k}x_{i_{k+2}}x_{i_{k+1}}^{d_1}x_{i_{k+3}}\cdots x_{i_r}+  x_{i_1}\cdots x_{i_k}[x_{i_{k+1}}^{d_1},x_{i_{k+2}}]x_{i_{k+3}}\cdots x_{i_r}.
     $$
     Consequently, any monomial of the form $ x_{i_1}\cdots x_{i_k}x_{i_{k+1}}^{d_1}x_{i_{k+2}}\cdots x_{i_r}$ can be expressed as a linear combination of polynomials of the type
     $ x_{i_1}\cdots x_{i_k}[x_{j_1}^{d_1},x_{j_2},\ldots, x_{j_{r-k}}]
     $. 
     Therefore, since Lemma~\ref{lem: consequences Ide_1} gives $x_1^{d}[x_2,x_3][x_4,x_5], \, x_1^{d}x_2[x_3,x_4][x_5,x_6]\in I$, it follows that, modulo $I$, the polynomial $f$ can be written as a linear combination of polynomials of the forms
     \begin{align}
&x_{i_{1}}\dots x_{i_{r}}[x_{j_1},x_{j_2},\dots,x_{j_s}][x_{k_1},x_{k_2},\dots,x_{k_t}],\label{eq: no Id 1}\\
&x_{i_{1}}\dots x_{i_{r}}[x_{h_1}^{d_1},x_{h_2},\dots,x_{h_q}][x_{m_1},x_{m_2},\dots,x_{m_w}], \label{eq: no Id 2}
\end{align}
where $r,s,t,w \geq 0,$ $s,t,w\neq 1,$ $q\geq 1$ such that $i_1< \cdots < i_r.$ 
By Lemma~\ref{lem: ordered comm}, we may assume, modulo $I$, that $j_3<  \cdots < j_s,$ $k_3<  \cdots < k_t$ and $m_3<  \cdots < m_w$ in \eqref{eq: no Id 1} and \eqref{eq: no Id 2}. 
Additionally, by applying the Jacobi identity, and, if necessary, Lemma~\ref{lem: ordered comm} once more, we may further assume, modulo $I$, that $j_1>j_2< j_3,$  $k_1>k_2< k_3$ and $m_1>m_2< m_3$ in \eqref{eq: no Id 1} and \eqref{eq: no Id 2}. Moreover, by Lemma~\ref{lem: ordered comm x_1^u}, we may suppose, modulo $I$, that $h_2< \cdots <h_q$ in \eqref{eq: no Id 2}.
Furthermore, since 
$$[x_1^{d_1},x_2][x_3,x_4]\equiv [x_2^{d_1},x_1][x_3,x_4] + [x_1,x_2][x_3,x_4] \pmod{I},$$ 
if $w\geq 2$ in \eqref{eq: no Id 2}, we may also assume, modulo $I$, that $h_1<h_2.$   
This implies that the following $L$-polynomials generate $P_{n}^L$ modulo $P_{n}^L\cap I$:
\begin{equation}\label{P_n^L modulo Id UT_3^e1}
\begin{split}
    &x_{i_{1}}\dots x_{i_{r}}[x_{j_1},x_{j_2},\dots,x_{j_s}][x_{k_1},x_{k_2},\dots,x_{k_t}],\\
&x_{i_{1}}\dots x_{i_{r}}[x_{h_1}^{d_1},x_{h_2},\dots,x_{h_q}][x_{m_1},x_{m_2},\dots,x_{m_w}],
\end{split}
\end{equation}
where $r,s,t,w \geq 0,$ with $s,t,w\neq 1,$ and $q\geq 1$ such that $i_1< \cdots < i_r,$ $j_1>j_2<  \cdots < j_s,$ $k_1>k_2<  \cdots < k_t,$ $m_1>m_2<  \cdots < m_w,$ $h_2< \cdots <h_q,$ and if $w\geq 2,$ $h_1< h_2.$

Next, we prove that these polynomials are also linearly independent modulo $\I^L(UT_3^{\e_1}).$ To this end, let $f\in \I^L(UT_3^{\e_1})$ be a linear combination of the differential polynomials  \eqref{P_n^L modulo Id UT_3^e1}. By taking into account \eqref{eq: e1}, we shall make suitable evaluations to prove that all the scalars appearing in $f$ are zero, and this will complete the proof.

Since $UT_3$ has a unit element, we may, without loss of generality, disregard the tail $x_{i_{1}}\dots x_{i_{r}}$ ($i_1 < \cdots < i_r$), where all the variables have exponent $1_{U(L)}$; 
indeed, let $r_{\max}$ be the maximal tail length occurring in $f$, and fix a term
with tail $x_{i_1}\ldots x_{i_{r_{\max}}}$. For any evaluation $\varphi$ with
$\varphi(x_{i_1})=\cdots=\varphi(x_{i_{r_{\max}}})=1_{UT_3}$: a term with this exact tail evaluates
to the value of its commutator(s) alone, since $1_{UT_3}$ is a multiplicative identity; a term
whose tail is different from $x_{i_1}\ldots x_{i_{r_{\max}}}$ vanishes, since at least a variable in $\{x_{i_1},\ldots,x_{i_{r_{\max}}}\}$ occurs inside a commutator and evaluates to $1_{UT_3}$ there. As $f\equiv0$,
this yields $\varphi(f)=0$ for every such $\varphi$, an equation involving only the coefficients of
the terms with tail exactly $x_{i_1}\ldots x_{i_{r_{\max}}}$; determining the coefficients as below shows
they vanish. Repeating for every maximal tail, and then by decreasing induction on the tail length,
shows all coefficients vanish, so we may assume that every variable $x_1,\ldots,x_n$ that appears
in $f$ with exponent $1_{U(L)}$ occurs exclusively within a commutator, i.e.,
we may suppose that
\begin{align*}
    f = &\sum_{\mathcal{J}_{j_1}, \mathcal{K}_{k_1}} \alpha_{ \mathcal{J}_{j_1}, \mathcal{K}_{k_1}} [x_{j_1},x_{j_2},\dots,x_{j_s}][x_{k_1},x_{k_2},\dots,x_{k_t}]\\
    &+ \sum_{ \mathcal{H}_{h_1}, \mathcal{M}_{m_1}}  \beta_{\mathcal{H}_{h_1}, \mathcal{M}_{m_1}} [x_{h_1}^{d_1},x_{h_2},\dots,x_{h_q}][x_{m_1},x_{m_2},\dots,x_{m_w}],
\end{align*}
    where $\mathcal{J}_{j_1}=\{j_1,\ldots,j_s\},$ $\mathcal{K}_{k_1}=\{k_1,\ldots,k_t\},$ $\mathcal{H}_{h_1}=\{h_1,\ldots,h_q\},$ $\mathcal{M}_{m_1}=\{m_1,\ldots,m_w\},$  are subsets of indices in $\{1,\ldots, n\}$ such that $\mathcal{J}_{j_1}\cap \mathcal{K}_{k_1}=\mathcal{H}_{h_1}\cap \mathcal{M}_{m_1}=\emptyset$, with  $s\geq 2,$ $t,w \geq 0,$ $t,w\neq 1,$ $q\geq 1,$ $s+t=q+w=n,$ and subjected to the conditions  $j_1>j_2<  \cdots < j_s,$ $k_1>k_2<  \cdots < k_t,$ $m_1>m_2<  \cdots < m_w,$ $h_2< \cdots <h_q,$ and if $w \geq 2,$ then $h_1< h_2.$

For a fixed set $\mathcal{J}_{j_1}=\{j_1,\ldots,j_n\}$ such that $j_1>j_2<  \cdots < j_n,$ if we evaluate $ x_{j_1}= e_{23}$ and  $x_{j_2}=\cdots = x_{j_n}= e_{33},$ then we get $\alpha_{ \mathcal{J}_{j_1}, \emptyset} e_{23}=0$ thus $\alpha_{ \mathcal{J}_{j_1}, \emptyset} =0.$

Now fixed the set $\mathcal{H}_{h_1}=\{h_1,\ldots,h_n\}$ such that $h_2<  \cdots < h_n,$ we make the evaluation $ x_{h_1}= e_{13}$ and  $x_{h_2}=\cdots = x_{h_n}= e_{33}.$ In this case we get $\beta_{ \mathcal{H}_{h_1}, \emptyset}  e_{13}=0,$ so  $\beta_{ \mathcal{H}_{h_1}, \emptyset}=0.$

Next, let $\mathcal{J}_{j_1}=\{j_1,\ldots,j_s\}$ and $\mathcal{K}_{k_1}=\{k_1,\ldots,k_t\}$ be two fixed disjoint subsets of $\{1,\ldots,n\}$ with $s,t\geq 2,$ $s+t=n,$ and subject to the conditions $j_1>j_2<  \cdots < j_s,$  $k_1>k_2<  \cdots < k_t.$ From the evaluation $ x_{j_1}= e_{12},$  $x_{j_2}=\cdots = x_{j_n}= e_{11},$ $ x_{k_1}= e_{23}$ and  $x_{k_2}=\cdots = x_{k_n}= e_{33},$ we get $\alpha_{ \mathcal{J}_{j_1}, \mathcal{K}_{k_1}} e_{13}=0$ thus $\alpha_{ \mathcal{J}_{j_1}, \mathcal{K}_{k_1}}=0.$ Here remark that the polynomial $[x_{h_1}^{d_1},x_{h_2},\dots,x_{h_s}][x_{k_1},x_{k_2},\dots,x_{k_t}]$ evaluates to zero since $j_1>j_2<  \cdots < j_s$ whereas $h_1<h_2<\cdots <h_s$ and $e_{11}^{\e_1}=0.$

Finally, for any fixed $\mathcal{H}_{h_1}=\{h_1,\ldots,h_q\}$ and  $\mathcal{M}_{m_1}=\{m_1,\ldots,m_w\}$ disjoint subset of $\{1,\ldots,n\}$ with $q\geq 1,$ $w\geq 2,$ $q+w=n,$ and subject to the conditions $h_1< \cdots < h_q,$  $m_1>m_2<  \cdots < m_w,$ 
we make the substitution $x_{h_1}=e_{12},$ $x_{h_2}=\cdots = x_{h_q}= e_{11},$ $x_{m_1}=e_{23}$ and $x_{m_2}= \cdots = x_{m_w}= e_{33}.$ Then, we get $\beta_{\mathcal{H}_{h_1}, \mathcal{M}_{m_1}}e_{13}=0,$ that is $\beta_{\mathcal{H}_{h_1}, \mathcal{M}_{m_1}}=0.$

		Therefore, all the scalars appearing in $f$ are zero, and as a consequence, the polynomials  \eqref{P_n^L modulo Id UT_3^e1} are linearly independent modulo $\I^L(UT_3^{\e_1})$. Since $P_n^L\cap I\subseteq P_n^L\cap \I^L(UT_3^{\e_1}),$ this proves that $\I^L(UT_3^{\e_1})= I$ and the $L$-polynomials \eqref{P_n^L modulo Id UT_3^e1} are a basis of $P_n^L$ modulo $P_n^L\cap \I^L(UT_3^{\e_1}).$

        Finally, to determine the $L$-codimensions of $UT_3^{\e_1}$, we count the $L$-polynomials in \eqref{P_n^L modulo Id UT_3^e1}. By Theorem \ref{Teo: Id^L UT_3 ordinarie},  we have
   \begin{align*}
            c_n^L(UT_3^{\e_1})=&c_n^L(UT_3)+ \sum_{m=1}^n \binom{n}{m} m +\sum_{r=0}^{n-3}\binom{n}{r}\left( \sum_{w=2}^{n-r-1}\binom{n-r}{w}(w-1)\right)\\
                =&(n^2-7n+9) 3^{n-2}+ (3n-6)2^{n-1}+3+  n2^{n-1}+   (n-3)3^{n-1}+(4-n)2^{n-1}-1\\
                =& (n^2-4n)3^{n-2}+ (3n-2)2^{n-1}  +2,
        \end{align*}
        as claimed.
\end{proof}

With similar techniques, we can also compute the $T_L$-ideal of differential identities of $UT_3^{\e_2}.$ Recall that $ \varepsilon_2 $ acts on the matrix units as follows:
\begin{equation}
  (e_{11})^{\varepsilon_2} = (e_{22})^{\varepsilon_2} = (e_{33})^{\varepsilon_2} = (e_{12})^{\varepsilon_2} = 0, \quad (e_{23})^{\varepsilon_2} = e_{23}, \quad (e_{13})^{\varepsilon_2} = e_{13}. \label{eq: e2}
\end{equation}

In analogy with Lemmas \ref{lem: consequences Ide_1} and \ref{lem: conseq 3 commutators 1}, we establish the following results.

\begin{lemma}\label{lem: consequences Ide_2}
    Let $d\in L$. Then
   $$[x_1,x_2][x_3,x_4]x_5^{d}, \ [x_1,x_2][x_3,x_4]x_5x_6^{d}\in \langle  x^{d^2}-x^{d}, \ [x_1,x_2]\big([x_3,x_4]^{d}-[x_3,x_4]\big)\rangle_{T_L}.$$
\end{lemma}

\begin{lemma}\label{lem: conseq 3 commutators 2}
    Let $u\in U(L)$. Then $[x_1,x_2][x_3,x_4][x_5,x_6]\in \langle x_1^{u} [x_2,x_3],\ [x_1,x_2]\big([x_3,x_4]^{u}-[x_3,x_4]\big) \rangle_{T_L}$.
\end{lemma}

Now, by using Lemmas~\ref{lem: ordered comm}, \ref{lem: ordered comm x_1^u}, \ref{lem: consequences Ide_1 1 var}, \ref{lem: consequences Ide_2} and  \ref{lem: conseq 3 commutators 2}, and following step-by-step the lines of Theorem \ref{Teo: Id^L UT_3 epsilon1} with the necessary changes, we can also compute the $T_L$-ideal of differential identities of $UT_3^{\e_2}.$

\begin{theorem}\label{Teo: Id^L UT_3 epsilon2}
Let $L$ be a Lie algebra and $UT_3^{\e_2}$  be the $L$-algebra $UT_3$ where $L$ acts via the one-dimensional Lie subalgebra of $\D(UT_3)$ spanned by the inner derivation $\e_2=\ad_{e_{33}}$. Then  
there exists a basis $\mathcal{B}_L=\{d_i \mid i \geq 1\}$ of $L$ over $F$ such that 
the $T_L$-ideal of $L$-identities of $UT_3^{\e_2}$ is generated by the following $L$-polynomials
$$x^{d_i}, \quad x^{d_1^2}-x^{d_1}, \quad  x_1^{d_1}[x_2,x_3], \quad [x_1,x_2] \big([x_3,x_4]^{d_1}-[x_3,x_4]\big),$$
for all $i\geq 2.$ Moreover, $c_n^L(UT_3^{\e_2})= (n^2-4n) 3^{n-2}+ (3n-2)2^{n-1}+2.$
\end{theorem}

\begin{remark}[General strategy]\label{rmk: general-strategy}
We record here, once and for all, the five-step scheme used throughout this section to compute
the $T_L$-ideal of identities and the $L$-codimension sequence of $UT_3$ endowed with an
$L$-action; each theorem below only specifies the case-dependent data.

\smallskip
\textbf{Step 1 (adapted basis).} Fix $\theta:L\to\D(UT_3)$ with the prescribed image, and a basis
$\mathcal B_L=\{d_i\mid i\ge1\}$ of $L$ adapted to $\theta$ (i.e., $\theta(d_1)=\delta_1$, or
$\theta(d_1)=\delta_1,\theta(d_2)=\delta_2$ in the two-dimensional case, and $d_i\in\ker\theta$
for the remaining indices).

\smallskip
\textbf{Step 2 (inclusion $I\subseteq\I^L(UT_3^\theta)$).} Let $I$ be the $T_L$-ideal generated by the
polynomials in the statement. Using the explicit action of $\theta(d_1)$ (and $\theta(d_2)$) on
the matrix units, one checks directly that $I\subseteq\I^L(UT_3^\theta)$.

\smallskip
\textbf{Step 3 (PBW reduction).} By the PBW Theorem, every $f\in P_n^L$ is a linear combination of
monomials $x_{i_1}^{u_{a_1}}\cdots x_{i_r}^{u_{a_r}}c_1\cdots c_b$. Case-specific technical lemmas
bound the exponents occurring modulo $I$, the number of commutators ($b\le2$), and force each
non-trivial exponent into a specific position (tail or a prescribed entry of a commutator).  In
particular, for any $u\in \B_{U(L)}$ and $1\le k\le r$, a variable with exponent $u$ occurring in
the tail can be moved into the first entry of a commutator via the identity
\begin{equation}\label{eq: tail-to-commutator}
x_{i_1}\cdots x_{i_{k-1}}x_{i_k}^{u}x_{i_{k+1}}\cdots x_{i_r} = x_{i_1}\cdots x_{i_{k-1}}x_{i_{k+1}}x_{i_k}^{u}x_{i_{k+2}}\cdots x_{i_r} + x_{i_1}\cdots x_{i_{k-1}}[x_{i_k}^{u},x_{i_{k+1}}]x_{i_{k+2}}\cdots x_{i_r},
\end{equation}
applied repeatedly (with $k$ decreasing) to obtain a linear combination of polynomials with the
variable of exponent $u$ in the first entry of a commutator.

\smallskip
\textbf{Step 4 (tail elimination).} Since $UT_3$ has a unit, evaluating the variables of a
maximal-length tail at $1_{UT_3}$ shows, by decreasing induction on the tail length, that every
variable with exponent $1_{U(L)}$ may be assumed to occur inside a commutator.

\smallskip
\textbf{Step 5 (linear independence via evaluations).} For each family of indices indexing a
candidate basis element, a suitable evaluation on matrix units, determined by the action of
$\theta(d_1)$ (and $\theta(d_2)$), isolates its coefficient and shows it must vanish, proving
linear independence; counting the resulting basis gives the $L$-codimension.
\end{remark}

\black
 Next, we describe the $T_L$-ideal of differential identities of $ UT_3^{\varepsilon_1+ \beta\varepsilon_2} $ with $\beta \in F\setminus \{0\}$. 
 To this end, notice that from \eqref{eq: e1} and \eqref{eq: e2}, it follows that
\begin{equation*}
   (e_{11})^{\e_1+\beta\e_2} = (e_{22})^{\e_1+\beta\e_2} = (e_{33})^{\e_1+\beta\e_2} =0,
   \end{equation*}
   \begin{equation*}
    (e_{12})^{\e_1+\beta\e_2} = e_{12}, \quad  (e_{23})^{\e_1+\beta\e_2} = \beta e_{23}, \quad (e_{13})^{\e_1+\beta\e_2} = (1+\beta) e_{13}.
\end{equation*}

\smallskip

We begin with the case 
$\beta=-1$. In this situation, $L$ acts on $UT_3$ as the one-dimensional Lie algebra generated by $\delta:=\e_1-\e_2=\ad_{e_{22}}.$ Moreover, if we set $\nu_{12}:=2^{-1}(\delta^2+\delta)$ and $\nu_{23}:=2^{-1}(\delta^2-\delta)$, we get
\begin{equation}\label{eq: nu_12}
    (e_{12})^{\nu_{12}}=e_{12}, \quad (e_{11})^{\nu_{12}} = (e_{22})^{\nu_{12}} = (e_{33})^{\nu_{12}}= (e_{23})^{\nu_{12}}=(e_{13})^{\nu_{12}} =0,
\end{equation}
and
\begin{equation}\label{eq: nu_23}
   (e_{23})^{\nu_{23}}=e_{23}, \quad  (e_{11})^{\nu_{23}} = (e_{22})^{\nu_{23}} = (e_{33})^{\nu_{23}}= (e_{12})^{\nu_{23}}=(e_{13})^{\nu_{23}} =0.
\end{equation}

Now, we determine the $T_L$-ideal of the differential identities of the $L$-algebra $UT_3^{\e_1-\e_2}.$ To this end, we first establish the following technical lemmas.
\begin{lemma}\label{lem: cons IdUT3e1-e2 1}
    Let $d\in L$. If we denote $u_{12}:=2^{-1}(d^2+d)$ and $u_{23}:=2^{-1}(d^2-d)$, then 
    \begin{equation*}
        \begin{split}
          x_1^{u_{12}} x_2^{u_{12}}, \quad x_1^{u_{23}} x_2^{u_{23}} \in \langle x^{d^3}-x^{d}\rangle_{T_L}.
        \end{split}
    \end{equation*}
\end{lemma}
\begin{proof}
    By the Leibniz Rule, we have
    \begin{align*}
        (x_1x_2)^{d^3} -(x_1x_2)^d&=x_1^{d^3}x_2+ x_1 x_2^{d^3}+ 3x_1^{d^2}x_2^d + 3 x_1^{d}x_2^{d^2}-x_1^dx_2-x_1x_2^d
    \end{align*}
    and
    \begin{align*}
        (x_1x_2)^{d^4}  -(x_1x_2)^{d^2} & =x_1^{d^4}x_2+ x_1 x_2^{d^4}+ 4x_1^{d^3}x_2^d + 4 x_1^{d}x_2^{d^3}+ 6x_1^{d^2}x_2^{d^2}-x_1^{d^2}x_2-x_1x_2^{d^2}-2x_1^dx_2^d.
    \end{align*}
    As a consequence, it follows that 
    $$x_1^{d^2}x_2^d + x_1^{d}x_2^{d^2}, \ x_1^{d^2}x_2^{d^2}+x_1^dx_2^d \in \langle x^{d^3}-x^{d}\rangle_{T_L}.$$
    Now, since $u_{12}=2^{-1}(d^2+d)$ and $u_{23}:=2^{-1}(d^2-d)$, we have
    $$
     x_1^{u_{12}} x_2^{u_{12}}
    =4^{-1}( x_1^{d^2}x_2^{d^2} + x_1^{d^2}x_2^{d}+x_1^{d}x_2^{d^2}+ x_1^{d}x_2^{d})\in \langle x^{d^3}-x^{d}\rangle_{T_L}
    $$
    and
    $$
    x_1^{u_{23}} x_2^{u_{23}} 
    =4^{-1}( x_1^{d^2}x_2^{d^2} - x_1^{d^2}x_2^{d}-x_1^{d}x_2^{d^2}+ x_1^{d}x_2^{d})\in \langle x^{d^3}-x^{d}\rangle_{T_L},
    $$
    as desired.
\end{proof}

As a consequence of the above lemma and by straightforward calculations, we obtain the following result.

\begin{lemma}\label{lem: cons IdUT3e1-e2 4}
     Let $d\in L$. If we denote $u_{12}:=2^{-1}(d^2+d)$ and $u_{23}:=2^{-1}(d^2-d)$, then 
     $$ x_1^{u_{12}}[x_2,x_3]x_4^{u_{23}} \in \langle x^{d^3}-x^{d}, \ x_1^{u_{12}} \big([x_2,x_3]^{u_{23}}-[x_2,x_3]\big) \rangle_{T_L}.$$
\end{lemma}

\begin{lemma}\label{lem: cons IdUT3e1-e2 2}
    Let $d\in L$. If we denote $u_{12}:=2^{-1}(d^2+d)$ and $u_{23}:=2^{-1}(d^2-d)$, then 
$$
(x_1 x_2)^{u_{12}}- x_1^{u_{12}} x_2 - x_1x_2^{u_{12}}+ x_1^{u_{12}} x_2^{u_{23}} \quad \text{and} \quad  
(x_1 x_2)^{u_{23}}- x_1^{u_{23}} x_2 - x_1x_2^{u_{23}}+ x_1^{u_{12}} x_2^{u_{23}}
$$
are consequences of the $L$-polynomials in $\{ x^{d^3}-x^{d}, \ x_1^{u_{23}}x_2^{u_{12}}\} $.
\end{lemma}
\begin{proof}
    Since $d=u_{12}-u_{23}$, it follows that 
    $x_1^d x_2^d =x_1^{u_{12}}x_2^{u_{12}}- x_1^{u_{12}}x_2^{u_{23}} - x_1^{u_{23}}x_2^{u_{12}} + x_1^{u_{23}} x_2^{u_{23}}.$
    Then, by Lemma \ref{lem: cons IdUT3e1-e2 1}, we get 
   $$
   x_1^{d}x_2^{d}\equiv - x_1^{u_{12}} x_2^{u_{23}}\pmod{\langle x^{d^3}-x^{d}, \ x_1^{u_{23}}x_2^{u_{12}} \rangle_{T_L}}.
   $$
As a consequence, since $u_{12}=2^{-1}(d^2+d)$ and $u_{23}=2^{-1}(d^2-d)$ and by the Leibniz rule, we obtain
$$
(x_1 x_2)^{u_{12}}\equiv x_1^{u_{12}} x_2 + x_1x_2^{u_{12}}- x_1^{u_{12}} x_2^{u_{23}} \pmod{\langle x^{d^3}-x^{d}, \ x_1^{u_{23}}x_2^{u_{12}} \rangle_{T_L}},
$$
and
$$
(x_1 x_2)^{u_{23}}\equiv x_1^{u_{23}} x_2 + x_1x_2^{u_{23}}- x_1^{u_{12}} x_2^{u_{23}} \pmod{\langle x^{d^3}-x^{d}, \ x_1^{u_{23}}x_2^{u_{12}} \rangle_{T_L}},
$$
as desired.
\end{proof}

\begin{lemma}\label{lem: cons IdUT3e1-e2 3}
    Let $d\in L$. If we denote $u_{12}:=2^{-1}(d^2+d)$ and $u_{23}:=2^{-1}(d^2-d)$, then
    $$ x_1^{u_{12}}x_2 x_3^{u_{12}}, \ x_1^{u_{23}} x_2 x_3^{u_{23}}, \ x_1^{u_{23}}x_2x_3^{u_{12}} \in \langle x^{d^3}-x^{d}, \ x_1^{u_{23}}x_2^{u_{12}}\rangle_{T_L}.$$
\end{lemma}
\begin{proof}
From Lemma \ref{lem: cons IdUT3e1-e2 2} it follows that 
$$
(x_1 x_2)^{u_{12}}x_3^{u_{12}}\equiv x_1^{u_{12}} x_2 x_3^{u_{12}} + x_1x_2^{u_{12}} x_3^{u_{12}}- x_1^{u_{12}} x_2^{u_{23}}x_3^{u_{12}} \pmod{\langle x^{d^3}-x^{d}, \ x_1^{u_{23}}x_2^{u_{12}} \rangle_{T_L}}.
$$
Consequently, by Lemma \ref{lem: cons IdUT3e1-e2 1} we obtain that  $x_1^{u_{12}}x_2 x_3^{u_{12}}\in \langle x^{d^3}-x^{d}, \ x_1^{u_{23}}x_2^{u_{12}}\rangle_{T_L}$. 

Analogous arguments show that $ x_1^{u_{23}} x_2 x_3^{u_{23}}, \, x_1^{u_{23}}x_2x_3^{u_{12}} \in \langle x^{d^3}-x^{d}, \ x_1^{u_{23}}x_2^{u_{12}}\rangle_{T_L}.$
\end{proof}

\begin{lemma}\label{lem: cons IdUT3e1-e2 5}
    Let $d\in L$. If we denote $u_{12}:=2^{-1}(d^2+d)$ and $u_{23}:=2^{-1}(d^2-d)$, then
     $$
 [x_1^{u_{12}},x_2] [x_3,x_4] - [x_2^{u_{12}},x_1] [x_3,x_4]-[x_1,x_2][x_3,x_4]
$$
is a consequence of the $L$-polynomials in $$\bigl\{x^{d^3}-x^{d}, \ x_1^{u_{23}}x_2^{u_{12}}, \ x_1^{u_{23}}[x_2,x_3],\ \big([x_1,x_2]^{u_{12}}-[x_1,x_2]\big)[x_3,x_4]\bigr\}.$$
\end{lemma}
\begin{proof}
From Lemma \ref{lem: cons IdUT3e1-e2 2} we obtain
$$
[x_1,x_2]^{u_{12}}\equiv [x_1^{u_{12}},x_2] - [x_2^{u_{12}},x_1]-x_1^{u_{12}} x_2^{u_{23}} + x_2^{u_{12}} x_1^{u_{23}} \pmod{\langle x^{d^3}-x^{d}, \ x_1^{u_{23}}x_2^{u_{12}} \rangle_{T_L}}.
$$
Consequently,
\begin{align*}
    \big([x_1,x_2]^{u_{12}}-[x_1,x_2]\big)[x_3,x_4]\equiv & [x_1^{u_{12}},x_2] [x_3,x_4] - [x_2^{u_{12}},x_1] [x_3,x_4]
    \\
    &-[x_1,x_2][x_3,x_4] \pmod{\langle x^{d^3}-x^{d}, \, x_1^{u_{23}}x_2^{u_{12}},\, x_1^{u_{23}}[x_2,x_3]\rangle_{T_L}}.
\end{align*}
    This proves the desired conclusion.
\end{proof}

Similarly, we can prove the following result.
\begin{lemma} \label{lem: cons IdUT3e1-e2 u_12 u_23}
     Let $d\in L$. If we denote $u_{12}:=2^{-1}(d^2+d)$ and $u_{23}:=2^{-1}(d^2-d)$, then:
     \begin{enumerate}
         \item $
 [x_1^{u_{12}},x_2] x_3^{u_{23}} - [x_2^{u_{12}},x_1] x_3^{u_{23}} -[x_1,x_2]x_3^{u_{23}} \in \langle x^{d^3}-x^{d}, \, x_1^{u_{23}}x_2^{u_{12}} , \, \big([x_1,x_2]^{u_{12}}-[x_1,x_2]\big) x_3^{u_{23}} \rangle_{T_L};
$
         \vspace{1mm}
         \item $ x_1^{u_{12}} [x_2^{u_{23}},x_3] - x_1^{u_{12}} [x_3^{u_{23}},x_2]-x_1^{u_{12}} [x_2,x_3]  \in \langle x^{d^3}-x^{d}, \, x_1^{u_{23}}x_2^{u_{12}} , \,   x_1^{u_{12}} \big([x_2,x_3]^{u_{23}}-[x_2,x_3]\big)\rangle_{T_L}$.
     \end{enumerate}
\end{lemma}

\begin{lemma}\label{lem: cons IdUT3e1-e2 6}
    Let $d\in L$. If we denote $u_{12}:=2^{-1}(d^2+d)$ and $u_{23}:=2^{-1}(d^2-d)$, then
   $$
 [x_1,x_2] [x_3^{u_{23}},x_4] - [x_1,x_2] [x_4^{u_{23}},x_3]-[x_1,x_2][x_3,x_4]
$$
is a consequence of the $L$-polynomials in 
\begin{align*}
   \Bigl\{  x^{d^3}-x^{d}, \ x_1^{u_{23}}x_2^{u_{12}}, \ [x_1,x_2]x_3^{u_{12}}, \ x_1^{u_{12}} \big([x_2,x_3]^{u_{23}}-[x_2,x_3]\big),\ \big([x_1,x_2]^{u_{12}}-[x_1,x_2]\big)[x_3,x_4] \Bigr\}.
\end{align*}
\end{lemma}
\begin{proof}
Since
\begin{equation*}
[x_1,x_2] \big([x_3,x_4]^{u_{23}}-[x_3,x_4]\big)\in \langle  x_1^{u_{12}} \big([x_2,x_3]^{u_{23}}-[x_2,x_3]\big),\ \big([x_1,x_2]^{u_{12}}-[x_1,x_2]\big)[x_3,x_4] \rangle_{T_L},
\end{equation*}
the result follows by arguments analogous to those of Lemma \ref{lem: cons IdUT3e1-e2 5}.
\end{proof}

\begin{remark}\label{rmk: leibniz rule comm}
     Since 
     the commutator satisfies the Leibniz rule, for any $u_1,u_2 \in U(L)$ we have:
    \begin{enumerate}
        \item $[x_1,x_2]x_3x_4^{u_{1}}\in \langle  [x_1,x_2]x_3^{u_{1}}\rangle_{T_L} $; 

        \vspace{1ex}
        \item  $ x_1^{u_{2}}x_2 [x_3, x_4]  \in \langle   x_1^{u_{2}}[x_2, x_3]\rangle_{T_L}$; 
        \vspace{1ex}
         \item $  x_1^{u_{1}}x_2[x_3,x_4]x_5^{u_{2}},  x_1^{u_{1}}[x_2,x_3]x_4x_5^{u_{2}} \in \langle x_1^{u_{1}}[x_2,x_3]x_4^{u_{2}} \rangle_{T_L}$.
    \end{enumerate}
\end{remark}

Now, by using the above results and following the scheme of Remark~\ref{rmk: general-strategy},
we can compute the $T_L$-ideal of differential identities of $UT_3^{\e_1-\e_2}.$

\begin{theorem}\label{teo: Id^L UT_3 e1-e2}
   Let $L$ be a Lie algebra and $UT_3^{\e_1-\e_2}$  be the $L$-algebra $UT_3$ where $L$ acts via the one-dimensional Lie subalgebra of $\D(UT_3)$ spanned by the inner derivation $\e_1-\e_2 =\ad_{e_{22}}$. Then  
there exists a basis $\mathcal{B}_L=\{d_i \mid i \geq 1\}$ of $L$ over $F$ such that 
the $T_L$-ideal of $L$-identities of $UT_3^{\e_1-\e_2}$ is generated by the following $L$-polynomials
    \begin{align*} 
       & x^{d_i}, \quad x^{d_1^3}-x^{d_1}, \quad  x_1^{u_{23}}x_2^{u_{12}}, \quad [x_1,x_2]x_3^{u_{12}}, \quad  x_1^{u_{23}}[x_2,x_3],\quad  x_1^{u_{12}} \big([x_2,x_3]^{u_{23}}-[x_2,x_3]\big), \\  & \big([x_1,x_2]^{u_{12}}-[x_1,x_2]\big) x_3^{u_{23}},\quad \big([x_1,x_2]^{u_{12}}-[x_1,x_2]\big)[x_3,x_4],
    \end{align*}
for all $i\geq 2,$ where $u_{12}:=2^{-1}(d_1^2+d_1)$ and $u_{23}:=2^{-1}(d_1^2-d_1).$ 

Moreover, $c_n^L(UT_3^{\e_1 -\e_2})=(n^2-n)3^{n-2}+(3n-2)2^{n-1}+2$.
\end{theorem}
\begin{proof}
We follow the scheme of Remark~\ref{rmk: general-strategy}, with $\theta(d_1)=\e_1-\e_2$; recall
that $\e_1-\e_2$ acts on the matrix units as in \eqref{eq: nu_12}--\eqref{eq: nu_23}, via
$u_{12}=2^{-1}(d_1^2+d_1)$ and $u_{23}=2^{-1}(d_1^2-d_1)$, whose action annihilates every matrix
unit except, respectively, $e_{12}$ and $e_{23}$.

\textbf{Step 2} is immediate from \eqref{eq: nu_12}--\eqref{eq: nu_23}.

\textbf{Step 3.} Let $f\in P_n^L$. As in the proof of Theorem~\ref{Teo: Id^L UT_3 epsilon1}, the exponents occurring
modulo $I$ reduce to $\{1_{U(L)}, d_1, d_1^2\}$ and as a consequence to $\{1_{U(L)},u_{12},u_{23}\}$, at most two left-normed commutators occur
(Lemma~\ref{lem: conseq 3 commutators 1}), and, by Lemmas~\ref{lem: cons IdUT3e1-e2 1}
and~\ref{lem: cons IdUT3e1-e2 3}, at most one variable may carry exponent $u_{12}$ and at most one
may carry exponent $u_{23}$; so a monomial contains $0$, $1$, or $2$ variables with exponent
different from $1_{U(L)}$.

If none occurs, the argument is exactly as in Theorem~\ref{Teo: Id^L UT_3 ordinarie} (the trivial
action case), giving the ordinary family
$$
x_{i_1}\cdots x_{i_r}[x_{j_1},\ldots,x_{j_s}][x_{k_1},\ldots,x_{k_t}], \qquad
j_1>j_2<\cdots<j_s, \quad k_1>k_2<\cdots<k_t.
$$
If exactly one occurs, with exponent $u_{12}$ (resp.\
$u_{23}$), following the argument as in the proof of Theorem~\ref{Teo: Id^L UT_3 epsilon1} for the
single exponent $d_1$, using \eqref{eq: tail-to-commutator} (applied to $u_{12}$, resp.\ $u_{23}$) in
place of the analogous identity there, and Lemma~\ref{lem: cons IdUT3e1-e2 5} (resp.\
Lemma~\ref{lem: cons IdUT3e1-e2 6})  in place of the Leibniz rule argument fixing the order of the commutator in which appears $u_{12}$ (resp.\ $u_{23}$); this gives the families 
\begin{align*}
&x_{i_{1}}\dots x_{i_{r}}[x_{h_1}^{u_{12}},x_{h_2},\dots,x_{h_{q_1}}][x_{m_1},x_{m_2},\dots,x_{m_{w_1}}],\\ 
&x_{i_{1}}\dots x_{i_{r}}[x_{m'_1},x_{m'_2},\dots,x_{m'_{w_2}}][x_{h'_1}^{u_{23}},x_{h'_2},\dots,x_{h'_{q_2}}],
\end{align*}
where $r,w_1,w_2\geq 0$,  $w_1,w_2\neq 1$, $ q_1,q_2\geq 1,$  $i_1<\cdots < i_r$, $m_1>m_2< \cdots < m_{w_1},$  $m'_1>m'_2< \cdots < m'_{w_2},$  $h_2< \cdots <h_{q_1}$  $h'_2< \cdots <h'_{q_2}.$ Additionally,  if $w_1\geq 2$  then $h_1<h_2$   and if $w_2\geq 2$  then $h'_1<h'_2.$

The genuinely new case is when both $u_{12}$ and $u_{23}$ occur together in one monomial. Since
$x_1^{u_{23}}x_2^{u_{12}}\in I$ and, by Lemma~\ref{lem: cons IdUT3e1-e2 3}, also
$x_1^{u_{23}}x_2x_3^{u_{12}}\in I$, the two corresponding variables must appear, modulo $I$, in the fixed
order $u_{12}$ before $u_{23}$.  We examine the cases where the number of commutators $b$ is equal to $2,1,0$ in turn.

\smallskip
\noindent\emph{Case $b=2$.} Since $x_1^{u_{23}}[x_2,x_3]\in I$ and its Leibniz consequence
(Remark~\ref{rmk: leibniz rule comm}), $u_{23}$ must lie in the second commutator $c_2$; since
$$
x_1^{u_{12}}[x_2,x_3]x_4^{u_{23}}\in I \qquad \text{(Lemma~\ref{lem: cons IdUT3e1-e2 4})}
$$
and its Leibniz consequences (Remark~\ref{rmk: leibniz rule comm}), $u_{12}$ must lie in the first commutator $c_1$;
both are placed in the first entry of their commutator.

\smallskip

\noindent\emph{Case $b=1$.} As above, $u_{23}$ must lie in the single commutator, while $u_{12}$
may initially lie in the tail or in the commutator. In the former case, applying \eqref{eq: tail-to-commutator} to $u_{12}$ rewrites the polynomial as a linear combination of terms with two separate commutators, one containing $u_{12}$ and the other containing $u_{23}$. In the latter case, Lemma~\ref{lem: cons IdUT3e1-e2 4} ensures that $u_{12}$ and $u_{23}$ occupy the first two positions of the commutator.  Since $x_1^{u_{23}}x_2^{u_{12}}\in I$, we have
$$
[x_1^{u_{12}},x_2^{u_{23}},x_3]\equiv[x_1^{u_{12}}x_2^{u_{23}},x_3]\pmod I,
$$
and Lemma~\ref{lem: x_1^{u_1}x_2^{u_2}} splits this into two separate commutators, one for each
exponent.

\smallskip
\noindent\emph{Case $b=0$.} Applying \eqref{eq: tail-to-commutator} first moves the variable with
exponent $u_{23}$ into a commutator, giving a linear combination of terms
$$
x_{i_1}\cdots x_{i_{a-1}}x_{i_a}^{u_{12}}x_{i_{a+1}}\cdots x_{i_l}
[x_{j_1}^{u_{23}},x_{j_2},\ldots,x_{j_{n-l}}];
$$
applying again \eqref{eq: tail-to-commutator} to the remaining tail variable with exponent $u_{12}$ then
yields a linear combination of terms
$$
x_{i_1}\cdots x_{i_r}[x_{k_1}^{u_{12}},x_{k_2},\ldots,x_{k_m}]
[x_{j_1}^{u_{23}},x_{j_2},\ldots,x_{j_{n-r-m}}].
$$

\smallskip
In each of these three sub-cases, modulo $I$, both  variables with exponent different from $1_{U(L)}$ lie in the first entries of
two separate commutators, with $u_{12}$ leading. Ordering the remaining entries via
Lemma~\ref{lem: ordered comm x_1^u}, and fixing $l_1<l_2$, $l_1'<l_2'$ via
Lemmas~\ref{lem: cons IdUT3e1-e2 5} and~\ref{lem: cons IdUT3e1-e2 6}, we obtain, modulo $I$, the
family $x_{i_1}\cdots x_{i_r}[x_{l_1}^{u_{12}},\ldots,x_{l_{v_1}}][x_{l_1'}^{u_{23}},\ldots,x_{l_{v_2}'}]$.

Combining the four families above, we conclude that the following $L$-polynomials generate
$P_n^L$ modulo $P_n^L\cap I$:
\begin{equation}\label{No IdUT_3^e_1-e_2}
\begin{split}
&x_{i_1}\dots x_{i_r}[x_{j_1},x_{j_2},\dots,x_{j_s}][x_{k_1},x_{k_2},\dots,x_{k_t}],\\
&x_{i_1}\dots x_{i_r}[x_{h_1}^{u_{12}},x_{h_2},\dots,x_{h_{q_1}}][x_{m_1},x_{m_2},\dots,x_{m_{w_1}}],\\
&x_{i_1}\dots x_{i_r}[x_{m'_1},x_{m'_2},\dots,x_{m'_{w_2}}][x_{h'_1}^{u_{23}},x_{h'_2},\dots,x_{h'_{q_2}}],\\
&x_{i_1}\dots x_{i_r}[x_{l_1}^{u_{12}},x_{l_2},\dots,x_{l_{v_1}}][x_{l'_1}^{u_{23}},x_{l'_2},\dots,x_{l'_{v_2}}],
\end{split}
\end{equation}
where $r,s,t,w_1,w_2\ge0$, with $s,t,w_1,w_2\ne1$, and $q_1,q_2,v_1,v_2\ge1$ such that
$i_1<\cdots<i_r$, $j_1>j_2<\cdots<j_s$, $k_1>k_2<\cdots<k_t$, $h_2<\cdots<h_{q_1}$,
$m_1>m_2<\cdots<m_{w_1}$, $h'_2<\cdots<h'_{q_2}$, $m'_1>m'_2<\cdots<m'_{w_2}$,
$l_1<\cdots<l_{v_1}$, $l'_1<\cdots<l'_{v_2}$. Additionally, if $w_1\ge2$, then $h_1<h_2$, and if
$w_2\ge2$, then $h'_1<h'_2$.

\textbf{Step 4} is as in the proof of Theorem~\ref{Teo: Id^L UT_3 epsilon1}: since $UT_3$ has a
unit, we may, without loss of generality, disregard the tail $x_{i_1}\cdots x_{i_r}$. Hence, if  $f\in \I^L(UT_3^{\e_1-\e_2})$ is a linear combination of the polynomials in \eqref{No IdUT_3^e_1-e_2},   we may
suppose that
\begin{align*}
    f = &\sum_{\mathcal{J}_{j_1}, \mathcal{K}_{k_1}} \alpha_{ \mathcal{J}_{j_1}, \mathcal{K}_{k_1}} [x_{j_1},x_{j_2},\dots,x_{j_s}][x_{k_1},x_{k_2},\dots,x_{k_t}]\\
    &+ \sum_{ \mathcal{H}_{h_1}, \mathcal{M}_{m_1}}  \beta_{\mathcal{H}_{h_1}, \mathcal{M}_{m_1}} [x_{h_1}^{u_{12}},x_{h_2},\dots,x_{h_{q_1}}][x_{m_1},x_{m_2},\dots,x_{m_{w_1}}]\\
    &+ \sum_{ \mathcal{M}'_{m'_1}, \mathcal{H}'_{h'_1}}  \gamma_{ \mathcal{M}'_{m'_1}, \mathcal{H}'_{h'_1}}[x_{m'_1},x_{m'_2},\dots,x_{m'_{w_2}}][x_{h'_1}^{u_{23}},x_{h'_2},\dots,x_{h'_{q_2}}]\\
    &+ \sum_{\mathcal{L}, \mathcal{L}'}  \lambda_{\mathcal{L}, \mathcal{L}'}[x_{l_1}^{u_{12}},x_{l_2},\dots,x_{l_{v_1}}][x_{l'_1}^{u_{23}},x_{l'_2},\dots,x_{l'_{v_2}}],
\end{align*}
where where $\mathcal{J}_{j_1}=\{j_1,\ldots,j_s\},$ $\mathcal{K}_{k_1}=\{k_1,\ldots,k_t\},$ $\mathcal{H}_{h_1}=\{h_1,\ldots,h_{q_1}\},$ $\mathcal{M}_{m_1}=\{m_1,\ldots,m_{w_1}\},$ $\mathcal{H}'_{h'_1}=\{h'_1,\ldots,h'_{q_2}\},$ $\mathcal{M}'_{m'_1}=\{m'_1,\ldots,m'_{w_2}\},$ $\mathcal{L}=\{l_1,\ldots,l_{v_1}\}$ and $\mathcal{L}'=\{l'_1,\ldots,l'_{v_2}\}$
 are subsets of indices in $\{1,\ldots, n\}$ such that $\mathcal{J}_{j_1}\cap \mathcal{K}_{k_1}=\mathcal{H}_{h_1}\cap \mathcal{M}_{m_1}=\mathcal{H}'_{h'_1}\cap \mathcal{M}'_{m'_1}=\mathcal{L}\cap \mathcal{L}'=\emptyset$, with
$s\geq 2,$ $t, w_1,w_2\geq 0,$ $t,w_1,w_2\neq 1,$ $q_1,q_2,v_1,v_2\geq 1,$ $s+t=q_1+w_1=w_2+q_2=v_1+v_2=n,$ and subjected  to the conditions $j_1>j_2<  \cdots < j_s,$ $k_1>k_2<  \cdots < k_t,$ $m_1> m_2< \cdots < m_{w_1},$ $m'_1> m'_2< \cdots < m'_{w_2},$
$l_1< \cdots < l_{v_1},$ $l'_1< \cdots < l'_{v_2},$ $h_2< \cdots <h_{q_1},$ $h'_2< \cdots <h'_{q_2}.$ Additionally, if $w_1\geq 2,$ then $h_1< h_2$, and  if $w_2\geq 2,$ then $h'_1< h'_2.$

\textbf{Step 5.} We show that the $L$-polynomials \eqref{No IdUT_3^e_1-e_2} are linearly independent modulo
$\I^L(UT_3^{\e_1-\e_2})$; by \eqref{eq: nu_12} and \eqref{eq: nu_23}, we make suitable evaluations
to prove that all the coefficients of $f$ are zero.

Let $\mathcal J_{j_1}=\{j_1,\ldots,j_n\}$ be a fixed set such that $j_1>j_2<\cdots<j_n$. If we
evaluate $x_{j_1}=e_{13}$ and $x_{j_2}=\cdots=x_{j_n}=e_{33}$, then we get
$\alpha_{\mathcal J_{j_1},\emptyset}\,e_{13}=0$, and hence $\alpha_{\mathcal J_{j_1},\emptyset}=0$.

Now, let $\mathcal H_{h_1}=\{h_1,\ldots,h_n\}$ be a fixed set such that $h_2<\cdots<h_n$. By
evaluating $x_{h_1}=e_{12}$ and $x_{h_2}=\cdots=x_{h_n}=e_{22}$, we get
$\beta_{\mathcal H_{h_1},\emptyset}\,e_{12}=0$, and so $\beta_{\mathcal H_{h_1},\emptyset}=0$.
Also, by making the evaluation $x_{h_1}=e_{23}$ and $x_{h_2}=\cdots=x_{h_n}=e_{33}$, we get
$\gamma_{\emptyset,\mathcal H_{h_1}}\,e_{23}=0$, and hence $\gamma_{\emptyset,\mathcal H_{h_1}}=0$.

Next, let $\mathcal J_{j_1}=\{j_1,\ldots,j_s\}$ and $\mathcal K_{k_1}=\{k_1,\ldots,k_t\}$ be two
fixed disjoint subsets of $\{1,\ldots,n\}$ such that $s,t\ge2$, $s+t=n$ and
$j_1>j_2<\cdots<j_s$, $k_1>k_2<\cdots<k_t$. If we evaluate $x_{j_1}=e_{12}$,
$x_{j_2}=\cdots=x_{j_s}=e_{11}$, $x_{k_1}=e_{23}$ and $x_{k_2}=\cdots=x_{k_t}=e_{33}$, then we get
$\alpha_{\mathcal J_{j_1},\mathcal K_{k_1}}\,e_{13}=0$ and hence $\alpha_{\mathcal J_{j_1},
\mathcal K_{k_1}}=0$. Here the polynomials of the form
$[x_{h_1}^{u_{12}},x_{h_2},\ldots,x_{h_s}][x_{k_1},x_{k_2},\ldots,x_{k_t}]$ evaluate to zero since
$j_1>j_2<\cdots<j_s$ whereas $h_1<h_2<\cdots<h_s$ and $e_{11}^{\nu_{12}}=0$. Similarly, the
polynomials of the form
$[x_{j_1},x_{j_2},\ldots,x_{j_s}][x_{h_1'}^{u_{23}},x_{h_2'},\ldots,x_{h_t'}]$ evaluate to zero
since $k_1>k_2<\cdots<k_t$ whereas $h_1'<h_2'<\cdots<h_t'$ and $e_{33}^{\nu_{23}}=0$. Also, with
the same reasoning, all the polynomials of the type
$[x_{l_1}^{u_{12}},x_{l_2},\ldots,x_{l_{v_1}}][x_{l_1'}^{u_{23}},x_{l_2'},\ldots,x_{l_{v_2}'}]$
evaluate to zero.

Let $\mathcal H_{h_1}=\{h_1,\ldots,h_{q_1}\}$ and $\mathcal M_{m_1}=\{m_1,\ldots,m_{w_1}\}$ be two
fixed disjoint subsets of $\{1,\ldots,n\}$ with $q_1\ge1$, $w_1\ge2$, $q_1+w_1=n$ and subject to
the conditions $h_1<h_2<\cdots<h_{q_1}$, $m_1>m_2<\cdots<m_{w_1}$. If we make the evaluation
$x_{h_1}=e_{12}$, $x_{h_2}=\cdots=x_{h_{q_1}}=e_{11}$, $x_{m_1}=e_{23}$ and
$x_{m_2}=\cdots=x_{m_{w_1}}=e_{33}$, we get $\beta_{\mathcal H_{h_1},\mathcal M_{m_1}}\,e_{13}=0$,
that is $\beta_{\mathcal H_{h_1},\mathcal M_{m_1}}=0$.

Similarly, if we evaluate $x_{h_1}=e_{23}$, $x_{h_2}=\cdots=x_{h_{q_1}}=e_{33}$, $x_{m_1}=e_{12}$
and $x_{m_2}=\cdots=x_{m_{w_1}}=e_{11}$, we get $\gamma_{\mathcal M_{m_1},\mathcal
H_{h_1}}\,e_{13}=0$, and so $\gamma_{\mathcal M_{m_1},\mathcal H_{h_1}}=0$.

Finally, fixed $\mathcal L=\{l_1,\ldots,l_{v_1}\}$ and $\mathcal L'=\{l_1',\ldots,l_{v_2}'\}$
disjoint subsets of $\{1,\ldots,n\}$ such that $v_1,v_2\ge1$, $v_1+v_2=n$ and $l_1<\cdots<l_{v_1}$,
$l_1'<\cdots<l_{v_2}'$, then from the evaluation $x_{l_1}=e_{12}$,
$x_{l_2}=\cdots=x_{l_{v_1}}=e_{11}$, $x_{l_1'}=e_{23}$ and $x_{l_2'}=\cdots=x_{l_{v_2}'}=e_{33}$,
it follows that $\lambda_{\mathcal L,\mathcal L'}\,e_{13}=0$ and consequently
$\lambda_{\mathcal L,\mathcal L'}=0$.

Therefore, all the scalars appearing in $f$ are zero, and consequently, the $L$-polynomials
\eqref{No IdUT_3^e_1-e_2} are linearly independent modulo $\I^L(UT_3^{\e_1-\e_2})$. This proves that
$\I^L(UT_3^{\e_1-\e_2})=I$ and the $L$-polynomials \eqref{No IdUT_3^e_1-e_2} are a basis of $P_n^L$
modulo $P_n^L\cap\I^L(UT_3^{\e_1-\e_2})$.

  Next we determine the $L$-codimensions of $UT_3^{\e_1-\e_2}$ by counting the $L$-polynomials \eqref{No IdUT_3^e_1-e_2}. 
Hence, by Theorems \ref{Teo: Id^L UT_3 ordinarie} and \ref{Teo: Id^L UT_3 epsilon1}, we get
\begin{align*}
     c_n^L(UT_3^{\e_1-\e_2}) =&   2 c_n^L(UT_3^{\e_1})-c_n^L(UT_3)+\sum_{r=0}^{n-2} \binom{n}{r} \left( \sum_{v=1}^{n-r-1}\binom{n-r}{v}\right)\\
      = &2(n^2-4n)3^{n-2}+ 2(3n-2)2^{n-1} +4-(n^2-7n+9) 3^{n-2}-(3n-6)2^{n-1}-3\\
     &+3^n-2^{n+1}+1\\
      = &(n^2 - n)3^{n-2} + (3n - 2)2^{n-1} + 2,
\end{align*}
as claimed.
\end{proof}

Next, consider the case $\beta=1$. In this case, $L$ acts on $UT_3$ as the one-dimensional Lie algebra generated by $\delta:=\e_1+\e_2.$ Moreover, if we set $\nu_{2}:=2\delta-\delta^2$ and $\nu_{13}:=2^{-1}(\delta^2-\delta)$, we get
\begin{equation}\label{eq: nu_2}
    (e_{12})^{\nu_{2}}=e_{12}, \quad (e_{23})^{\nu_{2}}=e_{23}, \quad (e_{11})^{\nu_{2}} = (e_{22})^{\nu_2} = (e_{33})^{\nu_{2}}= (e_{13})^{\nu_{2}} =0,
\end{equation}
and
\begin{equation}\label{eq: nu_13}
   (e_{13})^{\nu_{13}}=e_{13}, \quad  (e_{11})^{\nu_{13}} = (e_{22})^{\nu_{13}} = (e_{33})^{\nu_{13}}= (e_{12})^{\nu_{13}}=(e_{23})^{\nu_{13}} =0.
\end{equation}

Our next goal is to determine the $T_L$-ideal of the differential identities of the $L$-algebra $UT_3^{\e_1+\e_2}.$ We begin with several technical lemmas.

\begin{lemma}\label{lem: UT_3^e1+e2 u_13u_13}
    Let $d\in L$. If we denote $u_{13}:=2^{-1}(d^2 - d)$, then $x_1^{u_{13}}x_2^{u_{13}}\in \langle x^{d^3}-3x^{d^2}+2x^{d}\rangle_{T_L}$.
\end{lemma}
\begin{proof}
 By the Leibniz Rule, we have
    \begin{align*}
        (x_1x_2)^{d^3} -3(x_1x_2)^{d^2} +2(x_1x_2)^d=&x_1^{d^3}x_2+ x_1 x_2^{d^3}+ 3x_1^{d^2}x_2^d + 3 x_1^{d}x_2^{d^2} \\ 
        &-3x_1^{d^2}x_2-3x_1x_2^{d^2}-6x_1^dx_2^d 
        +2x_1^dx_2+2x_1x_2^d
    \end{align*}
    and
    \begin{align*}
        (x_1x_2)^{d^4} -3(x_1x_2)^{d^3} +2 (x_1x_2)^{d^2} = &x_1^{d^4}x_2+ x_1 x_2^{d^4}+ 4x_1^{d^3}x_2^d + 4 x_1^{d}x_2^{d^3}+ 6x_1^{d^2}x_2^{d^2}
        -3x_1^{d^3}x_2\\
        &-3 x_1 x_2^{d^3}-9x_1^{d^2}x_2^d -9 x_1^{d}x_2^{d^2}
        +2x_1^{d^2}x_2+2x_1x_2^{d^2}+4x_1^dx_2^d.
    \end{align*}
    As a consequence, it follows that 
    $$x_1^{d^2}x_2^d + x_1^{d}x_2^{d^2}-2x_1^dx_2^d, \ x_1^{d^2}x_2^{d^2}-x_1^dx_2^d \in \langle x^{d^3}-3x^{d^2}+2x^{d} \rangle_{T_L}.$$
    Now, since  $u_{13}=2^{-1}(d^2-d)$, we have
    $$
    x_1^{u_{13}} x_2^{u_{13}}
    =4^{-1}( x_1^{d^2}x_2^{d^2} - x_1^{d^2}x_2^{d}-x_1^{d}x_2^{d^2}+ x_1^{d}x_2^{d})\in \langle x^{d^3}-3x^{d^2}+2x^{d}\rangle_{T_L},
    $$
    as desired.    
\end{proof}

As a direct consequence of the above lemma, we obtain the following result.
\begin{lemma}\label{lem:  UT_3^e1+e2 u_13 com}
   Let $d\in L$. If we denote $u_{2}:=2d -d^2$ and $u_{13}:=2^{-1}(d^2 - d)$, then:
   \begin{enumerate}
       \item $x_1^{u_{13}} [x_2,x_3]\in \langle x^{d^3}-3x^{d^2}+2x^{d},\ x_1^{u_{13}}x_2^{u_{2}}, \ [x_1,x_2]^{u_{13}}+ [x_1,x_2]^{u_{2}} - [x_1,x_2]\rangle_{T_L}$;
\vspace{1mm}

\item $ [x_1,x_2] x_3^{u_{13}}\in \langle x^{d^3}-3x^{d^2}+2x^{d},\ x_1^{u_{2}}x_2^{u_{13}}, \ [x_1,x_2]^{u_{13}}+ [x_1,x_2]^{u_{2}} - [x_1,x_2]\rangle_{T_L}$.
   \end{enumerate}
\end{lemma}

\begin{lemma}\label{lem: UT_3^e1+e2 u_2 com}
    Let $d\in L$. If we denote $u_{2}:=2d -d^2$ and $u_{13}:=2^{-1}(d^2 - d)$, then 
    \begin{enumerate}
        \item $x_1^{u_{2}} x_2^{u_{2}}[x_3,x_4],  \ x_1^{u_{2}} [x_2, x_3][x_4,x_5]\in \langle x_1^{u_{2}}x_2^{u_{13}}, \, x_1^{u_{2}} x_2^{u_{2}}x_3^{u_{2}}, \, [x_1,x_2]^{u_{13}}+ [x_1,x_2]^{u_{2}} - [x_1,x_2] \rangle_{T_L}$;
        \vspace{1mm}
        \item $[x_1, x_2]x_3^{u_{2}}x_4^{u_{2}}, \ x_1^{u_{2}} [x_2, x_3]x_4^{u_{2}} \in \langle  x_1^{u_{13}}x_2^{u_{2}}, \, x_1^{u_{2}} x_2^{u_{2}}x_3^{u_{2}}, \, [x_1,x_2]^{u_{13}}+ [x_1,x_2]^{u_{2}} - [x_1,x_2] \rangle_{T_L}$.
    \end{enumerate}
\end{lemma}
\begin{proof}
We have 
    \begin{align*}
        x_1^{u_{2}} x_2^{u_{2}}[x_3,x_4]^{u_{13}}+ x_1^{u_{2}} x_2^{u_{2}}[x_3,x_4]^{u_{2}} - &x_1^{u_{2}} x_2^{u_{2}}[x_3,x_4]\equiv - x_1^{u_{2}} x_2^{u_{2}}[x_3,x_4] \pmod{\langle x_1^{u_{2}}x_2^{u_{13}}, \, x_1^{u_{2}} x_2^{u_{2}}x_3^{u_{2}} \rangle_{T_L}}
    \end{align*}
    and, as a consequence,
     \begin{align*}
        x_1^{u_{2}} [x_2, x_3]^{u_{13}}[x_4,x_5] + x_1^{u_{2}} [x_2, x_3]^{u_{2}}[x_4,x_5] - &x_1^{u_{2}} [x_2, x_3][x_4,x_5]\\
        &\equiv - x_1^{u_{2}} [x_2, x_3][x_4,x_5] \pmod{\langle x_1^{u_{2}}x_2^{u_{13}}, \, x_1^{u_{2}} x_2^{u_{2}}x_3^{u_{2}} \rangle_{T_L}}.
    \end{align*}
    This proves (1). The proof of (2) is analogous.
\end{proof}

\begin{lemma}\label{lem: UT_3^e1+e2 3comm}
    Let $d\in L$. If we denote $u_{2}:=2d -d^2$ and $u_{13}:=2^{-1}(d^2 - d)$, then 
    $$[x_1,x_2][x_3,x_4][x_5,x_6]$$
    is a consequence of the $L$-polynomials in 
    $$\mathcal{S}:=\{ x^{d^3}-3x^{d^2}+2x^{d},\ x_1^{u_{2}}x_2^{u_{13}}, \ x_1^{u_{13}}x_2^{u_{2}}, \ x_1^{u_{2}} x_2^{u_{2}}x_3^{u_{2}}, \ [x_1,x_2]^{u_{13}}+ [x_1,x_2]^{u_{2}} - [x_1,x_2]\}.$$
\end{lemma}
\begin{proof}
    From item (1) of Lemmas \ref{lem: UT_3^e1+e2 u_13 com} and \ref{lem: UT_3^e1+e2 u_2 com}, it follows that
    \begin{align*}
        [x_1,x_2]^{u_{13}}[x_3,x_4][x_5,x_6]+[x_1,x_2]^{u_{2}}[x_3,x_4][x_5,x_6] - &[x_1,x_2][x_3,x_4][x_5,x_6]\\
        &\equiv - [x_1,x_2][x_3,x_4][x_5,x_6] \pmod{\langle \mathcal{S}\rangle_{T_L}}.
    \end{align*}
    This proves the desired conclusion.
\end{proof}

\begin{lemma}\label{lem: u_2 u_13 on x_1x_2}
    Let $d\in L$. If we denote $u_{2}:=2d -d^2$ and $u_{13}:=2^{-1}(d^2 - d)$, then
    $$
(x_1 x_2)^{u_{2}}- x_1^{u_{2}} x_2 - x_1x_2^{u_{2}}+ 2x_1^{u_{2}} x_2^{u_{2}}
\quad \text{and} \quad 
(x_1 x_2)^{u_{13}}- x_1^{u_{13}} x_2 - x_1x_2^{u_{13}}+ x_1^{u_{2}} x_2^{u_{2}}
$$
are consequences of the $L$-polynomials in $\{  x^{d^3}-3x^{d^2}+2x^{d},\ x_1^{u_{2}}x_2^{u_{13}}, \ x_1^{u_{13}}x_2^{u_{2}} \}$.
\end{lemma}
\begin{proof}
Since $d=u_2 +2 u_{13}$, it follows that 
    $x_1^d x_2^d =x_1^{u_{2}}x_2^{u_{2}}+2 x_1^{u_{2}}x_2^{u_{13}} +2 x_1^{u_{13}}x_2^{u_{2}} +4 x_1^{u_{13}} x_2^{u_{13}}.$ Then, by Lemma \ref{lem: UT_3^e1+e2 u_13u_13}, we get 
      $$
    x_1^{d}x_2^{d}- x_1^{u_{2}} x_2^{u_{2}}\in \langle x^{d^3}-3x^{d^2}+2x^{d},\, x_1^{u_{2}}x_2^{u_{13}}, \, x_1^{u_{13}}x_2^{u_{2}}\rangle_{T_L}.
    $$

As a consequence, since $u_{2}=2d -d^2$ and $u_{13}=2^{-1}(d^2 - d)$ and by the Leibniz rule, we obtain
$$
(x_1 x_2)^{u_2}\equiv x_1^{u_{2}} x_2 + x_1x_2^{u_{2}}-2 x_1^{u_{2}} x_2^{u_{2}} \pmod{\langle x^{d^3}-3x^{d^2}+2x^{d},\, x_1^{u_{2}}x_2^{u_{13}}, \, x_1^{u_{13}}x_2^{u_{2}} \rangle_{T_L}},
$$
and
$$
(x_1 x_2)^{u_{13}}\equiv x_1^{u_{13}} x_2 + x_1x_2^{u_{13}}- x_1^{u_{2}} x_2^{u_{2}} \pmod{\langle x^{d^3}-3x^{d^2}+2x^{d},\, x_1^{u_{2}}x_2^{u_{13}}, \, x_1^{u_{13}}x_2^{u_{2}}  \rangle_{T_L}},
$$
as desired.
\end{proof}

As a direct consequence of Lemma \ref{lem: u_2 u_13 on x_1x_2} we obtain the following. 

\begin{lemma}\label{lem: UT_3^e1+e2 cons var in the mid}
    Let $d\in L$. If we denote $u_{2}:=2d -d^2$ and $u_{13}:=2^{-1}(d^2 - d)$, then:
    \begin{enumerate}
        \item $x_1^{u_{2}} x_2x_3^{u_{13}}, \ x_1^{u_{13}} x_2x_3^{u_{2}}, \ x_1^{u_{13}}x_2x_3^{u_{13}} \in \langle x^{d^3}-3x^{d^2}+2x^{d},\, x_1^{u_{2}}x_2^{u_{13}}, \, x_1^{u_{13}}x_2^{u_{2}}\rangle_{T_L};$
    \vspace{1mm}
    \item   $x_1^{u_{2}} x_2 x_3^{u_{2}}x_4^{u_{2}}, \ x_1^{u_{2}} x_2^{u_{2}}x_3x_4^{u_{2}}\in \langle x^{d^3}-3x^{d^2}+2x^{d},\, x_1^{u_{2}}x_2^{u_{13}}, \, x_1^{u_{13}}x_2^{u_{2}},\, x_1^{u_{2}} x_2^{u_{2}}x_3^{u_{2}}\rangle_{T_L}.;$
    \vspace{1mm}
    \item $x_1^{u_{2}} x_2 x_3^{u_{2}}[x_4,x_5]\in \langle x^{d^3}-3x^{d^2}+2x^{d},\ x_1^{u_{2}}x_2^{u_{13}}, \, x_1^{u_{2}} x_2^{u_{2}}x_3^{u_{2}}, \, [x_1,x_2]^{u_{13}}+ [x_1,x_2]^{u_{2}} - [x_1,x_2] \rangle_{T_L}$;
    \vspace{1mm}
        \item $[x_1, x_2]x_3^{u_{2}}x_4x_5^{u_{2}} \in \langle x^{d^3}-3x^{d^2}+2x^{d},\ x_1^{u_{2}}x_2^{u_{13}}, \, x_1^{u_{2}} x_2^{u_{2}}x_3^{u_{2}}, \, [x_1,x_2]^{u_{13}}+ [x_1,x_2]^{u_{2}} - [x_1,x_2] \rangle_{T_L}$.
    \end{enumerate}
\end{lemma}

\begin{lemma}\label{lem: UT_3^e1+e2 u_2 ordering}
    Let $d\in L$. If we denote $u_{2}:=2d -d^2$ and $u_{13}:=2^{-1}(d^2 - d)$, then:
    \begin{align*}
        &[x_1^{u_{2}}, x_2]x_3^{u_{2}}- [x_2^{u_{2}}, x_1]x_3^{u_{2}} - [x_1, x_2]x_3^{u_{2}},\\
        &x_1^{u_{2}}[x_2^{u_{2}}, x_3]- x_1^{u_{2}}[x_3^{u_{2}}, x_2] - x_1^{u_{2}}[x_2, x_3],\\
        & [x_1^{u_{2}}, x_2][x_3, x_4]- [x_2^{u_{2}}, x_1][x_3, x_4]- [x_1, x_2][x_3, x_4], \\
        & [x_1, x_2][x_3^{u_{2}}, x_4]- [x_1, x_2][x_4^{u_{2}}, x_3]- [x_1, x_2][x_3, x_4],
    \end{align*}
     are consequences of the $L$-polynomials in
     $$
     \mathcal{S}:= \Bigl\{ x^{d^3}-3x^{d^2}+2x^{d},\ x_1^{u_{2}}x_2^{u_{13}}, \ x_1^{u_{13}}x_2^{u_{2}}, \ x_1^{u_{2}} x_2^{u_{2}}x_3^{u_{2}}, \  [x_1,x_2]^{u_{13}}+ [x_1,x_2]^{u_{2}} - [x_1,x_2]\Bigr\}.
     $$
\end{lemma}
\begin{proof}
By Lemma \ref{lem: u_2 u_13 on x_1x_2},
\begin{equation}\label{eq: [x_1, x_2]^u_2}
    [x_1, x_2]^{u_{2}}\equiv [x_1^{u_{2}}, x_2]  - [x_2^{u_{2}}, x_1] -2 [x_1^{u_{2}}, x_2^{u_{2}}] 
\pmod{\langle x^{d^3}-3x^{d^2}+2x^{d},\, x_1^{u_{2}}x_2^{u_{13}}, \, x_1^{u_{13}}x_2^{u_{2}} \rangle_{T_L}}.
\end{equation}
Hence,  
$$
[x_1^{u_{2}}, x_2] x_3^{u_{2}} - [x_2^{u_{2}}, x_1] x_3^{u_{2}}- [x_1,x_2]  x_3^{u_{2}}, \ x_1^{u_{2}}[x_2^{u_{2}}, x_3]- x_1^{u_{2}}[x_3^{u_{2}}, x_2] - x_1^{u_{2}}[x_2, x_3]\in \langle \mathcal{S}\rangle_{T_L}.
$$

Combining \eqref{eq: [x_1, x_2]^u_2} with Lemmas \ref{lem: UT_3^e1+e2 u_13 com} and \ref{lem: UT_3^e1+e2 u_2 com}, we obtain
$$
[x_1^{u_{2}}, x_2][x_3, x_4]- [x_2^{u_{2}}, x_1][x_3, x_4]- [x_1, x_2][x_3, x_4], \ [x_1, x_2][x_3^{u_{2}}, x_4]- [x_1, x_2][x_4^{u_{2}}, x_3]- [x_1, x_2][x_3, x_4]\in \langle \mathcal{S}\rangle_{T_L}.
$$
\end{proof}

As a direct consequence of Lemma \ref{lem: u_2 u_13 on x_1x_2} we also obtain the following.
\begin{lemma}\label{lem: UT_3^e1+e2 u_13 ordering}
    Let $d\in L$. If we denote $u_{2}:=2d -d^2$ and $u_{13}:=2^{-1}(d^2 - d)$, then:
  $$
   [x_1^{u_{13}},x_2]- [x_2^{u_{13}},x_2]+ [x_1^{u_{2}},x_2] -  [x_2^{u_{2}},x_1]- 3 [x_1^{u_{2}}, x_2^{u_{2}}] - [x_1,x_2]
  $$
    is a consequence of the $L$-polynomials in
     $$
     \Bigl\{ x^{d^3}-3x^{d^2}+2x^{d},\ x_1^{u_{2}}x_2^{u_{13}}, \ x_1^{u_{13}}x_2^{u_{2}}, \ [x_1,x_2]^{u_{13}}+ [x_1,x_2]^{u_{2}} - [x_1,x_2]\Bigr\}.
     $$
\end{lemma}

Again, by following a strategy similar to that of Theorems \ref{Teo: Id^L UT_3 epsilon1} and \ref{teo: Id^L UT_3 e1-e2}, and carrying out some additional computations, we can compute the $T_L$-ideal of differential identities of $UT_3^{\e_1+\e_2}$.

\begin{theorem}\label{teo: Id^L UT_3 e1+e2}
   Let $L$ be a Lie algebra and let $UT_3^{\e_1+\e_2}$  be the $L$-algebra $UT_3$ where $L$ acts via the one-dimensional Lie subalgebra of $\D(UT_3)$ spanned by the inner derivation  $\e_1+\e_2 =\ad_{(e_{33}-e_{11})}$. Then 
there exists a basis $\mathcal{B}_L=\{d_i \mid i \geq 1\}$ of $L$ over $F$ such that 
   the $T_L$-ideal of $L$-identities of $UT_3^{\e_1+\e_2}$ is generated by the following $L$-polynomials
   \begin{equation*}
        x^{d_i}, \quad x^{d_1^3}-3x^{d_1^2}+2x^{d_1},\quad x_1^{u_{2}}x_2^{u_{13}}, \quad x_1^{u_{13}}x_2^{u_{2}}, \quad  x_1^{u_{2}} x_2^{u_{2}}x_3^{u_{2}}, \quad  [x_1,x_2]^{u_{13}}+ [x_1,x_2]^{u_{2}} - [x_1,x_2],
    \end{equation*}
for all $i\geq 2,$ where $u_{2}:=2d_1 -d_1^2$ and $u_{13}:=2^{-1}(d_1^2 - d_1)$. 

Moreover, $c_n^L(UT_3^{\e_1 +\e_2})= (n^2-n)3^{n-2}+n2^n+1.$ 
\end{theorem}
\begin{proof}
We follow the scheme of Remark~\ref{rmk: general-strategy}, with $\theta(d_1)=\e_1+\e_2$; recall that $\e_1+\e_2$ acts on the matrix units as in \eqref{eq: nu_2} and \eqref{eq: nu_13}, via $u_2:=2d_1-d_1^2$ and $u_{13}:=2^{-1}(d_1^2-d_1)$, whose action annihilates every matrix unit except $e_{12}$ and $e_{23}$ (for $u_2$) and $e_{13}$ (for $u_{13}$).

\textbf{Step 2} is immediate from \eqref{eq: nu_2} and \eqref{eq: nu_13}.
     
\textbf{Step 3.} Let  $f\in P_n^L.$ As in the proof of Theorem~\ref{Teo: Id^L UT_3 epsilon1}, since $x^{d_i}, \; x^{d_1^3}-3x^{d_1^2}+2x^{d_1}\in I$ for $i\geq2$, the exponents occurring modulo $I$ reduce to $\{1_{U(L)}, u_{2}, u_{13}\}$; at most two left-normed commutators occur (Lemma~\ref{lem: UT_3^e1+e2 3comm}); and, by Lemma~\ref{lem: UT_3^e1+e2 u_13u_13} and items (1)--(2) of Lemma~\ref{lem: UT_3^e1+e2 cons var in the mid}, a monomial may involve at most two distinct variables whose exponent differs from $1_{U(L)}$.

If no variable occurs with an exponent different from $1_{U(L)}$, the argument is exactly as in Theorem~\ref{Teo: Id^L UT_3 ordinarie}, giving the ordinary family
\begin{equation*}
    x_{i_{1}}\dots x_{i_{r}}[x_{j_1},x_{j_2},\dots,x_{j_s}][x_{k_1},x_{k_2},\dots,x_{k_t}],
\end{equation*}
where $ r, s, t\geq 0$ with $s,t\neq 1$, and $i_1< \cdots < i_r$, $j_1>j_2<  \cdots < j_s$,  $k_1>k_2<  \cdots < k_t$.

Now suppose that exactly one variable appears with an exponent different from $1_{U(L)}$, namely either $u_{2}$ or $u_{13}$.

If this exponent is $u_{13}$, following the argument as in the proof of Theorem~\ref{Teo: Id^L UT_3 epsilon1} for the single exponent $d_1$: by Lemma~\ref{lem:  UT_3^e1+e2 u_13 com} and items (1)--(2) of Remark~\ref{rmk: leibniz rule comm}, the variable must occur, modulo $I$, either in the tail or in the first entry of the (unique) commutator, and \eqref{eq: tail-to-commutator} (applied to $u_{13}$) moves a tail occurrence into a commutator.

If this exponent is $u_2$, the argument mirrors the single-exponent case of Theorem~\ref{teo: Id^L UT_3 e1-e2} (applied twice, once with the variable of exponent $u_2$ occurring in the first commutator and once in the second): by Lemma~\ref{lem: UT_3^e1+e2 3comm} the variable must occur, modulo $I$, inside a commutator, and \eqref{eq: tail-to-commutator} (applied to $u_2$) moves a tail occurrence into a commutator, which by Lemma~\ref{lem: ordered comm} and the Jacobi identity may be assumed to occur in its first entry.
Therefore, this gives the families
     \begin{align}
&x_{i_{1}}\dots x_{i_{r}}[x_{h_1}^{u_{2}},x_{h_2},\dots,x_{h_{q_1}}][x_{m_1},x_{m_2},\dots,x_{m_{w_1}}],\label{eq: no Id u_2 1}\\
&x_{i_{1}}\dots x_{i_{r}}[x_{m'_1},x_{m'_2},\dots,x_{m'_{w_2}}][x_{h'_1}^{u_{2}},x_{h'_2},\dots,x_{h'_{q_2}}], \label{eq: no Id u_2 2}\\
&x_{i_{1}}\dots x_{i_{r}}[x_{p_1}^{u_{13}},x_{p_2},\dots,x_{p_z}], \label{eq: no Id u_13} 
\end{align}
where $r, w_1 \geq 0,$ $w_1\neq 1$, $q_1,q_2,z\geq 1$, $w_2\geq 2$ and  $i_1< \cdots < i_r$.
By Lemma~\ref{lem: ordered comm} and the Jacobi identity, we may assume, modulo $I$, that $m_1>m_2< \cdots < m_{w_1}$ in \eqref{eq: no Id u_2 1} and $m'_1>m'_2< \cdots < m'_{w_2}$ in \eqref{eq: no Id u_2 2}.
 In addition, by Lemma~\ref{lem: ordered comm x_1^u}, we may assume, modulo $I$, that $h_2< \cdots <h_{q_1}$, $h'_2< \cdots <h'_{q_2}$ and $p_2<\cdots <p_z$ in \eqref{eq: no Id u_2 1}, \eqref{eq: no Id u_2 2} and \eqref{eq: no Id u_13}, respectively.  Moreover, by Lemma \ref{lem: UT_3^e1+e2 u_2 ordering}, we may assume, modulo $I$, that $h'_1<h'_2$ in \eqref{eq: no Id u_2 2}, and also if $w_1\geq 2$ in \eqref{eq: no Id u_2 1}, then $h_1<h_2$. Furthermore, by Lemma \ref{lem: UT_3^e1+e2 u_13 ordering}, we may also assume, modulo $I$, that $p_1<p_2$ in \eqref{eq: no Id u_13}.

Finally, suppose that a monomial $g$ contains exactly two distinct variables whose exponents differ from $1_{U(L)}$; by Lemma~\ref{lem: UT_3^e1+e2 u_13u_13} and item $(1)$ of Lemma~\ref{lem: UT_3^e1+e2 cons var in the mid}, both must equal $u_2$. We examine  in turn the cases where the number of commutators $b$ is equal to
$2, 1, 0$.

 \smallskip

\emph{Case $b=2$.} By Lemma~\ref{lem: UT_3^e1+e2 u_2 com}, items (3)--(4) of Lemma~\ref{lem: UT_3^e1+e2 cons var in the mid}, and the Leibniz rule for the commutator, every monomial with an occurrence of $u_2$ outside the first entry of $c_1$ or $c_2$ lies in $I$; as in Theorem~\ref{teo: Id^L UT_3 e1-e2}, this forces the two occurrences of $u_2$ into the first entries of $c_1$ and $c_2$, respectively.

\smallskip

\emph{Case $b=1$.} As in the proof of Theorem~\ref{teo: Id^L UT_3 e1-e2}, Lemma~\ref{lem: UT_3^e1+e2 u_2 com}, items~(3)--(4) of Lemma~\ref{lem: UT_3^e1+e2 cons var in the mid}, and the Leibniz rule for the commutator imply that the two occurrences of $u_2$ must be distributed as follows: one occurs in the single commutator, while the other may initially occur either in the tail or in the commutator.
In the former case, the occurrence of $u_2$ in the tail can be rewritten using \eqref{eq: tail-to-commutator}, yielding a linear combination of terms involving two separate commutators, each containing one occurrence of $u_2$. In the latter case, both occurrences of $u_2$ are forced into the first two positions of the commutator by Lemma~\ref{lem: UT_3^e1+e2 u_2 com}. Lemma~\ref{lem: x_1^{u_1}x_2^{u_2}} then decomposes it into two separate commutators, each containing one of the two occurrences of $u_2$.

\smallskip

\emph{Case $b=0$.} As in Theorem~\ref{teo: Id^L UT_3 e1-e2}, applying \eqref{eq: tail-to-commutator} to each of the two tail occurrences in turn yields the family $x_{i_1}\cdots x_{i_r}[x_{k_1}^{u_2},x_{k_2},\ldots,x_{k_m}][x_{j_1}^{u_2},x_{j_2},\ldots,x_{j_{n-r-m}}]$.

\smallskip

In each of the above sub-case, modulo $I$, both occurrences of $u_2$ lie in the first entries of two separate commutators; ordering the remaining entries via Lemma~\ref{lem: ordered comm x_1^u} and fixing $l_1<l_2$, $l_1'<l_2'$ via Lemma~\ref{lem: UT_3^e1+e2 u_2 ordering}, we obtain, modulo $I$, the family
     \begin{equation}\label{eq: no Id u_2 u_2 e1+e2}
           x_{i_{1}}\dots x_{i_{r}}[x_{l_1}^{u_{2}},x_{l_2},\dots,x_{l_{v_1}}][x_{l'_1}^{u_{2}},x_{l'_2},\dots,x_{l'_{v_2}}],
     \end{equation}
where $0\leq r\leq n-2$ and $v_1,v_2 \geq 1$.


Combining the observations above, we conclude that the  following $L$-polynomials generate $P_{n}^L$ modulo $P_{n}^L\cap I$:
\begin{equation} \label{No IdUT_3^e_1+e_2}
\begin{split}
&x_{i_{1}}\dots x_{i_{r}}[x_{j_1},x_{j_2},\dots,x_{j_s}][x_{k_1},x_{k_2},\dots,x_{k_t}],\\
&x_{i_{1}}\dots x_{i_{r}}[x_{h_1}^{u_{2}},x_{h_2},\dots,x_{h_{q_1}}][x_{m_1},x_{m_2},\dots,x_{m_{w_1}}],\\
&x_{i_{1}}\dots x_{i_{r}}[x_{m'_1},x_{m'_2},\dots,x_{m'_{w_2}}][x_{h'_1}^{u_{2}},x_{h'_2},\dots,x_{h'_{q_2}}],\\
&x_{i_{1}}\dots x_{i_{r}}[x_{l_1}^{u_{2}},x_{l_2},\dots,x_{l_{v_1}}][x_{l'_1}^{u_{2}},x_{l'_2},\dots,x_{l'_{v_2}}],\\
&x_{i_{1}}\dots x_{i_{r}}[x_{p_1}^{u_{13}},x_{p_2},\dots,x_{p_z}],
\end{split}
\end{equation}
where $r,s,t, w_1 \geq 0,$ with $s,t,w_1\neq 1,$ $q_1,q_2,v_1,v_2,z\geq 1$ and $w_2\geq 2$ such that $i_1< \cdots < i_r,$ $j_1>j_2<  \cdots < j_s,$ $k_1>k_2<  \cdots < k_t,$ $h_2< \cdots <h_{q_1},$ $m_1> m_2< \cdots < m_{w_1},$ $m'_1> m'_2< \cdots < m'_{w_2},$ $h'_1< \cdots <h'_{q_2},$ $l_1< \cdots < l_{v_1},$ $l'_1< \cdots < l'_{v_2},$  $p_1< \cdots < p_z.$ Additionally, if $w_1\geq 2,$ then $h_1< h_2$.

\textbf{Step 4} is as in the proof of Theorem~\ref{Teo: Id^L UT_3 epsilon1}: 
 since $UT_3$ has a unit element, we may, without loss of generality, disregard the tail $x_{i_{1}}\dots x_{i_{r}}$, so we may suppose that every variable $x_1,\ldots,x_n$ that appears in a polynomial $f\in \I^L(UT_3^{\e_1+\e_2})$ with exponent $1_{U(L)}$ occurs exclusively within a commutator. Thus, we may write
\begin{align*}
    f = &\sum_{\mathcal{J}_{j_1}, \mathcal{K}_{k_1}} \alpha_{ \mathcal{J}_{j_1}, \mathcal{K}_{k_1}} [x_{j_1},x_{j_2},\dots,x_{j_s}][x_{k_1},x_{k_2},\dots,x_{k_t}]\\
    &+ \sum_{ \mathcal{H}_{h_1}, \mathcal{M}_{m_1}}  \beta_{\mathcal{H}_{h_1}, \mathcal{M}_{m_1}} [x_{h_1}^{u_{2}},x_{h_2},\dots,x_{h_{q_1}}][x_{m_1},x_{m_2},\dots,x_{m_{w_1}}]\\
    &+ \sum_{ \mathcal{M}'_{m'_1}, \mathcal{H}'_{h'_1}}  \gamma_{ \mathcal{M}'_{m'_1}, \mathcal{H}'_{h'_1}}[x_{m'_1},x_{m'_2},\dots,x_{m'_{w_2}}][x_{h'_1}^{u_{2}},x_{h'_2},\dots,x_{h'_{q_2}}]\\
    &+\, \sum_{\mathcal{L}, \mathcal{L}'}  \lambda_{\mathcal{L}, \mathcal{L}'}[x_{l_1}^{u_{2}},x_{l_2},\dots,x_{l_{v_1}}][x_{l'_1}^{u_{2}},x_{l'_2},\dots,x_{l'_{v_2}}] + \mu [x_{1}^{u_{13}},x_{2},\dots,x_{n}],
\end{align*}
where $\mathcal{J}_{j_1}=\{j_1,\ldots,j_s\},$ $\mathcal{K}_{k_1}=\{k_1,\ldots,k_t\},$ $\mathcal{H}_{h_1}=\{h_1,\ldots,h_{q_1}\},$ $\mathcal{M}_{m_1}=\{m_1,\ldots,m_{w_1}\},$ $\mathcal{H}'_{h'_1}=\{h'_1,\ldots,h'_{q_2}\},$ $\mathcal{M}'_{m'_1}=\{m'_1,\ldots,m'_{w_2}\},$ $\mathcal{L}=\{l_1,\ldots,l_{v_1}\}$ and $\mathcal{L}'=\{l'_1,\ldots,l'_{v_2}\}$
 are subsets of $\{1,\ldots, n\}$ satisfying $\mathcal{J}_{j_1}\cap \mathcal{K}_{k_1}=\mathcal{H}_{h_1}\cap \mathcal{M}_{m_1}=\mathcal{H}'_{h'_1}\cap \mathcal{M}'_{m'_1}=\mathcal{L}\cap \mathcal{L}'=\emptyset$. Moreover, the indices fulfill the conditions
$s,w_2\geq 2,$ $t, w_1\geq 0,$ $t,w_1\neq 1,$ $q_1,q_2,v_1,v_2\geq 1,$ $s+t=q_1+w_1=w_2+q_2=v_1+v_2=n,$ together with the ordering constraints $j_1>j_2<  \cdots < j_s,$ $k_1>k_2<  \cdots < k_t,$ $m_1> m_2< \cdots < m_{w_1},$ $h_2< \cdots <h_{q_1},$ $m'_1> m'_2< \cdots < m'_{w_2},$ $h'_1< \cdots <h'_{q_2},$ $l_1< \cdots < l_{v_1},$ $l'_1< \cdots < l'_{v_2}$. In addition, if $w_1\geq 2$ then $h_1< h_2$.

\textbf{Step 5.} We show that the $L$-polynomials appearing in the decomposition of $f$ above are
linearly independent modulo $\I^L(UT_3^{\e_1+\e_2})$.

Let $\mathcal{J}_{j_1}=\{j_1,\ldots,j_n\}$ be a fixed set such that $j_1>j_2<  \cdots < j_n.$ If we evaluate $ x_{j_1}= e_{13}$ and  $x_{j_2}=\cdots = x_{j_n}= e_{33},$ then we get $\alpha_{ \mathcal{J}_{j_1}, \emptyset}  e_{13}=0$ since $j_1\neq 1$ and $e_{33}^{\nu_{13}}=0$. So, it follows that $ \alpha_{ \mathcal{J}_{j_1}, \emptyset}=0$.  
Moreover, evaluating $x_1=e_{13}$ and $x_2=\cdots=x_n=e_{33}$, we get $\mu  e_{13}=0$, and hence $\mu=0$.

Now, let $\mathcal{H}_{h_1}=\{h_1,\ldots,h_{n}\}$ be a fixed set such that $h_2< \cdots <h_n$. By evaluating $ x_{h_1}= e_{12}$ and  $x_{h_2}=\cdots = x_{h_n}= e_{22}$, we get $\beta_{ \mathcal{H}_{h_1}, \emptyset}  e_{12}=0,$ so  $\beta_{ \mathcal{H}_{h_1}, \emptyset}=0.$

Next, let $\mathcal{J}_{j_1}=\{j_1,\ldots,j_s\}$ and $\mathcal{K}_{k_1}=\{k_1,\ldots,k_t\}$ be two fixed disjoint subsets of $\{1,\ldots,n\}$ such that $s,t\geq 2,$ $s+t=n$ and $j_1>j_2<  \cdots < j_s,$ $k_1>k_2<  \cdots < k_t.$ If we evaluate
$ x_{j_1}= e_{12},$  $x_{j_2}=\cdots = x_{j_s}= e_{11},$ $ x_{k_1}= e_{23}$ and  $x_{k_2}=\cdots = x_{k_t}= e_{33},$ then we get $\alpha_{ \mathcal{J}_{j_1}, \mathcal{K}_{k_1}} e_{13}=0,$ thus $\alpha_{ \mathcal{J}_{j_1}, \mathcal{K}_{k_1}}=0.$
Here the polynomials  $[x_{h_1}^{u_{2}},x_{h_2},\dots,x_{h_{s}}][x_{k_1},x_{k_2},\dots,x_{k_{t}}]$  evaluate to zero since $j_1>j_2<  \cdots < j_s$ whereas $h_1<h_2<\cdots <h_{s}$ and $e_{11}^{\nu_{2}}=0.$ Similarly, the polynomials  $[x_{j_1},x_{j_2},\dots,x_{j_s}][x_{h'_1}^{u_{2}},x_{h'_2},\dots,x_{h'_{t}}]$ evaluate to zero since  $k_1>k_2<  \cdots < k_t$ whereas $h'_1<h'_2<\cdots <h'_t$ and $e_{33}^{\nu_{2}}=0.$
Also, with the same reasoning all the polynomials of the type $[x_{l_1}^{u_{2}},x_{l_2},\dots,x_{l_{v_1}}][x_{l'_1}^{u_{2}},x_{l'_2},\dots,x_{l'_{v_2}}]$ evaluate to zero.

Let $\mathcal{H}_{h_1}=\{h_1,\ldots,h_{q_1}\}$ and  $\mathcal{M}_{m_1}=\{m_1,\ldots,m_{w_1}\}$ be two fixed disjoint subsets of $\{1,\ldots,n\}$ with $q_1\geq 1,$ $w_1\geq 2,$ $q_1+w_1=n$ and subject to the conditions $h_1<h_2<  \cdots < h_{q_1},$ $m_1>m_2<  \cdots < m_{w_1}.$ If we make the evaluation $x_{h_1}=e_{12},$ $x_{h_2}=\cdots = x_{h_{q_1}}= e_{11},$ $x_{m_1}=e_{23}$ and $x_{m_2}= \cdots = x_{m_{w_1}}= e_{33},$ we get $\beta_{\mathcal{H}_{h_1}, \mathcal{M}_{m_1}}e_{13}=0,$ that is $\beta_{\mathcal{H}_{h_1}, \mathcal{M}_{m_1}}=0.$ Notice that, as above, with this evaluation all the polynomials of the type $[x_{m'_1},x_{m'_2},\dots,x_{m'_{w_2}}][x_{h'_1}^{u_{2}},x_{h'_2},\dots,x_{h'_{q_2}}]$ and of the type $[x_{l_1}^{u_{2}},x_{l_2},\dots,x_{l_{v_1}}][x_{l'_1}^{u_{2}},x_{l'_2},\dots,x_{l'_{v_2}}]$ evaluate to zero.
Similarly, if we evaluate $x_{h_1}=e_{23},$ $x_{h_2}=\cdots = x_{h_{q_1}}= e_{33},$ $x_{m_1}=e_{12}$ and $x_{m_2}= \cdots = x_{m_{w_1}}= e_{11},$ we get $\gamma_{\mathcal{M}_{m_1}, \mathcal{H}_{h_1}}e_{13}=0,$ and so $\gamma_{\mathcal{M}_{m_1}, \mathcal{H}_{h_1}}=0.$

Finally, for fixed $\mathcal{L}=\{l_1,\ldots,l_{v_1}\}$ and $\mathcal{L}'=\{l'_1,\ldots,l'_{v_2}\}$ disjoint subsets of $\{1,\ldots,n\}$ such that $v_1,v_2\geq 1,$ $v_1+v_2=n$ and $l_1< \cdots < l_{v_1},$ $l'_1<  \cdots < l'_{v_2},$ then from the evaluation  $x_{l_1}=e_{12},$ $x_{l_2}=\cdots = x_{l_{v_1}}= e_{11},$ $x_{l'_1}=e_{23}$ and $x_{l'_2}= \cdots = x_{l'_{v_2}}= e_{33},$ it follows that $\lambda_{\mathcal{L}, \mathcal{L}'}e_{13}=0$ and consequently $\lambda_{\mathcal{L}, \mathcal{L}'}=0.$ 

Therefore, all the scalars appearing in $f$ are zero, and as a consequence the $L$-polynomials \eqref{No IdUT_3^e_1+e_2} are linearly independent modulo $\I^L(UT_3^{\e_1+\e_2})$. Since $P_n^L\cap I\subseteq P_n^L\cap \I^L(UT_3^{\e_1+\e_2}),$ it follows that $\I^L(UT_3^{\e_1+\e_2})= I$. Consequently, the $L$-polynomials \eqref{No IdUT_3^e_1+e_2} form a basis of $P_n^L$ modulo $P_n^L\cap \I^L(UT_3^{\e_1+\e_2}).$ 

Finally, by counting these $L$-polynomials, we determine the $L$-codimensions of $UT_3^{\e_1+\e_2}$.  Hence, by Theorem \ref{teo: Id^L UT_3 e1-e2}, we get
\begin{align*}
     c_n^L(UT_3^{\e_1+\e_2})=&c_n^L(UT_3^{\e_1 - \e_2})-\sum_{m=1}^n \binom{n}{m} m + \sum_{r=0}^{n-1} \binom{n}{r}
=(n^2-n)3^{n-2}+n2^n+1,
\end{align*}
as claimed.
\end{proof}

Finally, consider the case where $\beta \neq 0, \pm 1$. In this situation, $L$ acts on $UT_3$ as the one-dimensional Lie algebra generated by $\delta:=\e_1+\beta\e_2.$ Moreover, defining $$\nu_{12}:=\frac{\delta^3- (2\beta+1)\delta^2+\beta(\beta+1)\delta}{\beta(\beta-1)}, \ \nu_{23}:= \frac{-\delta^3+(\beta+2)\delta^2-(\beta+1)\delta}{\beta(\beta-1)}, \ \nu_{13}:=\frac{\delta^3-(\beta+1)\delta^2+\beta \delta}{\beta(\beta+1)},$$
we obtain
\begin{align}
    \label{eq: nu_12 beta}
    &(e_{12})^{\nu_{12}}=e_{12},  \quad (e_{11})^{\nu_{12}} = (e_{22})^{\nu_{12}} = (e_{33})^{\nu_{12}}= (e_{13})^{\nu_{12}} =(e_{23})^{\nu_{12}}=0,\\
    \label{eq: nu_23 beta}
    &(e_{23})^{\nu_{23}}=e_{23},  \quad (e_{11})^{\nu_{23}} = (e_{22})^{\nu_{23}} = (e_{33})^{\nu_{23}}= (e_{12})^{\nu_{23}} =(e_{13})^{\nu_{23}}=0,\\
    \label{eq: nu_13 beta}
   &(e_{13})^{\nu_{13}}=e_{13}, \quad  (e_{11})^{\nu_{13}} = (e_{22})^{\nu_{13}} = (e_{33})^{\nu_{13}}= (e_{12})^{\nu_{13}}=(e_{23})^{\nu_{13}} =0.
\end{align}

We now proceed to determine the $T_L$-ideal of differential identities of the $L$-algebra $UT_3^{\e_1+\beta\e_2}.$ To this end, we first establish several technical lemmas.

\begin{lemma}\label{lem: u_12 u_23 u_13 on x_1x_2}
    Let $d\in L$ and $\beta\in F$ such that $\beta\neq 0, \pm 1$. If we denote $u_{12}:=\frac{d^3- (2\beta+1)d^2+\beta(\beta+1)d}{\beta(\beta-1)}$, $u_{23}:=\frac{-d^3+(\beta+2)d^2-(\beta+1)d}{\beta(\beta-1)}$ and $u_{13}:=\frac{d^3-(\beta+1)d^2+\beta d}{\beta(\beta+1)}$, then
    \begin{align*}
        &(x_1 x_2)^{u_{12}}- x_1^{u_{12}} x_2 - x_1x_2^{u_{12}} +x_1^{u_{12}} x_2^{u_{23}} ,
        \\
        & (x_1 x_2)^{u_{23}}- x_1^{u_{23}} x_2 - x_1x_2^{u_{23}} +x_1^{u_{12}} x_2^{u_{23}}, \\
        &(x_1 x_2)^{u_{13}}- x_1^{u_{13}} x_2 - x_1x_2^{u_{13}}- x_1^{u_{12}} x_2^{u_{23}},
    \end{align*}
are consequences of the $L$-polynomials in 
$$\mathcal{S}_1:=\left\{ x_1^{u_{ij}}x_2^{u_{kl}} \ : \ j\neq k \right\}.$$
\end{lemma}
\begin{proof}
Since $d= u_{12}+\beta u_{23}+(\beta +1)u_{13}$ and $d^2= u_{12}+\beta^2 u_{23}+(\beta +1)^2 u_{13}$, it follows that 
$$
x_1^{d}x_2^{d}-\beta x_1^{u_{12}} x_2^{u_{23}}, \  x_1^{d^2}x_2^{d}-\beta x_1^{u_{12}} x_2^{u_{23}}, \  x_1^{d}x_2^{d^2}-\beta^2 x_1^{u_{12}} x_2^{u_{23}}\in \langle \mathcal{S}_1\rangle_{T_L}.
$$

As a consequence, since $u_{12}=\frac{d^3- (2\beta+1)d^2+\beta(\beta+1)d}{\beta(\beta-1)}$ and by applying the Leibniz rule, we obtain
\begin{align*}
    (x_1 x_2)^{u_{12}}&= x_1^{u_{12}} x_2 + x_1x_2^{u_{12}}+\frac{3}{\beta(\beta-1)}(x_1^{d^2}x_2^{d}+x_1^dx_2^{d^2})- \frac{4\beta+2}{\beta(\beta-1)}x_1^dx_2^d\\
    &\equiv x_1^{u_{12}} x_2 + x_1x_2^{u_{12}}- x_1^{u_{12}} x_2^{u_{23}} \pmod{\langle \mathcal{S}_1  \rangle_{T_L}}.
\end{align*}

Similarly, we can prove the others and complete the proof.
\end{proof}

As a direct consequence of the above lemma, we have the following result.

\begin{lemma}\label{lem: UT_3^e1+be2 cons var in the mid}
     Let $d\in L$ and $\beta\in F$ such that $\beta\neq 0, \pm 1$. If we denote $u_{12}:=\frac{d^3- (2\beta+1)d^2+\beta(\beta+1)d}{\beta(\beta-1)}$, $u_{23}:=\frac{-d^3+(\beta+2)d^2-(\beta+1)d}{\beta(\beta-1)}$ and $u_{13}:=\frac{d^3-(\beta+1)d^2+\beta d}{\beta(\beta+1)}$, then
     $$
     x_1^{u_{ij}} x_2 x_3^{u_{kl}}, \quad  j\neq k,
     $$ 
     are consequences of the $L$-polynomials in
$$\mathcal{S}_1:=\left\{  x_1^{u_{ij}}x_2^{u_{kl}}\ : \ j\neq k \right\}.$$
\end{lemma}

By straightforward computation, we have the following.
\begin{lemma}\label{lem:  UT_3^e1+be2 u_ij com}
     Let $d\in L$ and $\beta\in F$ such that $\beta\neq 0, \pm 1$. If we denote $u_{12}:=\frac{d^3- (2\beta+1)d^2+\beta(\beta+1)d}{\beta(\beta-1)}$, $u_{23}:=\frac{-d^3+(\beta+2)d^2-(\beta+1)d}{\beta(\beta-1)}$ and $u_{13}:=\frac{d^3-(\beta+1)d^2+\beta d}{\beta(\beta+1)}$, then:
     \begin{enumerate}
         \item $ x_1^{u_{13}} [x_2,x_3] \in \langle x_1^{u_{13}} x_2^{u_{13}} , \, x_1^{u_{13}} x_2^{u_{12}} , \, x_1^{u_{13}} x_2^{u_{23}}, \,  [x_1,x_2]^{u_{13}}+ [x_1,x_2]^{u_{12}} + [x_1,x_2]^{u_{23}} - [x_1,x_2] \rangle_{T_L}$;
          \vspace{1mm}
         \item $[x_1,x_2] x_3^{u_{13}} \in \langle x_1^{u_{13}} x_2^{u_{13}} , \, x_1^{u_{12}} x_2^{u_{13}} , \, x_1^{u_{23}} x_2^{u_{13}}, \,  [x_1,x_2]^{u_{13}}+ [x_1,x_2]^{u_{12}} + [x_1,x_2]^{u_{23}} - [x_1,x_2] \rangle_{T_L}$;
         \vspace{1mm}
         \item $[x_1,x_2] x_3^{u_{12}} \in \langle x_1^{u_{13}} x_2^{u_{12}} , \, x_1^{u_{12}} x_2^{u_{12}} , \, x_1^{u_{23}} x_2^{u_{12}}, \,  [x_1,x_2]^{u_{13}}+ [x_1,x_2]^{u_{12}} + [x_1,x_2]^{u_{23}} - [x_1,x_2] \rangle_{T_L}$;
         \vspace{1mm}
         \item $ x_1^{u_{23}} [x_2,x_3] \in \langle x_1^{u_{23}} x_2^{u_{13}} , \, x_1^{u_{23}} x_2^{u_{12}} , \, x_1^{u_{23}} x_2^{u_{23}}, \,  [x_1,x_2]^{u_{13}}+ [x_1,x_2]^{u_{12}} + [x_1,x_2]^{u_{23}} - [x_1,x_2] \rangle_{T_L}$;
          \vspace{1mm}
         \item $x_1^{u_{12}} [x_2,x_3] x_4^{u_{23}} \in \langle x_1^{u_{12}} x_2^{u_{13}} , \, x_1^{u_{12}} x_2^{u_{12}} , \, x_1^{u_{23}} x_2^{u_{23}}, \,  [x_1,x_2]^{u_{13}}+ [x_1,x_2]^{u_{12}} + [x_1,x_2]^{u_{23}} - [x_1,x_2] \rangle_{T_L}$.
     \end{enumerate}
\end{lemma}

\begin{lemma}\label{lem: UT_3^e1+be2 3comm}
      Let $d\in L$ and $\beta\in F$ such that $\beta\neq 0, \pm 1$. If we denote $u_{12}:=\frac{d^3- (2\beta+1)d^2+\beta(\beta+1)d}{\beta(\beta-1)}$, $u_{23}:=\frac{-d^3+(\beta+2)d^2-(\beta+1)d}{\beta(\beta-1)}$ and $u_{13}:=\frac{d^3-(\beta+1)d^2+\beta d}{\beta(\beta+1)}$, then
      $$[x_1,x_2][x_3,x_4][x_5,x_6]$$
      is a consequence of the $L$-polynomials in
      $$\mathcal{S}_2:=\left\{ [x_1,x_2]^{u_{13}}+ [x_1,x_2]^{u_{12}} + [x_1,x_2]^{u_{23}} - [x_1,x_2],\ x_1^{u_{ij}}x_2^{u_{kl}}\ : \  j\neq k \right\}.$$
\end{lemma}
\begin{proof}
From items (1) and (4) of Lemma \ref{lem: UT_3^e1+be2 u_ij com}, we deduce that $\big([x_1,x_2]^{u_{12}}-[x_1,x_2]\big)[x_3,x_4]\in \langle \mathcal{S}_2\rangle_{T_L}$. Consequently, applying item (3) of Lemma \ref{lem: UT_3^e1+be2 u_ij com} together with Lemma \ref{lem: conseq 3 commutators 1}, we arrive at the desired conclusion.
\end{proof}

\begin{lemma}\label{lem: UT_3^e1+be2 u_12 u_23 ordering 1}
    Let $d\in L$ and $\beta\in F$ such that $\beta\neq 0, \pm 1$. If we denote $u_{12}:=\frac{d^3- (2\beta+1)d^2+\beta(\beta+1)d}{\beta(\beta-1)}$, $u_{23}:=\frac{-d^3+(\beta+2)d^2-(\beta+1)d}{\beta(\beta-1)}$ and $u_{13}:=\frac{d^3-(\beta+1)d^2+\beta d}{\beta(\beta+1)}$, then
    \begin{align*}
        &[x_1^{u_{12}},x_2] [x_3,x_4] - [x_2^{u_{12}},x_1] [x_3,x_4]-[x_1,x_2][x_3,x_4],\\
        & [x_1,x_2] [x_3^{u_{23}},x_4] - [x_1,x_2] [x_4^{u_{23}},x_3]-[x_1,x_2][x_3,x_4],
    \end{align*}
are consequences of the $L$-polynomials in
$$\mathcal{S}_2:=\left\{ [x_1,x_2]^{u_{13}}+ [x_1,x_2]^{u_{12}} + [x_1,x_2]^{u_{23}} - [x_1,x_2],\ x_1^{u_{ij}}x_2^{u_{kl}}\ : \ j\neq k \right\}.$$
\end{lemma}
\begin{proof}
    From Lemma \ref{lem: u_12 u_23 u_13 on x_1x_2} we obtain
$$
[x_1,x_2]^{u_{12}}\equiv [x_1^{u_{12}},x_2] - [x_2^{u_{12}},x_1]-x_1^{u_{12}} x_2^{u_{23}} + x_2^{u_{12}} x_1^{u_{23}} \pmod{\langle \mathcal{S}_2 \rangle_{T_L}}.
$$
and
$$
[x_1,x_2]^{u_{23}}\equiv [x_1^{u_{23}},x_2] - [x_2^{u_{23}},x_1]-x_1^{u_{12}} x_2^{u_{23}} + x_2^{u_{12}} x_1^{u_{23}} \pmod{\langle \mathcal{S}_2 \rangle_{T_L}}.
$$
Consequently, by items (1)--(4) of Lemma \ref{lem:  UT_3^e1+be2 u_ij com}, we obtain the desired conclusion.
\end{proof}

Similarly, as a direct consequence of Lemma \ref{lem: u_12 u_23 u_13 on x_1x_2}, we also obtain the following two results.

\begin{lemma}\label{lem: UT_3^e1+be2 u_12 u_23 ordering 2}
    Let $d\in L$ and $\beta\in F$ such that $\beta\neq 0, \pm 1$. If we denote $u_{12}:=\frac{d^3- (2\beta+1)d^2+\beta(\beta+1)d}{\beta(\beta-1)}$, $u_{23}:=\frac{-d^3+(\beta+2)d^2-(\beta+1)d}{\beta(\beta-1)}$ and $u_{13}:=\frac{d^3-(\beta+1)d^2+\beta d}{\beta(\beta+1)}$, then
    \begin{align*}
        & [x_1^{u_{12}},x_2] x_3^{u_{23}} - [x_2^{u_{12}},x_1] x_3^{u_{23}}-[x_1,x_2]x_3^{u_{23}},\\
        & x_1^{u_{12}} [x_2^{u_{23}},x_3] - x_1^{u_{12}} [x_3^{u_{23}},x_2]-x_1^{u_{12}}[x_2,x_3],
    \end{align*}
are consequences of the $L$-polynomials in
$$\mathcal{S}_2:=\left\{ [x_1,x_2]^{u_{13}}+ [x_1,x_2]^{u_{12}} + [x_1,x_2]^{u_{23}} - [x_1,x_2],\ x_1^{u_{ij}}x_2^{u_{kl}}\ : \  j\neq k \right\}.$$
\end{lemma}

\begin{lemma}\label{lem: UT_3^e1+be2 u_13 ordering}
    Let $d\in L$ and $\beta\in F$ such that $\beta\neq 0, \pm 1$. If we denote $u_{12}:=\frac{d^3- (2\beta+1)d^2+\beta(\beta+1)d}{\beta(\beta-1)}$, $u_{23}:=\frac{-d^3+(\beta+2)d^2-(\beta+1)d}{\beta(\beta-1)}$ and $u_{13}:=\frac{d^3-(\beta+1)d^2+\beta d}{\beta(\beta+1)}$, then
    $$
    [x_1^{u_{13}}, x_2]- [x_2^{u_{13}}, x_1] +[x_1^{u_{12}}, x_2]- [x_2^{u_{12}}, x_1] +[x_1^{u_{23}}, x_2]- [x_2^{u_{23}}, x_1]  - [x_1, x_2] -x_1^{u_{12}}x_2^{u_{23}} + x_2^{u_{12}}x_1^{u_{23}}
    $$
is a consequence of the $L$-polynomials in
$$\mathcal{S}_2:=\left\{ [x_1,x_2]^{u_{13}}+ [x_1,x_2]^{u_{12}} + [x_1,x_2]^{u_{23}} - [x_1,x_2],\ x_1^{u_{ij}}x_2^{u_{kl}}\ : \  j\neq k \right\}.$$
\end{lemma}

Once again, by adopting an approach analogous to that used in Theorems \ref{Teo: Id^L UT_3 epsilon1}, \ref{teo: Id^L UT_3 e1-e2}, and \ref{teo: Id^L UT_3 e1+e2}, and performing some additional computations, we determine the $T_L$-ideal of differential identities of $UT_3^{\e_1+\beta \e_2}$.

\begin{theorem} \label{teo: Id^L UT_3 e1+betae2}
   Let $L$ be a Lie algebra and $UT_3^{\e_1+\beta\e_2}$, $\beta\neq 0,\pm 1,$  be the $L$-algebra $UT_3$ where $L$ acts via the one-dimensional Lie subalgebra of $\D(UT_3)$ spanned by the inner derivation  $\e_1+\beta\e_2 =\ad_{(-e_{11}+\beta e_{33})}$. Then 
there exists a basis $\mathcal{B}_L=\{d_i \mid i \geq 1\}$ of $L$ over $F$ such that 
   the $T_L$-ideal of $L$-identities of $UT_3^{\e_1+\beta\e_2}$ is generated by the following $L$-polynomials
    \begin{align*}
       & x^{d_i}, \quad  x^{d_1^4}- 2(\beta+1)x^{d_1^3}+(\beta^2 + 3 \beta +1) x^{d_1^2}- \beta(\beta+1) x^{d_1},\\
     &  x_1^{u_{jk}}x_2^{u_{lm}}, \quad  [x_1,x_2]^{u_{13}}+ [x_1,x_2]^{u_{12}} + [x_1,x_2]^{u_{23}} - [x_1,x_2],
    \end{align*}
for all $i\geq 2,$ $1\leq j<k \leq 3$, $1\leq l < m\leq 3$, $k\neq l$,
where
$u_{12}=\frac{d_1^3- (2\beta+1)d_1^2+\beta(\beta+1)d_1}{\beta(\beta-1)},$ $ u_{23}=\frac{-d_1^3+(\beta+2)d_1^2-(\beta+1)d_1}{\beta(\beta-1)}$ and  $u_{13}=\frac{d_1^3-(\beta+1)d_1^2+\beta d_1}{\beta(\beta+1)}.$ 

Moreover, $c_n^L(UT_3^{\e_1 +\beta\e_2})= (n^2-n)3^{n-2}+3n2^{n-1}+1.$
\end{theorem}
\begin{proof}
\textbf{Step 1} is as in Remark~\ref{rmk: general-strategy}, with $\theta(d_1)=\e_1+\beta\e_2$.

\textbf{Step 2} is immediate from \eqref{eq: nu_12 beta}, \eqref{eq: nu_23 beta} and \eqref{eq: nu_13 beta}.

\textbf{Step 3.} Let $f\in P_n^L$. Since $x^{d_i}\in I$ for $i\geq 2$ and
$x^{d_1^4}- 2(\beta+1)x^{d_1^3}+(\beta^2 + 3 \beta +1) x^{d_1^2}- \beta(\beta+1) x^{d_1}\in I$, the
exponents occurring modulo $I$ reduce to $\{1_{U(L)}, d_1, d_1^2, d_1^3\}$; hence, since
$\beta \neq 0, \pm 1$, without loss of generality we may assume that they lie in
$\{1_{U(L)}, u_{12},u_{23}, u_{13}\}.$
By Lemma~\ref{lem: UT_3^e1+be2 3comm}, at most two left-normed commutators occur; and, since
$x_1^{u_{ij}}x_2^{u_{kl}}\in I$ and, by Lemma~\ref{lem: UT_3^e1+be2 cons var in the mid}, also
$x_1^{u_{ij}}x_2 x_3^{u_{kl}} \in I$ for $j\neq k$, a monomial may involve at most two distinct
variables whose exponent differs from $1_{U(L)}$.

If no variable occurs with an exponent different from $1_{U(L)}$, then, as in the previous theorems, the argument is exactly as in Theorem~\ref{Teo: Id^L UT_3 ordinarie}, giving the ordinary family
\begin{equation*}\label{eq: no id ord e1+be2}
    x_{i_{1}}\dots x_{i_{r}}[x_{j_1},x_{j_2},\dots,x_{j_s}][x_{k_1},x_{k_2},\dots,x_{k_t}],
\end{equation*}
where $ r, s, t\geq 0$ with $s,t\neq 1$, and the indices satisfy  $i_1< \cdots < i_r$, $j_1>j_2<  \cdots < j_s$,  $k_1>k_2<  \cdots < k_t$.

Now suppose that exactly  one variable appears with an exponent different from $1_{U(L)}$, namely either $u_{12}$ or $u_{23}$ or $u_{13}$.

Assume first that this exponent is $u_{13}$. 
Following the argument as in the proof of Theorem~\ref{teo: Id^L UT_3 e1+e2}: by Lemma~\ref{lem: UT_3^e1+be2 u_ij com} and items (1)--(2) of Remark~\ref{rmk: leibniz rule comm}, the variable must occur, modulo $I$, either in the tail or in the first entry of the (unique) commutator, and \eqref{eq: tail-to-commutator} (applied to $u_{13}$) moves a tail occurrence into a commutator.

 Assume now that exactly one variable occurs with exponent $u_{12}$ (resp.\ $u_{23}$). Following the argument in the proof of Theorem~\ref{teo: Id^L UT_3 e1-e2} and using item~(3) (resp.\ item~(4)) of Lemma~\ref{lem: UT_3^e1+be2 u_ij com} instead of the corresponding lemmas, we may assume, modulo $I$, that this variable occurs in the first entry of the first commutator (resp.\ in the first entry of either the only commutator or the second commutator).

Therefore, this gives the following families
     \begin{align}
&x_{i_{1}}\dots x_{i_{r}}[x_{h_1}^{u_{12}},x_{h_2},\dots,x_{h_{q_1}}][x_{m_1},x_{m_2},\dots,x_{m_{w_1}}],\label{eq: no Id u_12 beta}\\
&x_{i_{1}}\dots x_{i_{r}}[x_{m'_1},x_{m'_2},\dots,x_{m'_{w_2}}][x_{h'_1}^{u_{23}},x_{h'_2},\dots,x_{h'_{q_2}}], \label{eq: no Id u_23 beta}\\
&x_{i_{1}}\dots x_{i_{r}}[x_{p_1}^{u_{13}},x_{p_2},\dots,x_{p_z}], \label{eq: no Id u_13 beta} 
\end{align}
where $r, w_1, w_2 \geq 0,$ $w_1,w_2\neq 1,$, $q_1,q_2,z\geq 1$ and  $i_1< \cdots < i_r$. 
By Lemma~\ref{lem: ordered comm} and the Jacobi identity, we may assume, modulo $I$, that $m_1>m_2< \cdots < m_{w_1}$ in \eqref{eq: no Id u_12 beta} and $m'_1>m'_2< \cdots < m'_{w_2}$ in \eqref{eq: no Id u_23 beta}.
Moreover, by Lemma~\ref{lem: ordered comm x_1^u}, we may assume, modulo $I$, that $h_2< \cdots <h_{q_1}$, $h'_2< \cdots <h'_{q_2}$ and $p_2<\cdots <p_z$ in \eqref{eq: no Id u_12 beta}, \eqref{eq: no Id u_23 beta} and \eqref{eq: no Id u_13 beta}, respectively.  In addition, by Lemma \ref{lem: UT_3^e1+be2 u_12 u_23 ordering 1}, we may further assume, modulo $I$, that if $w_1\geq 2$ in \eqref{eq: no Id u_12 beta}, then $h_1<h_2$; similarly, if $w_2\geq 2$ in \eqref{eq: no Id u_23 beta}, then $h'_1<h'_2.$ Finally, by Lemma \ref{lem: UT_3^e1+be2 u_13 ordering}, we may also assume, modulo $I$, that $p_1<p_2$ in \eqref{eq: no Id u_13 beta}.

Finally, suppose that an $L$-polynomial $g$ contains exactly two distinct variables whose exponents differ from $1_{U(L)}$. Since $x_1^{u_{ij}}x_2^{u_{kl}} \in I$ for $j\neq k$ and, by Lemma~\ref{lem: UT_3^e1+be2 cons var in the mid}, also $x_1^{u_{ij}}x_2 x_3^{u_{kl}} \in I$ for $j\neq k$, it follows that, modulo $I$, the only possibility is that the exponents different from $1_{U(L)}$ are $u_{12}$ and $u_{23}$, with the first variable carrying exponent $u_{12}$ and the second carrying exponent $u_{23}$.

Therefore, taking into account items (3)–(5) of Lemma \ref{lem: UT_3^e1+be2 u_ij com} together with Remark \ref{rmk: leibniz rule comm}, and proceeding exactly as in Theorem \ref{teo: Id^L UT_3 e1-e2}, we deduce that if a polynomial  contains exactly two variables whose exponents differ from $1_{U(L)}$, then, modulo $I$, it can be written as a linear combination of polynomials of the type
     \begin{equation}\label{eq: no Id u_12 u_23 beta}
           x_{i_{1}}\dots x_{i_{r}}[x_{l_1}^{u_{12}},x_{l_2},\dots,x_{l_{v_1}}][x_{l'_1}^{u_{23}},x_{l'_2},\dots,x_{l'_{v_2}}],
     \end{equation}
where $0\leq r\leq n-2$, $ v_1,v_2 \geq 1$ and $i_1<\cdots < i_r$. Moreover, by Lemma~\ref{lem: ordered comm x_1^u}, we may assume, modulo $I$, that $l_2< \cdots <l_{v_1}$ and $l'_2< \cdots <l'_{v_2}$ in \eqref{eq: no Id u_12 u_23 beta}. In addition, by Lemma \ref{lem: UT_3^e1+be2 u_12 u_23 ordering 2}, we may further assume, modulo $I$, that $l_1<l_2$  and $l'_1<l'_2$ in  \eqref{eq: no Id u_12 u_23 beta}. 


By combining the arguments above, we conclude that the following $L$-polynomials generate $P_{n}^L$ modulo $P_{n}^L\cap I$:
\begin{equation}\label{No IdUT_3^e_1+be_2}
\begin{split}
&x_{i_{1}}\dots x_{i_{r}}[x_{j_1},x_{j_2},\dots,x_{j_s}][x_{k_1},x_{k_2},\dots,x_{k_t}],\\
&x_{i_{1}}\dots x_{i_{r}}[x_{h_1}^{u_{12}},x_{h_2},\dots,x_{h_{q_1}}][x_{m_1},x_{m_2},\dots,x_{m_{w_1}}],\\
&x_{i_{1}}\dots x_{i_{r}}[x_{m'_1},x_{m'_2},\dots,x_{m'_{w_2}}][x_{h'_1}^{u_{23}},x_{h'_2},\dots,x_{h'_{q_2}}],\\
&x_{i_{1}}\dots x_{i_{r}}[x_{l_1}^{u_{12}},x_{l_2},\dots,x_{l_{v_1}}][x_{l'_1}^{u_{23}},x_{l'_2},\dots,x_{l'_{v_2}}],\\
&x_{i_{1}}\dots x_{i_{r}}[x_{p_1}^{u_{13}},x_{p_2},\dots,x_{p_z}],
\end{split}
\end{equation}
where $r,s,t, w_1, w_2 \geq 0,$ with $s,t,w_1,w_2\neq 1,$ and $q_1,q_2,v_1,v_2,z\geq 1$ such that $i_1< \cdots < i_r,$ $j_1>j_2<  \cdots < j_s,$ $k_1>k_2<  \cdots < k_t,$ 
$m_1> m_2< \cdots < m_{w_1},$ $h_2< \cdots <h_{q_1},$  $m'_1> m'_2< \cdots < m'_{w_2},$  $h'_2< \cdots <h'_{q_2},$
$l_1< \cdots < l_{v_1},$ $l'_1< \cdots < l'_{v_2},$ $p_1< \cdots < p_z.$ Additionally, if $w_1\geq 2,$ then $h_1< h_2$, and  if $w_2\geq 2,$ then $h'_1< h'_2.$ 


\textbf{Step 4} is as in Remark~\ref{rmk: general-strategy}: since $UT_3$ has a unit element, we may,
without loss of generality, disregard the tail $x_{i_{1}}\dots x_{i_{r}}$. Hence, if $f \in I^L(UT_3^{\e_1+\beta\e_2}) $, we may write
\begin{align*}
    f = &\sum_{\mathcal{J}_{j_1}, \mathcal{K}_{k_1}} \alpha_{ \mathcal{J}_{j_1}, \mathcal{K}_{k_1}} [x_{j_1},x_{j_2},\dots,x_{j_s}][x_{k_1},x_{k_2},\dots,x_{k_t}]\\
    &+ \sum_{ \mathcal{H}_{h_1}, \mathcal{M}_{m_1}}  \beta_{\mathcal{H}_{h_1}, \mathcal{M}_{m_1}} [x_{h_1}^{u_{12}},x_{h_2},\dots,x_{h_{q_1}}][x_{m_1},x_{m_2},\dots,x_{m_{w_1}}]\\
    &+ \sum_{ \mathcal{M}'_{m'_1}, \mathcal{H}'_{h'_1}}  \gamma_{ \mathcal{M}'_{m'_1}, \mathcal{H}'_{h'_1}}[x_{m'_1},x_{m'_2},\dots,x_{m'_{w_2}}][x_{h'_1}^{u_{23}},x_{h'_2},\dots,x_{h'_{q_2}}]\\
    &+\, \sum_{\mathcal{L}, \mathcal{L}'}  \lambda_{\mathcal{L}, \mathcal{L}'}[x_{l_1}^{u_{12}},x_{l_2},\dots,x_{l_{v_1}}][x_{l'_1}^{u_{23}},x_{l'_2},\dots,x_{l'_{v_2}}] + \mu [x_{1}^{u_{13}},x_{2},\dots,x_{n}],
\end{align*}
where $\mathcal{J}_{j_1}=\{j_1,\ldots,j_s\},$ $\mathcal{K}_{k_1}=\{k_1,\ldots,k_t\},$ $\mathcal{H}_{h_1}=\{h_1,\ldots,h_{q_1}\},$ $\mathcal{M}_{m_1}=\{m_1,\ldots,m_{w_1}\},$ $\mathcal{H}'_{h'_1}=\{h'_1,\ldots,h'_{q_2}\},$ $\mathcal{M}'_{m'_1}=\{m'_1,\ldots,m'_{w_2}\},$ $\mathcal{L}=\{l_1,\ldots,l_{v_1}\}$ and $\mathcal{L}'=\{l'_1,\ldots,l'_{v_2}\}$
 are subsets of $\{1,\ldots, n\}$ satisfying $\mathcal{J}_{j_1}\cap \mathcal{K}_{k_1}=\mathcal{H}_{h_1}\cap \mathcal{M}_{m_1}=\mathcal{H}'_{h'_1}\cap \mathcal{M}'_{m'_1}=\mathcal{L}\cap \mathcal{L}'=\emptyset$. Moreover, the indices fulfill the conditions
$s\geq 2,$ $t, w_1,w_2\geq 0,$ $t,w_1,w_2\neq 1,$ $q_1,q_2,v_1,v_2\geq 1,$ $s+t=q_1+w_1=w_2+q_2=v_1+v_2=n,$ together with the ordering constraints $j_1>j_2<  \cdots < j_s,$ $k_1>k_2<  \cdots < k_t,$ $m_1> m_2< \cdots < m_{w_1},$ $m'_1> m'_2< \cdots < m'_{w_2},$
$l_1< \cdots < l_{v_1},$ $l'_1< \cdots < l'_{v_2},$ $h_2< \cdots <h_{q_1},$ $h'_2< \cdots <h'_{q_2}.$ In addition, if $w_1\geq 2$ then $h_1< h_2$, and  if $w_2\geq 2$ then $h'_1< h'_2.$

\textbf{Step 5.} We show that the $L$-polynomials appearing in the decomposition of $f$ above are
linearly independent modulo $\I^L(UT_3^{\e_1+\beta\e_2})$.
 We carry out essentially the same evaluations as in Theorem \ref{teo: Id^L UT_3 e1-e2}, with only a slight modification. We begin in the same manner by fixing a set $\mathcal{J}_{j_1}=\{j_1,\ldots,j_n\}$ such that $j_1>j_2<  \cdots < j_n.$ Evaluating at $ x_{j_1}= e_{13}$, $x_{j_2}=\cdots = x_{j_n}= e_{33},$ we obtain $\alpha_{ \mathcal{J}_{j_1}, \emptyset}  e_{13}=0$ since $j_1\neq 1$ and $e_{33}^{\nu_{13}}=0$. So, it follows that $ \alpha_{ \mathcal{J}_{j_1}, \emptyset}=0$. Next, evaluating $x_1=e_{13}$ and $x_2=\cdots=x_n=e_{33}$, we get $\mu  e_{13}=0$, and hence $\mu=0$. Then, proceeding exactly as in Theorem \ref{teo: Id^L UT_3 e1-e2}, we also obtain that $\alpha_{ \mathcal{J}_{j_1}, \mathcal{K}_{k_1}} = \beta_{ \mathcal{H}_{h_1}, \mathcal{M}_{m_1}}=\gamma_{ \mathcal{M}'_{m'_1}, \mathcal{H}'_{h'_1}}=\lambda_{ \mathcal{L}, \mathcal{L}'}  =0 $ for all $\mathcal{J}_{j_1}, \mathcal{K}_{k_1},\mathcal{H}_{h_1}, \mathcal{M}_{m_1},\mathcal{H}'_{h'_1}, \mathcal{M}'_{m'_1},\mathcal{L}, \mathcal{L}'$. 

Therefore, we conclude that all coefficients appearing in $f$ vanish. As a consequence, it follows that the $L$-polynomials  \eqref{No IdUT_3^e_1+be_2} are linearly independent modulo $\I^L(UT_3^{\e_1+\beta\e_2})$ and $\I^L(UT_3^{\e_1+\beta\e_2})= I$. Consequently, the $L$-polynomials \eqref{No IdUT_3^e_1+be_2} form a basis of $P_n^L$ modulo $P_n^L\cap \I^L(UT_3^{\e_1+\beta\e_2}).$

Finally, by counting these $L$-polynomials, we determine the $L$-codimensions of $UT_3^{\e_1+\beta\e_2}$, we get
\begin{align*}
     c_n^L(UT_3^{\e_1+\beta\e_2})=&c_n^L(UT_3^{\e_1 - \e_2}) + \sum_{r=0}^{n-1} \binom{n}{r}
    =(n^2-n)3^{n-2}+3n2^{n-1}+1,
\end{align*}
as claimed.
\end{proof}


Finally, we determine the $T_L$-ideal of the differential identities of $UT_3^{\e_1, \e_2}.$ To this end, we establish the following technical lemmas.

\begin{lemma} \label{lem: consequence e_1 e_2 1var}
Let $d_1,d_2\in L.$
Then $x_1^{d_1} x_2^{d_1 d_2}, \ x_1^{d_1} x_2 x_3^{d_1 d_2} \in \langle x^{d_1^2}-x^{d_1},\ x_1^{d_2}x_2^{d_1}\rangle_{T_L}.$
\end{lemma}
\begin{proof}
By Lemma~\ref{lem: consequences Ide_1 1 var}, we have  $x_1^{d_1}x_2^{d_1}\in \langle x^{d_1^2}-x^{d_1}\rangle_{T_L}.$ It then follows, by applying the Leibniz rule, the desired conclusion.
\end{proof}

\begin{lemma}\label{lem: UT_3^e1,e2 d1 d2 com}
    Let $d_1,d_2\in L.$
Then:
\begin{enumerate}
    \item $[x_1,x_2]x_3^{d_1}, \ [x_1,x_2]x_3^{d_1 d_2}\in \langle x^{d_1^2}-x^{d_1},\, x_1^{d_2}x_2^{d_1},\,  [x_1,x_2]^{d_1 d_2}-[x_1,x_2]^{d_1} - [x_1,x_2]^{d_2}+[x_1,x_2]\rangle_{T_L};$
    \vspace{1mm}
    \item $x_1^{d_2}[x_2,x_3], \  x_1^{d_1 d_2}[x_2,x_3]  \in \langle x^{d_2^2}-x^{d_2},\, x_1^{d_2}x_2^{d_1}, \,  [x_1,x_2]^{d_1 d_2}- [x_1,x_2]^{d_1} - [x_1,x_2]^{d_2}+ [x_1,x_2] \rangle_{T_L}$;
    \vspace{1mm}
    \item $x_1^{d_1}[x_2,x_3]x_4^{d_2}\in \langle x^{d_1^2}-x^{d_1},\, x^{d_2^2}-x^{d_2},\,  x_1^{d_2}x_2^{d_1} ,\,  [x_1,x_2]^{d_1 d_2}-[x_1,x_2]^{d_1} - [x_1,x_2]^{d_2}+ [x_1,x_2] \rangle_{T_L}$.
\end{enumerate}
\end{lemma}
\begin{proof}\black
Let $\mathcal{S}_1:=\{ x^{d_1^2}-x^{d_1},\ x_1^{d_2}x_2^{d_1},\   [x_1,x_2]^{d_1 d_2}-[x_1,x_2]^{d_1} - [x_1,x_2]^{d_2}+ [x_1,x_2] \}.$
    Since by Lemma~\ref{lem: consequences Ide_1 1 var} we have $x_1^{d_1}x_2^{d_1}\in \langle x^{d_1^2}-x^{d_1} \rangle_{T_L} $, it follows that
 $$
    [x_1,x_2]^{d_1 d_2}x_3^{d_1}- [x_1,x_2]^{d_1}x_3^{d_1} -[x_1,x_2]^{d_2}x_3^{d_1} + [x_1,x_2]x_3^{d_1}\equiv [x_1,x_2]x_3^{d_1}\pmod{\langle x^{d_1^2}-x^{d_1},\ x_1^{d_2} x_2^{d_1} \rangle_{T_L}},
     $$
and, as a consequence, $[x_1,x_2]x_3^{d_1}\in \langle \mathcal{S}_1 \rangle_{T_L}.$

Now, since $\big([x_1,x_2]x_3^{d_1}\big)^{d_2}=[x_1,x_2]^{d_2}x_3^{d_1} + [x_1,x_2]x_3^{d_1 d_2},$ it follows, by the above, that $[x_1,x_2]x_3^{d_1 d_2}\in \langle\mathcal{S}_1\rangle_{T_L}.$ This completes the proof of (1). Similarly, we can also prove (2) and (3).
\end{proof}

As a direct consequence of the above Lemma, we obtain the following result.
\begin{lemma}\label{lem: UT_3^e1,e2 comm ordin}
    Let $d_1,d_2\in L.$
Then:
\begin{enumerate}
    \item$[x_1,x_2] \big([x_3,x_4]^{d_2}-[x_3,x_4]\big)\in \langle x^{d_1^2}-x^{d_1},\, x_1^{d_2}x_2^{d_1}, \,  [x_1,x_2]^{d_1 d_2}- [x_1,x_2]^{d_1} - [x_1,x_2]^{d_2}+ [x_1,x_2]\rangle_{T_L};$
    \vspace{1mm}
    \item $\big([x_1,x_2]^{d_1}-[x_1,x_2]\big)[x_3,x_4] \in \langle x^{d_2^2}-x^{d_2},\, x_1^{d_2}x_2^{d_1}, \,   [x_1,x_2]^{d_1 d_2}- [x_1,x_2]^{d_1} - [x_1,x_2]^{d_2}+ [x_1,x_2] \rangle_{T_L}$.
\end{enumerate}
\end{lemma}

\begin{lemma}\label{lem: UT_3^e1,e2 d1 e d2 ordin}
       Let $d_1,d_2\in L.$
Then:
\begin{enumerate}
    \item $x_1^{d_1}\big([x_2,x_3]^{d_2}-[x_2,x_3]\big) \in \langle x^{d_1^2}-x^{d_1},\, x_1^{d_2}x_2^{d_1},\,  [x_1,x_2]^{d_1 d_2}- [x_1,x_2]^{d_1} - [x_1,x_2]^{d_2}+ [x_1,x_2]\rangle_{T_L};$
    \vspace{1mm}
    \item $\big([x_1,x_2]^{d_1}-[x_1,x_2]\big)x_3^{d_2} \in \langle x^{d_2^2}-x^{d_2},\, x_1^{d_2}x_2^{d_1}, \, [x_1,x_2]^{d_1 d_2}-[x_1,x_2]^{d_1} - [x_1,x_2]^{d_2}+[x_1,x_2] \rangle_{T_L}$.
\end{enumerate}
\end{lemma}
\begin{proof}
     Since, by Lemma~\ref{lem: consequences Ide_1 1 var}, we have $x_1^{d_1}x_2^{d_1}\in \langle x^{d_1^2}-x^{d_1} \rangle_{T_L} $ and, by Lemma \ref{lem: consequence e_1 e_2 1var}, also $x_1^{d_1} x_2^{d_1 d_2} \in \langle x^{d_1^2}-x^{d_1},\ x_1^{d_2}x_2^{d_1}\rangle_{T_L}$, we get
     \begin{align*}
         x_1^{d_1}[x_2,x_3]^{d_1 d_2}- x_1^{d_1}[x_2,x_3]^{d_1} - x_1^{d_1} [x_2,x_3]^{d_2}& + x_1^{d_1} [x_2,x_3]\\
         &\equiv x_1^{d_1}\big([x_2,x_3]^{d_2}-[x_2,x_3]\big) \pmod{\langle x^{d_1^2}-x^{d_1},\, x_1^{d_2}x_2^{d_1}\rangle_{T_L}},
     \end{align*}
     and statement (1) follows. The proof of (2) is analogous.
\end{proof}

\begin{lemma}\label{lem: UT_3^e1,e2 d1d2 ordin}
    Let $d_1,d_2\in L$. Then
     $$[x_1^{d_1 d_2},x_2]-  [x_2^{d_1 d_2},x_1]
    -[x_1^{d_1},x_2] +[x_2^{d_1},x_1]  - [x_1^{d_2},x_2]  + [x_2^{d_2},x_1]+ [x_1,x_2] + x_1^{d_1}x_2^{d_2} - x_2^{d_1}x_1^{d_2}$$
    is a consequence of the $L$-polynomials in
    $$
    \left\{ x_1^{d_2}x_2^{d_1}, \  [x_1,x_2]^{d_1 d_2}-[x_1,x_2]^{d_1} -[x_1,x_2]^{d_2} + [x_1,x_2]\right\}.
    $$
\end{lemma}
\begin{proof}
   Since
    $[x_1,x_2]^{d_1 d_2}\equiv  [x_1^{d_1 d_2},x_2]- [x_2^{d_1 d_2},x_1]+ x_1^{d_1}x_2^{d_2} - x_2^{d_1}x_1^{d_2} \pmod{\langle x_1^{d_2}x_2^{d_1} \rangle_{T_L}},$ the desired conclusion follows.
\end{proof}

\begin{theorem}\label{Teo: Id^L UT_3 epsilon1, epsilon2}
Let $L$ be a Lie algebra and  let $UT_3^{\e_1, \e_2}$  be the $L$-algebra $UT_3$ where $L$ acts via the two-dimensional abelian Lie subalgebra of $\D(UT_3)$ with ordered basis $\{\e_1=\ad_{(-e_{11})} , \e_2=\ad_{e_{33}} \}$. Then 
there exists a basis $\mathcal{B}_L=\{d_i \mid i \geq 1\}$ of $L$ over $F$ such that 
the $T_L$-ideal of differential identities of $UT_3^{\e_1, \e_2}$ is generated by the following $L$-polynomials
$$
x^{d_i}, \quad x^{d_1^2}-x^{d_1}, \quad x^{d_2^2}-x^{d_2}, \quad x^{d_1d_2}-x^{d_2d_1},  \quad x_1^{d_2}x_2^{d_1}, \quad   [x_1,x_2]^{d_1 d_2}-[x_1,x_2]^{d_1} -[x_1,x_2]^{d_2} + [x_1,x_2],$$
for all $i\geq 3.$ Moreover, $c_n^L(UT_3^{\e_1,\e_2})= (n^2-n)3^{n-2}+3n2^{n-1}+1.$
\end{theorem}
\begin{proof}
\textbf{Step 1} is as in Remark~\ref{rmk: general-strategy}, with $\theta(d_1)=\e_1$ and
$\theta(d_2)=\e_2$.

\textbf{Step 2} is immediate from \eqref{eq: e1} and \eqref{eq: e2}.

\textbf{Step 3.} Let $f\in P_n^L$. At most two left-normed commutators occur modulo $I$ (item (1)
of Lemma~\ref{lem: UT_3^e1,e2 d1 d2 com}, item (2) of Lemma~\ref{lem: UT_3^e1,e2 comm ordin} and
Lemma~\ref{lem: conseq 3 commutators 1}).

     Now, since $x^{d_i}, \; x^{d_1^2}-x^{d_1}, \; x^{d_2^2}-x^{d_2}\in I$, the exponents occurring modulo $I$ reduce to  $\{1_{U(L)}, d_1, d_2, d_1d_2\}.$ 
     Moreover, from  Lemma~\ref{lem: consequences Ide_1 1 var} it follows that $x_1^{d_1}x_2^{d_1},\, x_1^{d_2}x_2^{d_2},\,   x_1^{d_1 d_2} x_2^{d_1 d_2} \in I$,  and consequentially we also get $ x_1^{d_1}x_2 x_3^{d_1}, \, x_1^{d_2}x_2 x_3^{d_2}, \, x_1^{d_1 d_2} x_2x_3^{d_1 d_2}\in I$.
     In addition, the same lemma also yields $x_1^{d_1 d_2} x_2^{d_2}, \, x_1^{d_2} x_2^{d_1 d_2}\in I$ and $ x_1^{d_1 d_2} x_2x_3^{d_2}, \, x_1^{d_2} x_2x_3^{d_1 d_2}\in I.$ Also, by Lemma \ref{lem: consequence e_1 e_2 1var} we get $x_1^{d_1} x_2^{d_1 d_2}, \ x_1^{d_1} x_2 x_3^{d_1 d_2} \in I.$
     Additionally, since $x_1^{d_2} x_2^{d_1}\in I$, clearly we also have $x_1^{d_2} x_2 x_3^{d_1}\in I$ and also $x_1^{d_1d_2}x_2^{d_1}, \,x_1^{d_1d_2}x_2x_3^{d_1}\in I.$ Therefore, modulo $I$, an $L$-polynomial may contain either at most one variable with exponent $d_1$ or $d_2$ or $d_1d_2$, or two distinct variables with exponent the first one $d_1$ and the other one $d_2$.

Therefore, taking into account Lemmas \ref{lem: consequence e_1 e_2 1var}–\ref{lem: UT_3^e1,e2 d1d2 ordin}, together with Remark \ref{rmk: leibniz rule comm}, Lemmas \ref{lem: ordered comm}, \ref{lem: x_1^{u_1}x_2^{u_2}}, and \ref{lem: ordered comm x_1^u}, as well as the Jacobi identity, and arguing exactly as in Theorem \ref{teo: Id^L UT_3 e1+betae2}, we conclude that $f$  can be expressed, modulo $I$, as a linear combination of $L$-polynomials:
\begin{equation} \label{No Id UT_3e1e2}
\begin{split}
&x_{i_{1}}\dots x_{i_{r}}[x_{j_1},x_{j_2},\dots,x_{j_s}][x_{k_1},x_{k_2},\dots,x_{k_t}],\\
&x_{i_{1}}\dots x_{i_{r}}[x_{h_1}^{d_{1}},x_{h_2},\dots,x_{h_{q_1}}][x_{m_1},x_{m_2},\dots,x_{m_{w_1}}],\\
&x_{i_{1}}\dots x_{i_{r}}[x_{m'_1},x_{m'_2},\dots,x_{m'_{w_2}}][x_{h'_1}^{d_{2}},x_{h'_2},\dots,x_{h'_{q_2}}],\\
&x_{i_{1}}\dots x_{i_{r}}[x_{l_1}^{d_{1}},x_{l_2},\dots,x_{l_{v_1}}][x_{l'_1}^{d_{2}},x_{l'_2},\dots,x_{l'_{v_2}}],\\
&x_{i_{1}}\dots x_{i_{r}}[x_{p_1}^{d_1d_2},x_{p_2},\dots,x_{p_z}],
\end{split}
\end{equation}
where $r,s,t, w_1, w_2 \geq 0,$ with $s,t,w_1,w_2\neq 1,$ and $q_1,q_2,v_1,v_2,z\geq 1$ such that $i_1< \cdots < i_r,$ $j_1>j_2<  \cdots < j_s,$ $k_1>k_2<  \cdots < k_t,$ 
$m_1> m_2< \cdots < m_{w_1},$ $h_2< \cdots <h_{q_1},$  $m'_1> m'_2< \cdots < m'_{w_2},$  $h'_2< \cdots <h'_{q_2},$
$l_1< \cdots < l_{v_1},$ $l'_1< \cdots < l'_{v_2},$ $p_1< \cdots < p_z.$ Additionally, if $w_1\geq 2,$ then $h_1< h_2$, and  if $w_2\geq 2,$ then $h'_1< h'_2.$

\textbf{Step 4} is as in Remark~\ref{rmk: general-strategy}: since $UT_3$ has a unit element, we may,
without loss of generality, disregard the tail $x_{i_{1}}\dots x_{i_{r}}$. Hence, if $f\in \I^L(UT_3^{\e_1,\e_2})$, we may write

 \begin{align*}
    f = &\sum_{\mathcal{J}_{j_1}, \mathcal{K}_{k_1}} \alpha_{ \mathcal{J}_{j_1}, \mathcal{K}_{k_1}} [x_{j_1},x_{j_2},\dots,x_{j_s}][x_{k_1},x_{k_2},\dots,x_{k_t}]\\
    &+ \sum_{ \mathcal{H}_{h_1}, \mathcal{M}_{m_1}}  \beta_{\mathcal{H}_{h_1}, \mathcal{M}_{m_1}} [x_{h_1}^{d_1},x_{h_2},\dots,x_{h_{q_1}}][x_{m_1},x_{m_2},\dots,x_{m_{w_1}}]\\
    &+ \sum_{ \mathcal{M}'_{m'_1}, \mathcal{H}'_{h'_1}}  \gamma_{ \mathcal{M}'_{m'_1}, \mathcal{H}'_{h'_1}}[x_{m'_1},x_{m'_2},\dots,x_{m'_{w_2}}][x_{h'_1}^{d_2},x_{h'_2},\dots,x_{h'_{q_2}}]\\
    &+ \sum_{\mathcal{L}, \mathcal{L}'}  \lambda_{\mathcal{L}, \mathcal{L}'}[x_{l_1}^{d_1},x_{l_2},\dots,x_{l_{v_1}}][x_{l'_1}^{d_2},x_{l'_2},\dots,x_{l'_{v_2}}]+  \mu[x_{1}^{d_1 d_2},x_{2},\dots,x_{n}],
\end{align*}
where $\mathcal{J}_{j_1}=\{j_1,\ldots,j_s\},$ $\mathcal{K}_{k_1}=\{k_1,\ldots,k_t\},$ $\mathcal{H}_{h_1}=\{h_1,\ldots,h_{q_1}\},$ $\mathcal{M}_{m_1}=\{m_1,\ldots,m_{w_1}\},$ $\mathcal{H}'_{h'_1}=\{h'_1,\ldots,h'_{q_2}\},$ $\mathcal{M}'_{m'_1}=\{m'_1,\ldots,m'_{w_2}\},$ $\mathcal{L}=\{l_1,\ldots,l_{v_1}\}$ and $\mathcal{L}'=\{l'_1,\ldots,l'_{v_2}\}$
 are subsets of indices in $\{1,\ldots, n\}$ such that $\mathcal{J}_{j_1}\cap \mathcal{K}_{k_1}=\mathcal{H}_{h_1}\cap \mathcal{M}_{m_1}=\mathcal{H}'_{h'_1}\cap \mathcal{M}'_{m'_1}=\mathcal{L}\cap \mathcal{L}'=\emptyset$, with
$s\geq 2,$ $t, w_1,w_2\geq 0,$ $t,w_1,w_2\neq 1,$ $q_1,q_2,v_1,v_2\geq 1,$ $s+t=q_1+w_1=w_2+q_2=v_1+v_2=n,$ and subjected  to the conditions $j_1>j_2<  \cdots < j_s,$ $k_1>k_2<  \cdots < k_t,$ $h_2< \cdots <h_{q_1},$ $m_1> m_2< \cdots < m_{w_1},$ $m'_1> m'_2< \cdots < m'_{w_2},$ $h'_2< \cdots <h'_{q_2},$ $l_1< \cdots < l_{v_1},$ $l'_1< \cdots < l'_{v_2}$. Additionally, if $w_1\geq 2,$ then $h_1< h_2$, and  if $w_2\geq 2,$ then $h'_1< h'_2.$

\textbf{Step 5.} We show that the $L$-polynomials appearing in the decomposition of $f$ above are
linearly independent modulo $\I^L(UT_3^{\e_1,\e_2})$.

Let $\mathcal{J}_{j_1}=\{j_1,\ldots,j_n\}$ be a fixed set such that $j_1>j_2<  \cdots < j_n.$ If we evaluate $ x_{j_1}= e_{12}$ and  $x_{j_2}=\cdots = x_{j_n}= e_{22},$ then we get $(\alpha_{ \mathcal{J}_{j_1}, \emptyset} +\beta_{ \mathcal{J}_{j_1}, \emptyset}) e_{12}=0,$ thus 
\begin{equation}\label{evaluation 1}
    \alpha_{ \mathcal{J}_{j_1}, \emptyset} +\beta_{ \mathcal{J}_{j_1}, \emptyset}=0.
\end{equation}
If we make the evaluation $ x_{j_1}= e_{23}$ and  $x_{j_2}=\cdots = x_{j_n}= e_{33},$ then we obtain $(\alpha_{ \mathcal{J}_{j_1}, \emptyset} +\gamma_{ \emptyset, \mathcal{J}_{j_1}}) e_{23}=0,$ thus 
\begin{equation}\label{evaluation 2}
    \alpha_{ \mathcal{J}_{j_1}, \emptyset} +\gamma_{ \emptyset, \mathcal{J}_{j_1}}=0.
\end{equation}
 Moreover, from the evaluation $ x_{j_1}= e_{13}$ and  $x_{j_2}=\cdots = x_{j_n}= e_{33},$ it follows that $(\alpha_{ \mathcal{J}_{j_1}, \emptyset} +\beta_{ \mathcal{J}_{j_1}, \emptyset}+ \gamma_{ \emptyset, \mathcal{J}_{j_1}}) e_{13}=0,$ thus 
 \begin{equation}\label{evaluation 3}
    \alpha_{ \mathcal{J}_{j_1}, \emptyset} +\beta_{ \mathcal{J}_{j_1}, \emptyset}+ \gamma_{ \emptyset, \mathcal{J}_{j_1}}=0.
 \end{equation}
Therefore, putting  together \eqref{evaluation 1}, \eqref{evaluation 2} and \eqref{evaluation 3} we get 
$\alpha_{ \mathcal{J}_{j_1}, \emptyset} =\beta_{ \mathcal{J}_{j_1}, \emptyset}= \gamma_{ \mathcal{J}_{j_1}, \emptyset}=0.$

Now, let us make the evaluation $ x_{1}= e_{12}$ and  $x_{2}=\cdots = x_{n}= e_{22}.$ In this case we get $\beta_{ \mathcal{H}_{1}, \emptyset}  e_{12}=0,$ where $\mathcal{H}_{1}=\{1,\ldots,n\},$ so  $\beta_{ \mathcal{H}_{1}, \emptyset}=0.$ Also, by making the evaluation  $ x_{1}= e_{23}$ and  $x_{2}=\cdots = x_{n}= e_{33},$ we get $\gamma_{  \emptyset, \mathcal{H}'_{1}}  e_{23}=0,$ where $\mathcal{H}'_{1}=\{1,\ldots,n\},$ then  $\gamma_{  \emptyset, \mathcal{H}'_{1}}=0.$ Furthermore, from the evaluation $ x_{1}= e_{13}$ and  $x_{2}=\cdots = x_{n}= e_{33},$ we get $\mu  e_{13}=0,$ thus  $\mu=0.$ 

Next, let $\mathcal{J}_{j_1}=\{j_1,\ldots,j_s\}$ and $\mathcal{K}_{k_1}=\{k_1,\ldots,k_t\}$ be two fixed disjoint subsets of $\{1,\ldots,n\}$ such that $s,t\geq 2,$ $s+t=n$ and $j_1>j_2<  \cdots < j_s,$ $k_1>k_2<  \cdots < k_t.$ If we evaluate
$ x_{j_1}= e_{12},$  $x_{j_2}=\cdots = x_{j_s}= e_{11},$ $ x_{k_1}= e_{23}$ and  $x_{k_2}=\cdots = x_{k_t}= e_{33},$ then we get $\alpha_{ \mathcal{J}_{j_1}, \mathcal{K}_{k_1}} e_{13}=0,$ thus $\alpha_{ \mathcal{J}_{j_1}, \mathcal{K}_{k_1}}=0.$
Here the polynomials  $[x_{h_1}^{d_1},x_{h_2},\dots,x_{h_{s}}][x_{k_1},x_{k_2},\dots,x_{k_{t}}]$  evaluate to zero since $j_1>j_2<  \cdots < j_s$ whereas $h_1<h_2<\cdots <h_{s}$ and $e_{11}^{\e_1}=0.$ Similarly, the polynomials  $[x_{j_1},x_{j_2},\dots,x_{j_s}][x_{h'_1}^{d_2},x_{h'_2},\dots,x_{h'_{t}}]$ evaluate to zero since  $k_1>k_2<  \cdots < k_t$ whereas $h'_1<h'_2<\cdots <h'_t$ and $e_{33}^{\e_2}=0.$
Also, with the same reasoning all the polynomials of the type $[x_{l_1}^{d_1},x_{l_2},\dots,x_{l_{v_1}}][x_{l'_1}^{d_2},x_{l'_2},\dots,x_{l'_{v_2}}]$ evaluate to zero.

Let $\mathcal{H}_{h_1}=\{h_1,\ldots,h_{q_1}\}$ and  $\mathcal{M}_{m_1}=\{m_1,\ldots,m_{w_1}\}$ be two fixed disjoint subsets of $\{1,\ldots,n\}$ with $q_1\geq 1,$ $w_1\geq 2,$ $q_1+w_1=n$ and subject to the conditions $h_1<h_2<  \cdots < h_{q_1},$ $m_1>m_2<  \cdots < m_{w_1}.$ If we make the evaluation $x_{h_1}=e_{12},$ $x_{h_2}=\cdots = x_{h_{q_1}}= e_{11},$ $x_{m_1}=e_{23}$ and $x_{m_2}= \cdots = x_{m_{w_1}}= e_{33},$ we get $\beta_{\mathcal{H}_{h_1}, \mathcal{M}_{m_1}}e_{13}=0,$ that is $\beta_{\mathcal{H}_{h_1}, \mathcal{M}_{m_1}}=0.$ Notice that as above with this evaluation all the polynomials of the type $[x_{m'_1},x_{m'_2},\dots,x_{m'_{w_2}}][x_{h'_1}^{d_2},x_{h'_2},\dots,x_{h'_{q_2}}]$ and of the type $[x_{l_1}^{d_1},x_{l_2},\dots,x_{l_{v_1}}][x_{l'_1}^{d_2},x_{l'_2},\dots,x_{l'_{v_2}}]$ evaluate to zero.
Similarly, if we evaluate $x_{h_1}=e_{23},$ $x_{h_2}=\cdots = x_{h_{q_1}}= e_{33},$ $x_{m_1}=e_{12}$ and $x_{m_2}= \cdots = x_{m_{w_1}}= e_{11},$ we get $\gamma_{\mathcal{M}_{m_1}, \mathcal{H}_{h_1}}e_{13}=0,$ and so $\gamma_{\mathcal{M}_{m_1}, \mathcal{H}_{h_1}}=0.$

Finally, fixed $\mathcal{L}=\{l_1,\ldots,l_{v_1}\}$ and $\mathcal{L}'=\{l'_1,\ldots,l'_{v_2}\}$ disjoint subset of $\{1,\ldots,n\}$ such that $v_1,v_2\geq 1,$ $v_1+v_2=n$ and $l_1< \cdots < l_{v_1},$ $l'_1<  \cdots < l'_{v_2},$ then from the evaluation  $x_{l_1}=e_{12},$ $x_{l_2}=\cdots = x_{l_{v_1}}= e_{11},$ $x_{l'_1}=e_{23}$ and $x_{l'_2}= \cdots = x_{l'_{v_2}}= e_{33},$ it follows that $\lambda_{\mathcal{L}, \mathcal{L}'}e_{13}=0$ and consequently $\lambda_{\mathcal{L}, \mathcal{L}'}=0.$ 

Therefore, all the scalars appearing in $f$ are zero, and as a consequence, the $L$-polynomials  \eqref{No Id UT_3e1e2} are linearly independent modulo $\I^L(UT_3^{\e_1,\e_2}).$ This proves that $\I^L(UT_3^{\e_1,\e_2})= I$ and the polynomials \eqref{No Id UT_3e1e2} are a basis of $P_n^L$ modulo $P_n^L\cap \I^L(UT_3^{\e_1,\e_2}).$ Hence, as in Theorem \ref{teo: Id^L UT_3 e1+betae2}, by counting we get
   $ c_n^L(UT_3^{\e_1,\e_2})=(n^2-n)3^{n-2}+3n2^{n-1}+1,$ as claimed.
\end{proof}

\end{document}